\documentclass[11pt]{amsart}
\usepackage{amssymb,latexsym,amsmath,graphicx,graphics,epic,eepic}

\theoremstyle{plain}
\newtheorem{theorem}{Theorem}
\numberwithin{theorem}{section}
\newtheorem{lemma}[theorem]{Lemma}
\newtheorem{proposition}[theorem]{Proposition}

\theoremstyle{definition}
\newtheorem{definition}[theorem]{Definition}

\newcommand{\D}{{\mathbb D}}
\newcommand{\C}{{\mathbb C}}
\newcommand{\R}{{\mathbb R}}
\newcommand{\Z}{{\mathbb Z}}
\newcommand{\Q}{{\mathbb Q}}
\renewcommand{\P}{{\mathbb P}}
\newcommand{\s}{{\mathbb S}}

              \newcommand{\J}{{\mathcal J}}
              \newcommand{\M}{{\mathcal M}}

              \newcommand{\E}{{\mathcal E}}

\begin{document}
\title{Finite group actions on symplectic Calabi-Yau $4$-manifolds with $b_1>0$}
\author{Weimin Chen}
\subjclass[2010]{Primary 57R55; Secondary 57S17, 57R17}
\keywords{Four-manifolds, smooth structures, finite group actions, fixed-point set structures, 
orbifolds, symplectic resolution, symplectic Calabi-Yau, configurations of symplectic surfaces, 
rational $4$-manifolds, branched coverings, pseudo-holomorphic curves.}
\thanks{}
\date{\today}
\maketitle

\begin{abstract}
This is the first of a series of papers devoted to the topology of symplectic Calabi-Yau $4$-manifolds endowed with certain symplectic finite group actions. We completely determine the fixed-point set structure of a finite cyclic action on a symplectic Calabi-Yau $4$-manifold with $b_1>0$. As an outcome of this fixed-point set analysis, the $4$-manifold is shown to be a $T^2$-bundle over $T^2$ in some circumstances,
e.g., in the case where the group action is an involution which fixes a $2$-dimensional surface in the $4$-manifold. Our project on symplectic Calabi-Yau $4$-manifolds is based on an analysis of the existence and classification of disjoint embeddings of certain configurations of symplectic surfaces in a rational $4$-manifold. This paper lays the ground work for such an analysis at the homological level. Some other result which is of independent 
interest, concerning the maximal number of disjointly embedded symplectic $(-2)$-spheres in a rational $4$-manifold, is also obtained. 
\end{abstract}

\section{Introduction and the main results}

In this paper, we study symplectic finite group actions on symplectic Calabi-Yau $4$-manifolds with
$b_1>0$. (Recall that a symplectic $4$-manifold $M$ is called {\it Calabi-Yau} if $K_M$ is trivial.)
Our starting point is the recent construction in \cite{C3} (see also \cite{MR}), where to each symplectic $4$-manifold $M$
equipped with a finite symplectic $G$-action, we associate a symplectic $4$-manifold, denoted by $M_G$, and an embedding 
$D\rightarrow M_G$ of a disjoint union of configurations of symplectic surfaces. Roughly speaking, the $4$-manifold $M_G$ 
is constructed by first de-singularizing the
symplectic structure of the quotient orbifold $M/G$ along the $2$-dimensional singular strata, making
the underlying space $|M/G|$ into a symplectic $4$-orbifold with only isolated singularities. Then $M_G$ is taken to be the minimal symplectic resolution of the symplectic $4$-orbifold $|M/G|$, and $D$ is simply the pre-image of the singular set of the original 
orbifold $M/G$ in $M_G$. 
See \cite{C3} for more details. The idea is to recover the $G$-action on $M$, in particular the $4$-manifold $M$, by analyzing the
embedding $D\rightarrow M_G$. 
With this understood, it was shown (cf. \cite{C3}, Theorem 1.9) that if $M$ is Calabi-Yau, then $M_G$ is either of torsion canonical class, or is a rational $4$-manifold, or an irrational ruled $4$-manifold over $T^2$. Moreover, $M_G$ is of torsion canonical class if and only if the quotient orbifold $M/G$ has at most isolated Du Val singularities (cf. \cite{C3}, Lemma 4.1).  
Our basic observation is that, when $M_G$ is rational or ruled, it is possible to effectively recover the original $4$-manifold $M$ 
by analyzing the embedding $D\rightarrow M_G$. Moreover, as it turns out, one can also derive new constraints on the fixed-point set structure of the $G$-action from non-existence results for the embedding $D\rightarrow M_G$.

As an initial step toward understanding the topology of symplectic Calabi-Yau $4$-manifolds endowed with a 
symplectic finite group action, we consider first the case where the $4$-manifold $M$ has $b_1>0$, and determine 
the fixed-point set structure of a finite cyclic action on $M$. As a result of our analysis, we obtain the following

\begin{theorem}
Suppose $M$ is a symplectic Calabi-Yau $4$-manifold with $b_1>0$ which is endowed with
a finite symplectic $G$-action. If the resolution $M_G$ is irrational ruled, or $M_G$ is rational and
$G=\Z_2$, then $M$ must be diffeomorphic to a $T^2$-bundle over $T^2$ with homologically essential fibers.
\end{theorem}

We remark that in Theorem 1.1, $M$ is in fact diffeomorphic to a hyperelliptic surface in the case of
$G=\Z_2$ and $M_G$ is rational. On the other hand, we note that in the case of $G=\Z_2$,
$M_G$ is rational or ruled if and only if the fixed-point set $M^G$ contains a $2$-dimensional component. We state this special case in the following

\vspace{2mm}

\noindent{\bf Corollary:}\hspace{2mm} 
{\it Let $M$ be a symplectic Calabi-Yau $4$-manifold with $b_1>0$,
which is equipped with a symplectic involution whose fixed-point set contains a $2$-dimensional component. Then $M$ must be diffeomorphic to a $T^2$-bundle over $T^2$ with homologically essential fibers.}

\vspace{2mm}

R. Inanc Baykur \cite{Bay} informed us that he has examples of symplectic Calabi-Yau $4$-manifolds with $b_1=2$ and $4$, which are
constructed using symplectic Lefschetz pencils, and which come with a natural symplectic involution whose fixed-point set 
contains a $2$-dimensional component. Our theorem shows that these symplectic Calabi-Yau $4$-manifolds all have the standard
smooth structure. 

To put Theorem 1.1 in a perspective, recall that symplectic $4$-manifolds can be classified into
four classes according to their symplectic Kodaira dimension $\kappa^s$, which is a smooth invariant and takes values in
$\{-\infty,0,1,2\}$. (The classification is analogous to the classification in complex surface theory, but the relevant definitions are given in completely different ways. For K\"{a}hler surfaces, the two classifications coincide. See \cite{Li}.) Furthermore, as a culmination of the seminal works of 
Gromov, McDuff, and Taubes \cite{G0,McD,T}, the case of $\kappa^s=-\infty$ is completely determined: these symplectic $4$-manifolds are precisely the rational or ruled surfaces. 

Much effort has also been devoted to the next case, i.e., $\kappa^s=0$. First, based on 
Taubes' theory \cite{T}, T.-J. Li (cf. \cite{Li}) showed that a minimal symplectic $4$-manifold $M$ 
with $\kappa^s=0$ is either Calabi-Yau (i.e., $K_M$ is trivial), or a double cover of $M$ is 
Calabi-Yau. Note that a symplectic Calabi-Yau $4$-manifold is spin. Using the Bauer-Furuta theory of spin $4$-manifolds, together with Taubes' theorem \cite{T} and the classical Rochlin 
Theorem, the following homological constraints were obtained, see \cite{ Bauer, Li, Li1, MSz}:
\begin{itemize}
\item A symplectic Calabi-Yau $4$-manifold $M$ either has the integral homology and intersection form of a $K3$ surface, or has the rational homology and intersection form of a $T^2$-bundle 
over $T^2$; in particular, $0\leq b_1(M)\leq 4$, and if $b_1(M)>0$, $M$ has zero Euler
number and signature. (If $M$ is non-Calabi-Yau but a double cover of $M$ is Calabi-Yau, 
then $M$ is an integral homology Enriques surface.)
\item In addition, for the case of $b_1(M)=4$, the cohomology ring $H^\ast(M;\R)$ is isomorphic to
$H^\ast(T^4;\R)$ (cf. \cite{Sr}).
\end{itemize}

The above homological constraints are in sharp contrast to the flexibility known in higher dimensional symplectic Calabi-Yau manifolds, see e.g., \cite{FP}. Using a covering trick, one can also obtain 
interesting constraints on the fundamental group (as well as homotopy type in the case of $b_1>0$) 
of a symplectic Calabi-Yau $4$-manifold (cf. \cite{FV}), e.g., in the case of $b_1=0$, the fundamental group has no
subgroup of finite index.

As for examples, besides $K3$ surfaces, all orientable $T^2$-bundles over $T^2$ are symplectic 
Calabi-Yau $4$-manifolds (cf. \cite{G,Li}). (A topological classification of $T^2$-bundles over $T^2$ is
given in \cite{SF}.) We remark that not all $T^2$-bundles over $T^2$ admit a complex structure, and 
not all $T^2$-bundles over $T^2$ have homologically essential fibers (cf. \cite{G}). If a complex surface
is a symplectic Calabi-Yau $4$-manifold, then it is either a $K3$ surface, a complex torus, a primary Kodaira surface, or a hyperelliptic surface. With this understood, the following has been an open question (cf. \cite{Don, Li}):

\vspace{1.5mm}

{\it Does there exist a symplectic Calabi-Yau $4$-manifold other than the known examples, i.e., 
a $T^2$-bundle over $T^2$ or a $K3$ surface?}

\vspace{1.5mm}

We remark that the basic smooth invariants in $4$-manifold theory (e.g., the Seiberg-Witten invariants) are ineffective in distinguishing homeomorphic symplectic Calabi-Yau $4$-manifolds. As a result, 
one hopes to
construct new examples which have different topological invariants such as the fundamental group.
On the other hand, concerning characterizing the diffeomorphism types of symplectic Calabi-Yau $4$-manifolds, Theorem 1.1 is the first result 
of such kind (under a finite symmetry condition). Finally, for connections of this question with hypersymplectic structures and Donaldson's conjecture, we refer the readers to the recent 
article \cite{FY}. 

\vspace{1.5mm}

With the preceding understood, the idea of our project is to specialize in symplectic Calabi-Yau $4$-manifolds 
$M$ which admits a $G$-action such that $M_G$ is rational or ruled, and through $D\rightarrow M_G$, to gain insight 
about the topology of $M$. Note that with Theorem 1.1, the case where $M_G$ is irrational ruled is settled.

Now we state the results on the fixed-point set structure of a finite cyclic action on symplectic Calabi-Yau $4$-manifolds with $b_1>0$. 
We shall separate the prime order and non-prime order cases.

\begin{theorem}
Let $G$ be a cyclic group of prime order, and let $M$ be a symplectic Calabi-Yau $4$-manifold with $b_1>0$, equipped with a non-free symplectic $G$-action. Then the fixed-point set structure of the $G$-action and the symplectic resolution $M_G$ must belong to one of the following cases:
\begin{itemize}
\item [{(1)}] Suppose $M_G$ has torsion canonical class. Then either $G=\Z_2$ or $G=\Z_3$.
In the former case, $G$ either has $8$ isolated fixed points, with $b_1(M)<4$ and $M_G$ being
an integral homology Enriques surface, or has $16$ isolated fixed points, with $b_1(M)=4$ and $M_G$ being an integral homology $K3$ surface. In the latter case where $G=\Z_3$, the fixed point set consists of $9$ isolated points of type $(1,2)$, 
with $b_1(M)=4$ and $M_G$ being an integral homology $K3$ surface. 
\item [{(2)}] Suppose $M_G$ is irrational ruled. Then $G=\Z_2$ or $\Z_3$, the fixed point
set consists of only tori with self-intersection zero, and $M_G$ is a $\s^2$-bundle over $T^2$. 
\item [{(3)}] Suppose $M_G$ is rational. Then $G=\Z_2$, $\Z_3$ or $\Z_5$. The fixed-point set structure and $M_G$ are listed below:
\begin{itemize}
\item [{(i)}] If $G=\Z_2$, the fixed point set consists of one or two torus 
of self-intersection zero and $8$ isolated points, and $M_G=\C\P^2\# 9\overline{\C\P^2}$, $b_1(M)=2$.
\item [{(ii)}] If $G=\Z_3$, there are three possibilities, where $b_1(M)=2$ in $\left(a\right)$,
$\left(b\right)$, and $b_1(M)=4$ in $\left(c\right)$:  
\begin{itemize}
\item [{(a)}] the fixed point set consists of $6$ isolated points, where exactly $3$ of the fixed points 
are of type $(1,1)$, and $M_G=\C\P^2\# 10\overline{\C\P^2}$; 
\item [{(b)}] the fixed point set consists of one torus with self-intersection zero and $6$ isolated points, where exactly $3$ of the fixed points are of type $(1,1)$, and $M_G=\C\P^2\# 10\overline{\C\P^2}$;
\item [{(c)}] the fixed point set consists of $9$ isolated points of type $(1,1)$, and $M_G=\C\P^2\# 12\overline{\C\P^2}$.
\end{itemize}
\item [{(iii)}] If $G=\Z_5$, the fixed point set consists of $5$ isolated points of type $(1,2)$, and
$M_G=\C\P^2\# 11\overline{\C\P^2}$, $b_1(M)=4$.
\end{itemize}
\end{itemize}
\end{theorem}

\begin{theorem}
Let $G$ be a cyclic group of non-prime order, and let $M$ be a symplectic Calabi-Yau $4$-manifold with $b_1>0$, equipped with a symplectic $G$-action such that no subgroups of $G$ act freely on $M$. Suppose $M_G$ is rational or ruled, but for any prime order subgroup $H$, $M_H$ has torsion canonical class. Then $G=\Z_4$ or $\Z_8$. Moreover,
\begin{itemize}
\item [{(i)}] If $G=\Z_4$, there are two possibilities:
\begin{itemize}
\item [{(a)}] the $G$-action has $4$ isolated fixed points, where exactly $2$ of the fixed points are of
type $(1,1)$, and $4$ isolated points of isotropy of order $2$, with 
$M_G=\C\P^2\# 11\overline{\C\P^2}$; in this case, $b_1(M)=2$,
\item [{(b)}] the $G$-action has $4$ isolated fixed points, all of type $(1,1)$, and $12$ isolated points of isotropy of order $2$, with $M_G=\C\P^2\# 13\overline{\C\P^2}$; in this case, $b_1(M)=4$.
\end{itemize}
\item [{(ii)}] If $G=\Z_8$, there are two possibilities, where in both cases, $b_1(M)=4$:
\begin{itemize}
\item [{(a)}] the $G$-action has $2$ isolated fixed points, all of type $(1,3)$, and $2$ isolated points of isotropy of order $4$ of type $(1,3)$, and $12$ isolated points of isotropy of order $2$, with $M_G=\C\P^2\# 11\overline{\C\P^2}$; 
\item [{(b)}] the $G$-action has $2$ isolated fixed points, all of type $(1,5)$, and $2$ isolated points of isotropy of order $4$ of type $(1,1)$, and $12$ isolated points of isotropy of order $2$, with $M_G=\C\P^2\# 11\overline{\C\P^2}$.
\end{itemize}
\end{itemize}
\end{theorem}

\noindent{\bf Remarks:} (1) Let $n$ be the order of $G$ (prime or non-prime). An isolated fixed point $q$ is said to be of type $(1,b)$ (where $0<b<n$) if there is a generator $g$ of $G$ such that
the induced action of $g$ on $T_q M$ has eigenvalues $\exp(2\pi i/n)$ and 
$\exp(2\pi i b/n)$ (with respect to a complex structure on $T_q M$ compatible with the symplectic structure on $M$). More generally, an isolated point $q$ is of isotropy of order $m$ of
type $(1,b)$ (where $0<b<m$) if there is a generator $g$ of the isotropy subgroup $G_q$ at $q$, such that $m=|G_q|$ and the induced action of $g$ on $T_q M$ has eigenvalues $\exp(2\pi i/m)$ and 
$\exp(2\pi i b/m)$. Note that if $G=\Z_2$ or the isotropy order $m=2$, $q$ is always of type $(1,1)$. 

(2) In light of Theorem 1.1, it remains to determine $M$ when $M_G$ is
rational and $G\neq \Z_2$. Examining the cases in Theorems 1.2 and 1.3 where $M_G$ is rational,
we see that either $b_1(M)=2$ or $b_1(M)=4$. In particular, if $M$ admits a complex structure, then
$M$ must be either a hyperelliptic surface or a complex torus. 

(3) We point out that for all the cases where $M_G$ is rational, the fixed-point set structures can be realized by holomorphic actions (either on a hyperelliptic surface or a complex torus). Explicit examples realizing the fixed-point set structures listed in Theorem 1.2(3)(iii) and Theorem 1.3(ii) 
(where $G=\Z_5$ or $\Z_8$) can be found in Fujiki \cite{Fu}, Table 6 (examples for
the remaining cases can be easily constructed by hand).

\vspace{1.5mm}

For a large part, the proofs of Theorems 1.2 and 1.3 employ the standard techniques in group actions, i.e., the Lefschetz fixed point theorem and the $G$-signature theorem, coupled with the standard results in symplectic topology of rational and ruled surfaces and the topological constraints of minimal
symplectic $4$-manifolds with $\kappa^s=0$ through the use of $M_G$. Some of the cases also
require the use of $G$-index theorem for Dirac operators and Seiberg-Witten theory. These 
traditional methods are quite efficient in determining the fixed-point set structure for the isolated 
fixed points, however, for the $2$-dimensional fixed components, these methods have their natural limitations. The reason is that the $2$-dimensional fixed components (particularly the tori of
self-intersection zero) often do not make any contribution in the various $G$-index theorem calculations, hence cannot be detected by these methods. (See \cite{CK2}, Section 3, for a summary of these traditional methods.)

With this understood, in order to obtain further constraints on the $2$-dimensional fixed
components (as well as for a proof of Theorem 1.1), we shall analyze the embedding of $D$ in $M_G$. It turns out that the main difficulty occurs when $M_G$ is rational.

To explain this aspect of the story, we let $(X,\omega)$ be a symplectic rational $4$-manifold, where $X=\C\P^2\# N\overline{\C\P^2}$. We fix a reduced basis $H,E_1,E_2,\cdots,E_N$ of $(X,\omega)$ (a more detailed discussion on reduced bases will be given in Section 3). Then for any symplectic surface in $X$, its homology class $A$ can be expressed in terms of the reduced basis $H,E_1,E_2,\cdots,E_N$:
$$
A=aH-\sum_{i=1}^N b_i E_i, \mbox{ where } a\in \Z, b_i\in\Z.
$$
The numbers $a,b_i$ are called the $a$-coefficient and $b_i$-coefficients of $A$. By the adjunction formula, the numbers 
$a$ and $b_i$ are bound by a set of equations involving the self-intersection number $A^2$ and the genus of the surface. 
It follows easily from these equations that for each fixed
value of the $a$-coefficient, there are only finitely many possible values for the $b_i$-coefficients.
However, for each given symplectic surface, there is no a priori upper bound for the $a$-coefficient of its class $A$, although one can show that there is a lower bound for the $a$-coefficient (cf. Lemmas 3.3 and 3.4). 

Now suppose $D$ is a disjoint union of configurations of symplectic surfaces embedded in $X$,
where its components are denoted by $F_k$. The first step in approaching the problem 
of existence and classification of $D\rightarrow X$ is to look at the classes of the components $F_k$ in a given reduced 
basis. This process often involves a case-by-case examination, hence it is important that for each component $F_k$, there are only finitely many possible homological expressions. Such a finiteness can be achieved by bounding the values of the $a$-coefficient of each $F_k$, as the self-intersection number $F_k^2$ and the genus of $F_k$ are all pre-determined by $D\rightarrow X$.

In the present situation, $c_1(K_X)$ is supported in $D$. More precisely, 
$$
c_1(K_X)=\sum_k c_k F_k, \mbox{ where } c_k\in\Q \mbox{ and } c_k\leq 0.
$$
As $c_1(K_X)=-3H+\sum_{i=1}^N E_i$, the $a$-coefficient of $c_1(K_X)$ equals $-3$. It follows
easily that for those components $F_k$ with $c_k\neq 0$, the $a$-coefficient is bounded from above. However, if $F_k$ is a $(-2)$-sphere, which is either disjoint from the other components, or appears in a configuration of only $(-2)$-spheres, then $c_k=0$ and there is no bearing on the $a$-coefficient of $F_k$ from $c_1(K_X)$.

It turns out that we can remedy this issue by imposing an auxiliary area condition. More
concretely, let $A$ be the class of a symplectic $(-\alpha)$-sphere where $\alpha=2$ or $3$.
If the area condition $\omega(A)<-c_1(K_X)\cdot [\omega]$ is satisfied, then $A$ must take the following expression 
in a given reduced basis: 
$$
A=aH-(a-1)E_{j_1}-E_{j_2}-\cdots-E_{j_{2a+\alpha}}.
$$
In particular, the $a$-coefficient of $A$ has an upper bound in terms of $N$:
$$
a\leq \frac{1}{2}(N-\alpha)
$$
(See Lemma 3.6.) On the other hand, for the problem of existence and topological classification of embeddings of $D$ in $X$, one can always freely impose such an area condition by working with a different symplectic structure (cf. Lemma 4.1). Thus in principle, at least for the problem we have at hand, we have developed the necessary tools in this paper to classify the possible embeddings 
$D\rightarrow M_G$ at the homological level. In particular, by choosing an appropriate symplectic structure $\omega$ on $M_G$, there are only finitely many possible homological expressions for the components of $D$ with respect to any given reduced basis of $(M_G,\omega)$. 
In forthcoming papers, we shall further develop techniques in order to understand the possible embeddings $D\rightarrow M_G$ beyond the homological level. (See \cite{C} for more discussions.)


In the course of the proof of Theorem 1.1, we also discover the following result which is of independent interest.

\begin{theorem}
Let $X=\C\P^2\# N\overline{\C\P^2}$ where $N=7$, $8$ or $9$. There exist no $N$ disjointly embedded symplectic $(-2)$-spheres 
in $X$.
\end{theorem}

We remark that by a theorem of Ruberman \cite{Ru}, there exist $N$ disjointly embedded smooth 
$(-2)$-spheres in $X=\C\P^2\# N\overline{\C\P^2}$ for any $N\geq 2$. On the other hand, for $N=7$ and $8$, there exist
$N$ homology classes $F_1,F_2,\cdots,F_N\in H_2(X)$, where  $F_i\cdot F_j=0$ for any $i\neq j$,
and each individual $F_i$ can be represented by a symplectic $(-2)$-sphere (cf. Lemma 5.1). 
The above theorem says that these homology classes can not be represented simultaneously by disjoint symplectic $(-2)$-spheres. 
For $N=9$, the corresponding homology classes do not exist (cf. Lemma 5.1). 

The proof of Theorem 1.4 relies on a recent theorem of Ruberman and Starkston, which asserts that the combinatorial line arrangement coming from the Fano plane has no topological
$\C$-realization (cf. \cite{RuS}). Our result and method raises naturally the following interesting 

\vspace{1.5mm}

\noindent{\bf Question:} {\it For each $N\geq 2$, what is the maximal number of disjointly embedded symplectic $(-2)$-spheres in the rational $4$-manifold $\C\P^2\# N\overline{\C\P^2}$?}

\vspace{1.5mm}

We point out that for any $N\geq 3$ and odd, there always exist $N-1$ disjointly embedded symplectic $(-2)$-spheres in 
$\C\P^2\# N\overline{\C\P^2}$. So for $N=7$ and $9$, the maximal number is $6$ and $8$ respectively. 

As for the proof of Theorem 1.1, the case where $G=\Z_2$ and $M_G$ is rational is the most delicate one. Here the key technical
result, stated as Lemma 5.1, is a classification of all possible homological expressions (in a reduced basis) of the classes of any given set of $8$ 
disjointly embedded symplectic $(-2)$-spheres in the rational elliptic surface $\C\P^2\# 9\overline{\C\P^2}$, where the symplectic structure on
$\C\P^2\# 9\overline{\C\P^2}$ is chosen to obey a certain set of delicate area constraints 
on the $(-2)$-spheres
(such a symplectic structure always exists by Lemma 4.1). The proof of Theorem 1.4 also relies on this
technical result. 

\vspace{2.5mm}

The organization of the paper is as follows. In Section 2, we give an examination of the fixed-point
set structure using the traditional methods in group actions, coupled with some standard 
results and techniques in symplectic $4$-manifolds and Seiberg-Witten theory. Section 3 is occupied 
by a study of symplectic surfaces in rational $4$-manifolds. We begin by deriving some basic constraints on the $a$, $b_i$-coefficients of a class $A$ which is represented by a connected,
embedded symplectic surface. The later part of the section focuses on the classes of symplectic spheres; in particular, it contains Lemma 3.6, which gives an upper bound on the $a$-coefficient of a symplectic $(-2)$-sphere or $(-3)$-sphere under an area condition. In Section 4, we begin by 
proving a lemma (i.e., Lemma 4.1) which allows us to freely
impose certain auxiliary area conditions. This lemma, especially when combined with Lemma 3.6, proves to be very critical in our 
analysis of the embedding $D\rightarrow M_G$. We then prove several non-existence results 
concerning certain symplectic configurations in rational $4$-manifolds. These results are used to
further remove some ambiguities concerning $2$-dimensional fixed components in Section 2.
In Section 5, we give proofs of the main theorems.

\section{The fixed-point set: a preliminary examination}

We give a preliminary analysis of the fixed-point set structure, using mainly the traditional methods. For the reader's convenience, we shall begin with a brief review of the various 
$G$-index theorems that will be frequently used in this section. To this end, let $M$ be a symplectic $4$-manifold equipped with a symplectic $G$-action, where $G$ is cyclic of order $n$. For any nontrivial element $g\in G$, the fixed-point set $Fix(g)$ of $g$ consists of a disjoint union of symplectic surfaces $\{Y_i\}$ and isolated points $\{q_j\}$. In general, $Fix(g)$ depends on $g$, but when $G$ is of prime order, $Fix(g)$ coincides with the fixed-point set $M^G$ of the $G$-action. 
The local action of $g$ near $Fix(g)$ is determined by a set of weights $\{c_i\}, \{(a_j,b_j)\}$ (where $0<c_i,a_j,b_j<|g|$) as follows. Along each fixed symplectic surface $Y_i$, the symplectic structure on $M$ determines a complex structure on the normal bundle of $Y_i$. With this understood,
the action of $g$ on the normal bundle of $Y_i$ is given by multiplication of $exp(2\pi ic_i/|g|)$. Likewise, at each isolated fixed point $q_j$, the action of $g$ on the tangent space at $q_j$ has eigenvalues $\exp(2\pi ia_j/|g|)$, $\exp(2\pi ib_j/|g|)$ with respect to a complex structure compatible to the symplectic structure on $M$. We recall that an isolated point $q\in M$ is of isotropy of order $m$ 
of type $(1,b)$ if $q\in Fix(g)$ for some element $g\in G$ of order $|g|=m$ with weights $(1,b)$. 
(Note that when $m=n$, the order of $G$, $q\in M^G$ is an isolated fixed point of $G$.) We remark that $q\in M$ corresponds to an isolated Du Val singularity in $M/G$ precisely when $b=m-1$. 

The fixed-point set $Fix(g)$ and the associated weights $\{c_i\},\{(a_j,b_j)\}$ play a prominent role
in the various $G$-index theorems, which we review next. See \cite{CK2} and the references therein
for more details. We begin with the Lefschetz fixed point theorem and the $G$-signature theorem. 
Recall that the Lefschetz number $L(g,M)$ is defined, for any $g\in G$, as 
$$
L(g,M)=\sum_{k=0}^4 (-1)^k tr(g|_{H^k(M;\R)}),
$$
where $tr(g|_{H^k(M;\R)})$ stands for the trace of the induced action of $g$ on $H^k(M;\R)$. Likewise,
the number $Sign(g,M)$ is defined as 
$$
Sign(g,M)=tr(g|_{H^{2,+}(M;\R)})-tr(g|_{H^{2,-}(M;\R)}),
$$
where for the action of $g$ on $H^{2,+}(M;\R)$ and $H^{2,-}(M;\R)$, we fix a $G$-invariant Riemannian metric on $M$ and look at the action of $G$ on the space of self-dual and anti-self-dual harmonic forms respectively. With this understood, the Lefschetz fixed point theorem states that
$$
L(g,M)=\chi(Fix(g))=\sum_i \chi(Y_i)+ \# \{q_j\},
$$
and the $G$-signature theorem states that 
$$
Sign(g,M)=-\sum_j \cot (\frac{a_j\pi}{|g|})\cdot \cot (\frac{b_j\pi}{|g|})+ 
\sum_i \csc^2(\frac{c_i\pi}{|g|})\cdot Y_i^2.
$$
We remark that when $G$ is of prime order $n=p$, one can sum over all the nontrivial elements 
$g\in G$ and obtain the following (weak) version of the Lefschetz fixed point theorem 
$$
p\cdot \chi(M/G)=\chi(M) + (p-1)\cdot \chi (M^G),
$$
and the $G$-signature theorem
$$
p\cdot Sign(M/G)=Sign(M)+\sum_j def_{q_j} +\sum_i def_{Y_i},
$$
where $def_{q_j}=\sum_{1\neq \lambda\in\C, \lambda^p=1} \frac{(1+\lambda^{a_j})(1+\lambda^{b_j})}
{(1-\lambda^{a_j})(1-\lambda^{b_j})}$, $def_{Y_i}=\frac{p^2-1}{3}\cdot Y_i^2$, which are called the signature defects. 

Next we review the $G$-index theorem for Dirac operators, where we further assume that $M$ is a spin $4$-manifold and $G$ is of an odd prime order $n=p$. In this case, it was shown in \cite{CK2}
that the $G$-action on $M$ must be spin; in particular, the orbifold $M/G$ is spin. With this understood, we fix a $G$-invariant Riemannian metric on $M$ and let $\D$ be the corresponding 
Dirac operator. Then $\text{Ker} \D$ and $\text{Coker} \D$ are complex $G$-representations. For any
nontrivial element $g\in G$, we write 
$$
\text{Ker} \D=\oplus_{k=0}^{p-1} V_k^{+}, \;\; \text{Coker} \D=\oplus_{k=0}^{p-1} V_k^{-},
$$
where $V_k^{+}$, $V_k^{-}$ are the eigenspaces of $g$ with eigenvalue 
$\mu_p^k:=\exp\frac{2k\pi i}{p}$. Then the Spin number $Spin(g,M)$ is defined as
$$
Spin(g,M)=\sum_{k=0}^{p-1} d_k \mu_p^k, \mbox{ where }
d_k\equiv \dim_\C V_k^{+}-\dim_\C V_k^{-}. 
$$
Since both $\text{Ker} \D$ and $\text{Coker} \D$ are quaternion vector spaces, 
and the quaternions $i$ and $j$ are anti-commutative, it follows that $V_0^{\pm}$ are quaternion 
vector spaces, and that multiplication by $j$ maps $V_k^{\pm}$ isomorphically to $V_{p-k}^{\pm}$
for $k>0$. This implies that $d_0$ is even and  $d_k=d_{p-k}$ for $k>0$. Finally, we note that $d_0$
equals the index of the Dirac operator on the spin orbifold $M/G$.

The following formula for $Spin(g,M)$ is given in Lemma 3.8 of 
\cite{CK2}, assuming that the weights of the action of $g$ near $Fix(g)=\{Y_i\}\cup \{q_j\}$ are 
$\{c_i\}$, $\{(a_j,b_j)\}$: 
$$
Spin(g,M)=-\sum_j (-1)^{k(g,q_j)} \frac{1}{4} \csc(\frac{a_j\pi}{p})\csc(\frac{b_j\pi}{p})+
\sum_i (-1)^{k(g,Y_i)} \frac{Y_i^2}{4}  \csc(\frac{c_i\pi}{p})\cot(\frac{c_i\pi}{p}),
$$
where $k(g,q_j)$ is given by the equation $k(g,q_j)\cdot p= 2r_j +a_j+b_j$ for
$0\leq r_j<p$, and $k(g,Y_i)$ is given by the equation $k(g,Y_i)\cdot p= 2r_i+c_i$ for $0<r_i<p$.
This concludes the review of the $G$-index theorems to be used in this section. 

With the preceding understood, for the rest of this section, we assume that $M$ is Calabi-Yau with
$b_1>0$; in particular, $M$ is spin. Furthermore, we assume that no subgroups of $G$ act freely on
$M$. We shall denote by $g_i$ the genus of the symplectic surface $Y_i$.
Then the adjunction formula, together with the fact that $c_1(K_M)=0$, implies that $Y_i^2=2g_i-2$ for each $i$. 

The following homological constraints on $M$ will be frequently used: $2\leq b_1(M)\leq 4$, and 
$\chi(M)=Sign(M)=0$, which implies 
$$
b_2^{+}(M)=b_2^{-}(M)=b_1(M)-1.
$$
Finally, recall from \cite{C3}, Theorem 1.9 and Lemma 4.1, that $M_G$ is of torsion canonical
class if $M/G$ has at most isolated Du Val singularities; otherwise, $M_G$ is either a rational surface,
which occurs precisely when $b_1(M/G)=b_1(M_G)=0$, or $M_G$ is a ruled surface over $T^2$ and 
$b_1(M/G)=b_1(M_G)=2$. We note that the above homological constraints on $M$ apply to $M_G$ as well
when $M_G$ is of torsion canonical class and $b_1(M_G)=b_1(M/G)>0$. When $M_G$ is of torsion canonical class and $b_1(M_G)=0$, we note that either $\chi(M_G)=12$ (where $M_G$ is a homology Enriques surface) or $\chi(M_G)=24$ (where $M_G$ is a homology $K3$ surface). 

Now we begin with our analysis on the fixed-point set structures. First, we observe the following lemma. 

\begin{lemma}
Suppose $b_1(M)=2$ or $3$, and $G$ is of prime order $p$ such that $M_G$ has torsion canonical class.  Then $p=2$ and $M^G$ consists of $8$ isolated points. Furthermore, $b_1(M/G)=0$ and $b_2^{+}(M/G)=1$.
\end{lemma}

\begin{proof}
Since $M_G$ has torsion canonical class, $M/G$ has only isolated Du Val singularities
(cf. \cite{C3}, Lemma 4.1). By the Lefschetz fixed point theorem,
$$
p\cdot \chi(M/G)=\chi(M)+ (p-1)\cdot \# M^G.
$$
With $\chi(M)=0$, and observing that the resolution of each singular point of $M/G$ is a chain 
of $p-1$ spheres, we obtain the following expression 
$$
\chi(M_G)=\chi(M/G)+(p-1) \cdot \# M^G=(p-1)(\frac{1}{p}+1)\cdot \# M^G.
$$
On the other hand, note that $\chi(M_G)=0$, $12$, or $24$.
It is clear that $\chi(M_G)>0$, as $M^G\neq \emptyset$, so that $\chi(M_G)=12$ or $24$. We also note that $b_1(M_G)=0$
in these two cases. Moreover, since $b_1(M)=2$ or $3$, we have 
$b_2^{+}(M/G)\leq b_2^{+}(M)=b_1(M)-1\leq 2$, so that $\chi(M_G)=12$ must be true. The equation 
$(p-1)(\frac{1}{p}+1)\cdot \# M^G=12$ has only one solution: $p=2$ and 
$\# M^G=8$. Finally, note that $b_1(M/G)=b_1(M_G)=0$,  and $b_2^{+}(M/G)=b_2^{+}(M_G)=1$.
This finishes off the proof.

\end{proof}

\subsection{The case where $b_1=2$}
We first assume $G$ is of prime order $p$. Let $g\in G$ be a generator of $G$.
Let $\{q_j\}$ be the set of isolated fixed points and set $z:=\#\{q_j\}$, and let $\{Y_i\}$ be the
set of $2$-dimensional fixed components, with $g_i$ the genus of $Y_i$. 

We begin with the case where $M_G$ is irrational ruled. Note that this happens exactly when
$b_1(M/G)=2=b_1(M)$, which means that the action of $G$ on $H^1(M;\R)$ is trivial. 

\begin{lemma}
Suppose $G$ is of prime order and $M_G$ is irrational ruled. Then the fixed-point set $M^G$ 
consists of a disjoint union of tori of self-intersection zero.
\end{lemma}

\begin{proof}
We begin by observing $b_2^{-}(M)=b_1(M)-1=1$, so that either $b_2^{-}(M/G)=0$ or 
$b_2^{-}(M/G)=1$. We claim $b_2^{-}(M/G)=1$. To see this, suppose to the contrary that 
$b_2^{-}(M/G)=0$. Then $G=\Z_2$ must be true. With this understood, with 
$b_2^{+}(M)=b_1(M)-1=1$ and $b_2^{+}(M/G)=b_2^{+}(M_G)=1$, 
the Lefschetz fixed point theorem gives 
$$
\sum_i (2-2g_i)+ z=L(g,M)=2-2\times 2+1-1=-2,
$$
and the $G$-signature theorem gives 
$$
\sum_i Y_i^2=Sign(g,M)=1-(-1)=2.
$$
With $Y_i^2=2g_i-2$ for each $i$, it follows easily that $z=0$, i.e., there are no isolated fixed points.
As a consequence, we note that the underlying space of $M/G$ is smooth, and it is simply the resolution $M_G$, which is an irrational ruled 
$4$-manifold by the assumption. But this implies that $b_2^{-}(M/G)=b_2^{-}(M_G)\geq 1$, contradicting the assumption $b_2^{-}(M/G)=0$. 
Hence we must have $b_2^{-}(M/G)=1$. With $b_2^{-}(M/G)=1$, it
follows easily that $L(g,M)=0$ and $Sign(g,M)=0$. 

The equation $L(g,M)=0$ implies $z=\sum_i (2g_i-2)=\sum_i Y_i^2$. Suppose to the contrary that $z>0$. Then there must be a component $Y_i$ such that $Y_i^2>0$. 
Since $b_2^{+}(M/G)=b_2^{+}(M_G)=1$ as $M_G$ is irrational ruled, it follows easily that there can be only one such component. As a consequence, by replacing $g$ with a suitable power, we may assume that the weight of the action of $g$ along the component $Y_i$ with $Y_i^2>0$ equals $1$
(i.e., $c_i=1$ if $Y_i^2>0$). It follows from the $G$-signature theorem that 
\begin{eqnarray*}
Sign(g,M) & = & -\sum_j \cot (\frac{a_j\pi}{p})\cdot \cot (\frac{b_j\pi}{p})+ \sum_i \csc^2(\frac{c_i\pi}{p}) Y_i^2\\
                 & > & \sum_i (\csc^2(\frac{c_i\pi}{p})-\csc^2(\frac{\pi}{p}) )\cdot Y_i^2,
 \end{eqnarray*}
where we use the fact that $z=\sum_i Y_i^2$ and $\cot (\frac{a_j\pi}{p})\cdot \cot (\frac{b_j\pi}{p})< \csc^2(\frac{\pi}{p})$ for each $j$. Since $c_i=1$ when $Y_i^2>0$, it follows easily that 
$\sum_i (\csc^2(\frac{c_i\pi}{p})-\csc^2(\frac{\pi}{p}) )\cdot Y_i^2\geq 0$.
This leads to a contradiction that $Sign(g,M)>0$, hence $z=0$ must be true. 

With $z=0$, $M_G$ is simply the underlying manifold $|M/G|$, which must be a $\s^2$-bundle over $T^2$ as $b_2^{-}(M/G)=1$. It remains to show that each $Y_i$ is a torus. This follows easily by observing that $\sum_i (2g_i-2)=z=0$, and that $g_i>0$ for each $i$. The latter is true because if 
$Y_i$ is a sphere, then $Y_i^2=-2$, so that $Y_i$ descends to a $(-2p)$-sphere in $M_G$. But $M_G$ is a $\s^2$-bundle over $T^2$, it does not contain any $(-2p)$-sphere. This finishes off the proof.

\end{proof}

Next we consider the case where $M_G$ is rational; note that this happens exactly when 
$b_1(M/G)=0$. With $b_1(M)=2$, the action of $G$ on $H^1(M;\Z)/\text{Tor}$ is given by
elements of $SL(2,\Z)$. It follows easily that $G$ is either $\Z_2$ or $\Z_3$. 

\begin{lemma}
Suppose $M_G$ is rational and $G=\Z_2$. Then $G$ has $8$ isolated fixed points. 
Furthermore, $\{Y_i\}\neq\emptyset$ and $\sum_i Y_i^2=2(1-b_2^{-}(M/G))$. 
\end{lemma}

\begin{proof}
For $G=\Z_2$, we first observe that the $G$-Signature theorem gives 
$$
1-tr(g|_{H^{2,-}})=Sign(g,M)=\sum_i Y_i^2=\sum_i (2g_i-2).
$$
On the other hand, the Lefschetz fixed point theorem implies that
$$
z+\sum_i (2-2g_i)=L(g,M)=2-4\times (-1)+1+tr(g|_{H^{2,-}})=8-\sum_i (2g_i-2). 
$$
It follows that $z=8$. Finally,  $\sum_i Y_i^2=1-tr(g|_{H^{2,-}})=2(1-b_2^{-}(M/G))$
because $b_2^{-}(M)=1$. Note that $\{Y_i\}\neq\emptyset$ because $M_G$ is rational. This finishes the proof. 

\end{proof}

\begin{lemma}
Suppose $M_G$ is rational and $G=\Z_3$. Then $G$ has $6$ isolated fixed points, exactly three of which are of type $(1,1)$.  Furthermore, 
$\sum_i Y_i^2=0$, and at most one of the components in $\{Y_i\}$ is a sphere.
\end{lemma}

\begin{proof}
First of all, observe that $b_2^{-}(M/G)=1$ as $G=\Z_3$, and consequently, 
$$
L(g,M)=2-4\times (-\frac{1}{2})+1+1=6,\;\; Sign(g,M)=1-1=0.
$$
If we let $x,y$ be the number of isolated fixed points of $G$ which are of type $(1,1)$ and type $(1,2)$ respectively, then the Lefschetz fixed point theorem and the 
$G$-Signature theorem imply, respectively, that
$$
x+y+ \sum_i (2-2g_i)=6 \mbox{ and } -\frac{1}{3}x+ \frac{1}{3} y+\frac{4}{3}\cdot \sum_i Y_i^2=0.
$$
With $Y_i^2=2g_i-2$, we eliminate the variable $x$ and obtain $2y+3\sum_i Y_i^2=6$. On the other hand, observe that $b_2^{-}(M/G)=1$ implies that there is at most one component $Y_i$ 
such that $Y_i^2<0$ (note that these are precisely the spherical components in $\{Y_i\}$). Consequently, it is easily seen that 
$\sum_i Y_i^2\geq -2$, and with this, it follows easily that $y=0$ or $3$ are the only possibilities, where $x=8$ or $3$ and $\sum_i Y_i^2=2$ or $0$ respectively. 

It remains to eliminate the possibility that $x=8$, $y=0$ and $\sum_i Y_i^2=2$. To this end, we observe that the $G$-action is spin because the order of $G$ is an odd prime (cf.  \cite{CK2}).
Moreover, the index of the Dirac operator on the spin orbifold $M/G$ must be zero because 
$b_2^{+}(M/G)=b_2^{-}(M/G)=1$ (cf. Fukumoto-Furuta \cite{FF}, Corollary 1). We shall prove 
the index is nonzero, thus eliminating the case $x=8$, $y=0$ and $\sum_i Y_i^2=2$.

First, let $\D$ be the Dirac operator on $M$. Then as $G=\Z_3$, it follows easily from 
$\text{Index } \D=-\frac{1}{8} Sign (M)=0$ that $Spin(g,M)=\frac{3}{2} d_0$, where $d_0$ equals the index of the Dirac operator on $M/G$.

Next we compute $Spin(g,M)$ using the $G$-index theorem for Dirac operators (cf. \cite{CK2}, Lemma 3.8). In order to apply the formula for $Spin(g,M)$, we note that for a type $(1,1)$ isolated 
fixed point $q_j$, the number $k(g,q_j)=2$, and for a type $(1,2)$ isolated fixed point $q_j$, $k(g,q_j)=1$. On the other hand, it is easy to check that $k(g,Y_i)=c_i$ for any $Y_i$. 
With these understood, it follows easily that the Spin number 
$$
Spin(g,M)=-\frac{1}{3}x+\frac{1}{3}y+\sum_i(-\frac{1}{6} Y_i^2)
=-\frac{1}{3}\times 8+\frac{1}{3}\times 0-\frac{1}{6}\times 2=-3. 
$$
Consequently, $d_0=\frac{2}{3} Spin(g,M)=-2$, which is nonzero. This finishes the proof. 

\end{proof}

Finally, we consider the case where $G$ is of non-prime order $n$. We assume that $M_G$ is rational or ruled, and that for any subgroup $H$ of prime order, $M_H$ has torsion canonical class. 

First, by Lemma 2.1, the order $n$ must be a power of $2$; more precisely, $n=2^k>2$. Furthermore, $b_1(M/G)=0$, so that $M_G$ must be rational. Finally, note that the action of $G$ on 
$H^1(M;\Z)/\text{Tor}$ is given by elements of $SL(2,\Z)$. It follows easily that $n=4$. 

With the preceding understood, we fix a generator $g$ of $G$, and let $H$ be the subgroup of order $2$ generated by $h:=g^2$. Then by our assumption, $M_H$ has torsion canonical class.
By Lemma 2.1, $M^H$ consists of $8$ isolated fixed points. Since $M^G$ is contained in $M^H$, the action of $G$ has no $2$-dimensional fixed components.

To proceed further, note that there are two possibilities: $b_2^{-}(M/G)=0$ or $1$. Consider first the case where $b_2^{-}(M/G)=0$.
In this case, $L(g,M)=2-4\times 0+1-1=2$, so the $G$-action has $2$ isolated fixed points. 
Examining the induced action of $G$ on $M^H$, the remaining $6$ fixed points of $H$ are 
of isotropy of order $2$, and consequently, the orbifold $M/G$ has $5$ singular points -- two of order $4$ and three of order $2$. 
Let $x,y$ be the number of fixed points of $G$ of type $(1,1)$ and type $(1,3)$ respectively. Note that the resolution of a type $(1,1)$
fixed point in $M_G$ is a $(-4)$-sphere and the resolution of a type $(1,3)$ fixed point is a linear chain of three $(-2)$-spheres.
A point of isotropy of order $2$ gives rise to a $(-2)$-sphere in $M_G$. As a result, we have
$$
b_2^{-}(M_G)=b_2^{-}(M/G)+x+3y+3=x+3y+3.
$$
On the other hand, $c_1(K_{M_G})=\sum_i -\frac{1}{2} E_i$, where $E_i$ are the $(-4)$-spheres in $M_G$ coming from the resolution of type $(1,1)$ fixed points of $G$ (cf. \cite{C3},
Proposition 3.2). Thus $c_1(K_{M_G})^2=\sum_i \frac{1}{4} E_i^2=-x$. Since $M_G$ is rational,
we have $c_1(K_{M_G})^2=9-b_2^{-}(M_G)$, which is $-x=9-(x+3y+3)$. It follows that 
$y=2$, and $x=2-y=0$. But this is a contradiction as it implies that $c_1(K_{M_G})=0$. Hence the case where $b_2^{-}(M/G)=0$ is eliminated. 

For the case where $b_2^{-}(M/G)=1$, it is easy to see that $L(g,M)=4$, so the $G$-action has $4$ isolated fixed points. A similar calculation results 
$$
b_2^{-}(M_G)=x+3y+3 \mbox{ and } c_1(K_{M_G})^2=-x.
$$ 
It follows easily that $x=y=2$. We summarize our discussions in the following

\begin{lemma}
Suppose $M_G$ is rational or ruled, but for any subgroup $H$ of prime order, $M_H$ has torsion canonical class. Then $M_G$ must be rational, and $G$ is of order $4$.  Furthermore, the fixed-point set $M^G$ consists of $4$ isolated points, exactly two of which are of type $(1,1)$, and there are 
$4$ isolated points of isotropy of order $2$ in $M$. 
\end{lemma}

\subsection{The case where $b_1=3$}
Assume $M_G$ is rational or ruled. We first observe that $G$ must be $\Z_2$, which, with
$b_2^{+}(M)=b_1(M)-1=2$ and $b_2^{+}(M/G)=1$, follows easily by Lemma 2.1. 

\begin{lemma}
Suppose $M_G$ is rational or ruled. Then $G$ must be of order $2$. Moreover, 
\begin{itemize}
\item [{(i)}] if $M_G$ is irrational ruled, then the fixed-point set $M^G$ consists of a disjoint union 
of tori of self-intersection zero;
\item [{(ii)}] if $M_G$ is rational, then the fixed-point set $M^G$ contains $8$ isolated points, and 
the $2$-dimensional fixed components $\{Y_i\}\neq \emptyset$ and 
$\sum_i Y_i^2=2(1-b_2^{-}(M/G))$.
\end{itemize}
\end{lemma}

\begin{proof}
Since $M_G$ is rational or ruled and $G=\Z_2$, $\{Y_i\}\neq \emptyset$. We denote by $z$ the number of isolated fixed points and let $1\neq g\in G$. Then by the $G$-Signature theorem, 
$$
\sum_i Y_i^2=Sign(g,M)=(1-1)-tr(g|_{H^{2,-}})=-tr(g|_{H^{2,-}}).
$$

First, consider case (i) where $M_G$ is irrational ruled. In this case, $b_1(M/G)=2$, 
so the Lefschetz fixed point theorem implies that
$$
z+\sum_i (2-2g_i)=L(g,M)=2-2\times (1+1-1)+(1-1)+ tr(g|_{H^{2,-}})= tr(g|_{H^{2,-}}).
$$
With $Y_i^2=2g_i-2$ for each $i$, it follows immediately that $z=0$. As a consequence, $M_G$ is simply
the underlying manifold of $M/G$. This immediately ruled out the possibility that $b_2^{-}(M/G)=0$, because as an irrational ruled $4$-manifold, $M_G$ has non-zero $b_2^{-}$. 

Next, assume $b_2^{-}(M/G)=1$. In this case, by the same argument as in Lemma 2.2, 
each $Y_i$ is a torus of self-intersection zero and $M_G$ is a $\s^2$-bundle over $T^2$.

Finally, we rule out the possibility that $b_2^{-}(M/G)=2$. In this case, 
$\sum_i Y_i^2=-tr(g|_{H^{2,-}})=-(1+1)=-2$, so that there
must be a $Y_i$ which is a $(-2)$-sphere. On the other hand, $b_2^{-}(M/G)=2$ implies 
that $M_G$ is a $\s^2$-bundle over $T^2$
blown up at one point. The descendent of $Y_i$ is a symplectic $(-4)$-sphere in $M_G$,
to be denoted by $C$. To derive a contradiction, let $F$ and $E$ be the fiber class and 
the exceptional $(-1)$-class of $M_G$ respectively. Note that $c_1(K_{M_G})\cdot F=-2$ and $c_1(K_{M_G})\cdot E=-1$. With this understood, 
since $\pi_2(M_G)$ is generated by $F$ and $E$, we write $C=aF+bE$. Then $-4=C^2=-b^2$ and $2=c_1(K_{M_G})\cdot C=-2a-b$, giving either $C=-2F+2E$ or $C=-2E$. Note that in both cases, 
$C$ has a negative symplectic area. Hence the possibility $b_2^{-}(M/G)=2$ is ruled out.

For case (ii) where $M_G$ is rational, $b_1(M/G)=0$. In this case, the Lefschetz number 
$L(g,M)=2-2\times (-1-1-1)+(1-1)+ tr(g|_{H^{2,-}})= 8+tr(g|_{H^{2,-}})$, which implies $z=8$.
The assertion $\sum_i Y_i^2=2(1-b_2^{-}(M/G))$ follows easily from the fact that
$tr(g|_{H^{2,-}})=2(b_2^{-}(M/G)-1)$. This finishes the proof. 

\end{proof}

\subsection{The case where $b_1=4$}
The fact that the cohomology ring $H^\ast(M;\R)$ is isomorphic to that of $T^4$ (cf. \cite{Sr})
plays a crucial role in the analysis of the fixed-point set structure in this case. In particular,
this fact has the following two corollaries: (1) it allows us to express the action of $G$ on the entire cohomology $H^\ast(M;\R)$ in terms of its action on $H^1(M;\R)$, and (2) since the Hurwitz map 
$\pi_2(M)\rightarrow H_2(M)$ has trivial image, the fixed-point set $M^G$ does not have any
spherical components. With the help of the adjunction formula, this is equivalent to the statement 
that all the $2$-dimensional fixed components have nonnegative self-intersection. 

For the first point above, to be more concrete, let $g\in G$ be any nontrivial element. 
Since the action of $g$ on $M$ is orientation-preserving, the representation of $g$ on $H^1(M;\R)$ splits into a sum of two 
complex $1$-dimensional representations. This said, there is a basis $\{\alpha_i\}$, $i=1,2,3,4$,  
of $H^1(M;\R)$ such that $\alpha_1\cup \alpha_2\cup\alpha_3\cup \alpha_4\in H^4(M;\R)$ is positive according to the natural orientation of $M$. Furthermore, we assume that the span of 
$\alpha_1,\alpha_2$ and the span of  $\alpha_3,\alpha_4$
are invariant under the action of $g$, and with respect to the orientation given by the
above order, the action of $g$ is given by a rotation of angle $\theta_1$, $\theta_2$
respectively. 

\begin{lemma}
With $g, \theta_1, \theta_2$ as given above, the following hold true:
\begin{itemize}
\item [{(1)}] $2(\cos \theta_1+\cos\theta_2), 4\cos\theta_1\cos\theta_2 \in\Z$.
\item [{(2)}] The Lefschetz number $L(g,M)=4(1-\cos\theta_1)(1-\cos\theta_2)$.
\item [{(3)}] The representation of $g$ on $H^{2,+}(M;\R)$ {\em(}resp.  $H^{2,-}(M;\R)${\em)} splits into 
a trivial $1$-dimensional representation and a $2$-dimensional one on which $g$ acts as a rotation 
of angle $\theta_1+\theta_2$ {\em(}resp. $\theta_1-\theta_2${\em)}. Consequently, 
$$Sign(g,M)=2(\cos (\theta_1+\theta_2)-\cos (\theta_1-\theta_2))=-4\sin\theta_1\sin\theta_2.$$
\end{itemize}
\end{lemma}

\begin{proof}
Let $\gamma_1:=\alpha_1\cup\alpha_3$, $\gamma_2:=\alpha_1\cup\alpha_4$, 
$\gamma_3:=\alpha_2\cup\alpha_3$, and $\gamma_4:=\alpha_2\cup\alpha_4$. Then
a straightforward calculation gives 
$$
g\cdot (\alpha_1\cup\alpha_2)=\alpha_1\cup\alpha_2,\; g\cdot (\alpha_3\cup\alpha_4)=\alpha_3\cup\alpha_4,
$$
$$
g\cdot \gamma_1=\cos\theta_1\cos\theta_2\gamma_1+\cos\theta_1\sin\theta_2\gamma_2+
\sin\theta_1\cos\theta_2\gamma_3+\sin\theta_1\sin\theta_2\gamma_4,
$$
$$
g\cdot \gamma_2=-\cos\theta_1\sin\theta_2\gamma_1+
\cos\theta_1\cos\theta_2\gamma_2-\sin\theta_1\sin\theta_2\gamma_3+\sin\theta_1\cos\theta_2\gamma_4,
$$
$$
g\cdot \gamma_3=-\sin\theta_1\cos\theta_2\gamma_1-\sin\theta_1\sin\theta_2\gamma_2+
\cos\theta_1\cos\theta_2\gamma_3+\cos\theta_1\sin\theta_2\gamma_4,
$$
and
$$
g\cdot \gamma_4=\sin\theta_1\sin\theta_2\gamma_1-\sin\theta_1\cos\theta_2\gamma_2
-\cos\theta_1\sin\theta_2\gamma_3+\cos\theta_1\cos\theta_2\gamma_4.
$$
The action on $H^3(M;\R)$ can be similarly determined. From these calculations we deduce easily that 
$$
L(g,M)=2-4(\cos\theta_1+\cos\theta_2)+(2+4\cos\theta_1\cos\theta_2)=4(1-\cos\theta_1)
(1-\cos\theta_2).
$$
In order to understand the action of $g$ on $H^{2,+}(M;\R)$ and $H^{2,-}(M;\R)$, and to 
compute $Sign(g,M)$, we note that $H^{2,+}(M;\R)$ is spanned by
$\beta_i$, $i=1,2,3$, where
$$
\beta_1=\alpha_1\cup\alpha_2+\alpha_3\cup\alpha_4,\;\;
\beta_2=\alpha_1\cup\alpha_3-\alpha_2\cup\alpha_4,\;\;
\beta_3=\alpha_1\cup\alpha_4+\alpha_2\cup\alpha_3.
$$
Likewise, $H^{2,-}(M;\R)$ is spanned by $\beta_i^\prime$, $i=1,2,3$, where
$$
\beta_1^\prime=\alpha_1\cup\alpha_2-\alpha_3\cup\alpha_4,\;\;
\beta_2^\prime=\alpha_1\cup\alpha_3+\alpha_2\cup\alpha_4,\;\;
\beta_3^\prime=\alpha_1\cup\alpha_4-\alpha_2\cup\alpha_3.
$$
With this understood, the action of $g$ on $H^{2,+}(M;\R)$ and $H^{2,-}(M;\R)$ is as follows:
both $\beta_1$ and $\beta_1^\prime$ are fixed by $g$, and $g$ acts on the span of
$\beta_2,\beta_3$ and the span of $\beta_2^\prime,\beta_3^\prime$ as a rotation of angle
$\theta_1+\theta_2$, $\theta_1-\theta_2$ respectively. It follows in particular that
$Sign(g,M):=tr(g|_{H^{2,+}})-tr(g|_{H^{2,-}})$ is given by 
$$
Sign(g,M)=2(\cos (\theta_1+\theta_2)-\cos (\theta_1-\theta_2))=-4\sin\theta_1\sin\theta_2.
$$
Finally, note that $tr(g|_{H^1(M;\R)})=2(\cos \theta_1+\cos\theta_2)$, hence 
$2(\cos \theta_1+\cos\theta_2)\in\Z$. With this, $L(g,M)=4(1-\cos\theta_1)(1-\cos\theta_2)\in \Z$
implies that $4\cos\theta_1\cos\theta_2 \in\Z$ as well.  This completes the proof of the lemma.

\end{proof}

With Lemma 2.7 at hand, we shall first examine the fixed-point set structure when 
$G$ is of prime order.

\begin{lemma}
Suppose $G$ is of prime order $p>1$. Then the following hold true.
\begin{itemize}
\item [{(1)}] Either $b_2^{+}(M/G)=1$ or $b_2^{+}(M/G)=3$. Moreover, $M_G$ has torsion canonical 
class if and only if $b_2^{+}(M/G)=3$ and $b_1(M/G)=0$. 
\item [{(2)}] If $M_G$ has torsion canonical class, then $p=2$ or $p=3$, where in the former case,
the fixed-point set $M^G$ consists of $16$ isolated points, and in the latter case, $M^G$ consists of
$9$ isolated points of type $(1,2)$. 
\item [{(3)}] If $M_G$ is irrational ruled, then $M^G$ consists of a disjoint union of tori of self-intersection zero.
\item [{(4)}] If $M_G$ is rational, then $p\neq 2$ and $p\leq 5$.
\end{itemize}
\end{lemma}

\begin{proof}
For (1), note that by Lemma 2.7, $b_2^{+}(M/G)=3$ if and only if $\theta_1+\theta_2=2\pi$
for a generator $g$ of $G$. If $\theta_1+\theta_2\neq 2\pi$, then $b_2^{+}(M/G)=1$.
Hence either $b_2^{+}(M/G)=1$ or $b_2^{+}(M/G)=3$ as claimed. It remains to show that
if $M_G$ has torsion canonical class, then $b_2^{+}(M/G)\neq 1$ but $b_1(M/G)=0$.
To see this, suppose $M_G$ has torsion canonical class. Then the same argument as in 
Lemma 2.1 shows that $\chi(M_G)=12$ or $24$, and $b_1(M/G)=0$. If $b_2^{+}(M/G)=1$, 
then $\chi(M_G)=12$, and as in Lemma 2.1, $p=2$ must be true. 
With $p=2$ and $b_1(M/G)=0$, the angles $\theta_1,\theta_2$ in Lemma 2.7 must be both 
equal to $\pi$. But this implies that $b_2^{+}(M/G)=3$, contradicting the assumption of 
$b_2^{+}(M/G)=1$. Hence part (1) is proved. 

Part (2) follows readily from the same argument as in Lemma 2.1. Note that when $\chi(M_G)=24$, $p=2,3$ or $5$. The case of $p=5$ can be further eliminated by the (weak version) $G$-signature theorem. 

For part (3), if $M_G$ is irrational ruled, then $b_1(M/G)=2$. This means that in Lemma 2.7,
one of the angles $\theta_1,\theta_2$ must be $0$. As a corollary, $L(g,M)=Sign(g,M)=0$
for any nontrivial element $g\in G$, and $b_2^{-}(M/G)=1$. With this understood, part (3) follows
by the same argument as in Lemma 2.2. 

Finally, for part (4) we assume $M_G$ is rational. Then $b_2^{+}(M/G)=1$ and $b_1(M/G)=0$,
so that by Lemma 2.7, $p\neq 2$. On the other hand, assume $p\geq 5$. We fix a generator 
$g\in G$ such that in Lemma 2.7, the angles $\theta_1=\frac{2\pi}{p}$ and $\theta_2=\frac{2q\pi}{p}$ for some $0<q<p-1$ (note that $q\neq p-1$ as $b_2^{+}(M/G)=1$). 
Then it follows easily from $p\geq 5$ that 
$L(g,M)=4(1-\cos \theta_1)(1-\cos \theta_2)$ satisfies the bound
$L(g,M)\leq 7$. With this understood, we appeal to the following version of Lefschetz fixed
point theorem 
$$
p\cdot \chi(M/G)=\chi(M)+(p-1)\cdot L(g,M),
$$
where $\chi(M)=0$ and $L(g,M)\in \Z$. It follows easily that $L(g,M)$ is divisible by $p$, and 
with $p\geq 5$ and $L(g,M)\leq 7$, we have $L(g,M)=p$. A further examination easily removes the possibility that $p=7$. Hence $p\leq 5$. This finishes the proof of the lemma.

\end{proof}

In the next two lemmas, we shall determine the fixed-point set structure where $M_G$ is rational and $G=\Z_3$ or $\Z_5$.
Let $g\in G$ be a generator.

\begin{lemma}
Assume $M_G$ is rational and $G=\Z_3$. Then the fixed-point set $M^G$ consists of $9$ isolated points 
of type $(1,1)$, plus possible $2$-dimensional components $\{Y_i\}$ which are tori of self-intersection zero. 
\end{lemma}

\begin{proof}
We observe that since $M_G$ is rational, $b_1(M/G)=0$, which implies that the angles 
$\theta_1,\theta_2$ in Lemma 2.7 are both nonzero. Furthermore, $b_2^{+}(M/G)=1$ and 
$G=\Z_3$, which implies $\theta_1=\theta_2$. It follows easily that $L(g,M)=9$
and $Sign(g,M)=-3$. 

With this understood, let $x,y$ be the number of isolated fixed points of type $(1,1)$ and 
type $(1,2)$ respectively. Then the Lefschetz fixed point theorem and  the $G$-signature theorem
imply that 
$$
x+y-\sum_i Y_i^2=9 \mbox{ and } -\frac{1}{3}x+\frac{1}{3}y+\frac{4}{3}\sum_i Y_i^2=-3. 
$$
Combining the two equations, we get $x+\frac{5}{3}y=9$. It is easy to see that the solutions are
$x=9,y=0$ or $x=4,y=3$. In the former case, $\sum_i Y_i^2=0$, while in the latter case,
$\sum_i Y_i^2=-2$. The latter case is not possible since $Y_i^2\geq 0$ for all $i$.
For the same reason, we must have $Y_i^2=0$ for all $i$ in the former case. By the
adjunction formula, each $Y_i$ is a torus. This finishes the proof. 

\end{proof}

\begin{lemma}
Assume $M_G$ is rational and $G=\Z_5$. Then the fixed-point set $M^G$ consists of $5$ isolated 
points of type $(1,2)$, plus possible $2$-dimensional components $\{Y_i\}$ which are
tori of self-intersection zero.
\end{lemma}

\begin{proof}
We shall first apply the Lefschetz fixed point theorem and the weak version of the $G$-signature
theorem. To this end, recall from the proof of Lemma 2.8(4), that $L(g,M)=5$ and $\chi(M/G)=4$.
The latter easily implies that $Sign(M/G)=0$. On the other hand, note that the signature defect for an isolated fixed point of type $(1,1)$, $(1,2)$ (the same as $(1,3)$) and $(1,4)$ is $-4$, $0$, $4$ respectively (cf. \cite{CK1}). Thus if we let $x,y,z$ be the number of fixed points of type $(1,1)$, $(1,4)$ and $(1,2)$ respectively, then
$$
x+y+z-\sum_i Y_i^2=5 \mbox{ and} -4x+4y+\sum_i \frac{5^2-1}{3}Y_i^2=0.
$$
Combining the two equations, we have $x+3y+2z=10$. Note that $x+y+z$ must be odd,
because $\sum_iY_i^2=\sum_i (2g_i-2)$ is even. It follows that $z$ must be odd. 
The solutions of $x,y,z$ and $\sum_i Y_i^2$ are listed below:
\begin{itemize}
\item [{(1)}] $x=8, y=0,z=1$, and $\sum_i Y_i^2=4$,
\item [{(2)}] $x=5, y=1,z=1$, and $\sum_i Y_i^2=2$,
\item [{(3)}] $x=2, y=2,z=1$, and $\sum_i Y_i^2=0$,
\item [{(4)}] $x=4, y=0,z=3$, and $\sum_i Y_i^2=2$,
\item [{(5)}] $x=1, y=1,z=3$, and $\sum_i Y_i^2=0$,
\item [{(6)}] $x=0, y=0,z=5$. and $\sum_i Y_i^2=0$.
\end{itemize}

Next we shall first eliminate cases (1),(2), and (4) where $\sum_i Y_i^2\neq 0$ by 
computing with the $G$-index theorem for Dirac operators, using the formula for the Spin number
$Spin(g,M)$ in Lemma 3.8 of \cite{CK2}. To this end, we divide the isolated fixed points $\{q_j\}$ 
of each type and the fixed components $\{Y_i\}$ into two groups, I and II, according to the following rule: for type $(1,1)$, group I consists of fixed points $q_j$ with $(a_j,b_j)=(1,1)$ or $(4,4)$ (and the rest are group II), for type $(1,4)$, a fixed point $q_j$ belongs to group I if $(a_j,b_j)=(1,4)$, and 
to group II if $(a_j,b_j)=(2,3)$, and for type $(1,2)$, group I consists of fixed points $q_j$ with 
$(a_j,b_j)=(1,2)$ or $(3,4)$, and group II consists of fixed points $q_j$ with $(a_j,b_j)=(2,4)$ 
or $(1,3)$, and finally, for a fixed component $Y_i$, it belongs to group I if and only if $c_i=1$
or $4$. With this understood, the contribution to the Spin number $Spin(g,M)$ from an isolated
fixed point $q_j$ takes values as follows:
\begin{itemize}
\item $-\frac{1}{4}\csc^2\frac{\pi}{5}$ if $q_j$ is in group I and of type $(1,1)$,
\item $-\frac{1}{4}\csc^2\frac{2\pi}{5}$ if $q_j$ is in group II and of type $(1,1)$,
\item $\frac{1}{4}\csc^2\frac{\pi}{5}$ if $q_j$ is in group I and of type $(1,4)$,
\item $\frac{1}{4}\csc^2\frac{2\pi}{5}$ if $q_j$ is in group II and of type $(1,4)$,
\item $\frac{1}{4}\csc\frac{\pi}{5}\csc\frac{2\pi}{5}$ if $q_j$ is in group I and of type $(1,2)$,
\item $-\frac{1}{4}\csc\frac{\pi}{5}\csc\frac{2\pi}{5}$ if $q_j$ is in group II and of type $(1,2)$,
\end{itemize}
and the contribution from a fixed component $Y_i$ takes values as follows:
\begin{itemize}
\item $-\frac{1}{4}Y_i^2\csc\frac{\pi}{5}\cot\frac{\pi}{5}$ if $Y_i$ is in group I, 
\item $\frac{1}{4}Y_i^2\csc\frac{2\pi}{5}\cot\frac{2\pi}{5}$ if $Y_i$ is in group II.
\end{itemize}
If we denote by $x_k$, $y_k$, $z_k$, for $k=1,2$, the number of fixed points $q_j$ belonging to
group I, II, of type $(1,1)$, $(1,4)$, and $(1,2)$ respectively, and we denote by $w_1,w_2$
the sum of $Y_i^2$ for $Y_i$ belonging to group I, II respectively, 
then the Spin number 
$$
Spin(g,M)=\frac{1}{4}(\sum_{k=1}^2(y_k-x_k)\csc^2\frac{k\pi}{5}
+(-1)^k w_k\csc\frac{k\pi}{5}\cot\frac{k\pi}{5}+
(z_1-z_2)\csc\frac{\pi}{5}\csc\frac{2\pi}{5}).
$$
Now the key observation is that for $g^2$, the contributions to the Spin number for group I and 
group II switch values. It follows easily then, with the identities $\sum_{k=1}^2\csc^2\frac{k\pi}{5}=4$
and $\sum_{k=1}^2(-1)^k \csc\frac{k\pi}{5}\cot\frac{k\pi}{5}=-2$, that
$$
Spin(g,M)+Spin(g^2,M)=\sum_{k=1}^2(y_k-x_k-\frac{1}{2} w_k)=y-x-\frac{1}{2}\sum_i Y^2_i
=-\frac{5}{2}\sum_i Y^2_i.
$$
(Note that $2\sum_i Y_i^2=x-y$ from the weak version of $G$-signature theorem.) 
On the other hand, recall that in the definition of Spin number
$$
Spin(g,M)=d_0+d_1\mu+d_2\mu^2+d_3\mu^3+d_4\mu^4, \mbox{ where } \mu=\exp(2\pi i/5),
$$
one has $d_1=d_4$, $d_2=d_3$. As $Spin(g^2,M)=d_0+d_1\mu^2+d_2\mu^4+d_3\mu+d_4\mu^3$, 
it follows easily that
$$
-\frac{5}{2}\sum_i Y^2_i =Spin(g,M)+Spin(g^2,M)=2d_0-d_1-d_2. 
$$
Finally, $d_0+d_1+d_2+d_3+d_4=\text{Ind }\D=-Sign(M)/8=0$. It follows immediately that
$d_0=-\sum_i Y_i^2$. The integer $d_0$ is the index of Dirac operator on the spin orbifold $M/G$, 
which equals $0$ because $b_2^{-}(M/G)=b_2^{+}(M/G)=1$ (see Fukumoto-Furuta \cite{FF}, 
Corollary 1). This rules out the cases (1),(2),(4), where $d_0=-\sum_i Y_i^2\neq 0$. 

The above calculation also shows that in the remaining cases, $d_0=d_1+d_2=0$. Moreover,
note that each $Y_i$ is a torus with $Y_i^2=0$. In particular, $w_1=w_2=0$.

To deal with the remaining possibilities, we use the Mod $p$ vanishing theorem of 
Seiberg-Witten invariants (cf. \cite{N}). We shall first compute with the $G$-signature theorem 
(not the weak version). First, recall that $\chi(M/G)=4$, so that $b_2^{-}(M/G)=1$ is true. 
It follows that in Lemma 2.7,  the angles  $\theta_1\neq \pm \theta_2$. Without loss of generality, 
we assume 
$\theta_1=\frac{2\pi}{5}$ and $\theta_2=\frac{4\pi}{5}$ in Lemma 2.7. With this we have 
$$
Sign(g,M)=2(\cos \frac{6\pi}{5}-\cos\frac{-2\pi}{5})=-2(\cos \frac{\pi}{5}+\cos\frac{2\pi}{5}).
$$
On the other hand, we observe that the same division of fixed points or components into group I 
or group II works here too. With this understood, noting that $w_1=w_2=0$, it follows easily from
the $G$-signature theorem that 
$$
Sign(g,M)=\sum_{k=1}^2(y_k-x_k)\cot^2\frac{k\pi}{5}+(z_2-z_1)\cot\frac{\pi}{5}\cot\frac{2\pi}{5}.
$$
Next we observe that $Sign(g^2,M)=2(\cos \frac{12\pi}{5}-\cos\frac{-4\pi}{5})=-Sign(g,M)$,
and moreover, for $g^2$ the contributions to the Sign number for group I and 
group II switch values. Taking the difference $Sign(g,M)-Sign(g^2,M)$, and using the identities
(see Lemma 6.4 in \cite{CK2})
$$
\cot^2\frac{\pi}{5}-\cot^2\frac{2\pi}{5}=\csc^2\frac{\pi}{5}-\csc^2\frac{2\pi}{5}=4\cot\frac{\pi}{5}\cot\frac{2\pi}{5},
$$
we obtain 
$$
Sign(g,M)=(2(y_1-y_2+x_2-x_1)+(z_2-z_1))\cdot \cot\frac{\pi}{5}\cot\frac{2\pi}{5}.
$$
Now finally, observing the identity $5\cot\frac{\pi}{5}\cot\frac{2\pi}{5}
=2(\cos \frac{\pi}{5}+\cos\frac{2\pi}{5})=-Sign(g,M)$, we obtain the following constraint 
$$
2(y_1-y_2+x_2-x_1)+z_2-z_1=-5.
$$

With these preparations, we examine the remaining cases (3), (5) in more detail.
First consider case (3), where $x=y=2$, $z=1$. Observe that $y_1-y_2+x_2-x_1$ is always 
even. It follows easily that $z_2-z_1=-1$ and $y_1-x_1=-(y_2-x_2)=-1$ in this case. 
For case (5) where $x=y=1$, $z=3$, note that $y_1-y_2+x_2-x_1=\pm 2$. It follows that 
$z_2-z_1=-1$ and $y_1-x_1=-(y_2-x_2)=-1$ as well. 

Next we check this against the formula for the Spin number $Spin(g,M)$. To this end, we will use the following
identities:
$$
\csc\frac{\pi}{5}\csc\frac{2\pi}{5}=4\cot\frac{\pi}{5}\cot\frac{2\pi}{5},\;\;
\csc\frac{\pi}{5}\cot\frac{\pi}{5}+\csc\frac{2\pi}{5}\cot\frac{2\pi}{5}=6\cot\frac{\pi}{5}\cot\frac{2\pi}{5},
$$
which can be easily verified by direct calculation. Now with this understood, note that 
on the one hand, the definition of the Spin number gives 
$$
Spin(g,M)=\sum_{k=0}^4 d_k\mu^k=2d_1(\cos \frac{\pi}{5}+\cos\frac{2\pi}{5})=5d_1\cot\frac{\pi}{5}\cot\frac{2\pi}{5},
$$
and on the other hand, we have from the formula in Lemma 3.8 of \cite{CK2} that 
$$
Spin(g,M)=\frac{1}{4}(-\csc^2\frac{\pi}{5}+\csc^2\frac{2\pi}{5}+\csc\frac{\pi}{5}\csc\frac{2\pi}{5})=0.
$$
It follows immediately that in cases (3), (5), we have $d_1=0$, and as a result, $d_k=0$ for
all $k=0,1,\cdots,4$. 

With the preceding understood, recall that the condition in the Mod $p$ vanishing theorem of 
Seiberg-Witten invariants (cf. \cite{N}) is $2d_k<1-b_1^G+b_{+}^G$ for any $k=0,1,\cdots,4$, 
where $b_1^G=b_1(M/G)=0$ and $b_{+}^G=b_2^{+}(M/G)=1$ (note that since $b_1(M/G)=0$,
the fixed-point set $J^G$ in the Mod $p$ vanishing
theorem consists of a single point, i.e.,  $[0]$, so the integers 
$\{k^l_j\}$ in the theorem are given by $\{d_k\}$ for any $l$, and the integer $d(c)=0$). 
With $d_k=0$ for all $k$, the above condition in the Mod $p$ vanishing theorem is satisfied, 
so the Seiberg-Witten invariant for the canonical $Spin^c$ structure (which is induced by
a spin structure on $M$) vanishes (mod $5$). But by Taubes' theorem \cite{T}, the Seiberg-Witten invariant 
equals $1$, which is a contradiction. Hence cases (3), (5) are ruled out. This finishes the proof. 

\end{proof}

It remains to consider the case where $G$ is of non-prime order, $M_G$ is rational or ruled,
but for any prime order subgroup $H$, $M_H$ has torsion canonical class. Let $n$ be the order
of $G$. Then by Lemma 2.8, $n=2^k3^l$. We first note that $n\neq 6$. This is because if $n=6$,
then $G=\Z_2\times \Z_3$, and with the assumption that for any prime order subgroup $H$, $M_H$ has torsion canonical class, it follows easily that $M_G$ has torsion canonical class as well, 
which is a contradiction. Consequently, either $k>1$ or $l>1$ in $n=2^k3^l$. Finally, note that 
for any nontrivial element $g\in G$, the angles $\theta_1,\theta_2$ in Lemma 2.7 are both nonzero. 
In particular, $b_1(M/G)=0$, and $M_G$ must be rational. 

First of all, we have
\begin{lemma}
Suppose $G=\Z_4$ and for the order $2$ subgroup $H$, $M_H$ has torsion canonical class.
Then there are two possibilities:
\begin{itemize}
\item [{(i)}] $b_2^{+}(M/G)=3$, and the $G$-action has $4$ isolated fixed points, all of type $(1,3)$, 
and $12$ isolated points of isotropy of order $2$. 
\item [{(ii)}] $M_G$ is rational, and the $G$-action has $4$ isolated fixed points, all of type $(1,1)$, 
and $12$ isolated points of isotropy of order $2$. 
\end{itemize}
\end{lemma}

\begin{proof}
Fix a generator $g\in G$. It is easy to see that in Lemma 2.7, either $\theta_1=-\theta_2$
or $\theta_1=\theta_2$. So either $b_2^{+}(M/G)=3$, $b_2^{-}(M/G)=1$, or 
$b_2^{+}(M/G)=1$, $b_2^{-}(M/G)=3$. In any case, we have $\chi(M/G)=6$. Finally, observe that
$L(g,M)=4$ in both cases.

On the other hand, by examining the action of $G$ on $M^H$, which consists of $16$ isolated 
points, and with $L(g,M)=4$, it follows easily that $M/G$ has $10$ isolated singularities.
With $\chi(M/G)=6$, it follows that $\chi(M_G)>12$, so that if $M_G$ has torsion canonical class,
then $b_2^{+}(M/G)=3$ must be true. Case (i) follows immediately. 

Suppose $M_G$ is rational, and let $x,y$ be the number of fixed points of type $(1,1)$ and
$(1,3)$ respectively. Then note that each type $(1,1)$ fixed point contributes a $(-4)$-sphere in
$M_G$, which in turn contributes $-1$ to $c_1(K_{M_G})^2$. The other singular points of $M/G$
contribute zero, hence $c_1(K_{M_G})^2=-x$. On the other hand, note that $\chi(M_G)=
\chi(M/G)+x+3y+6=12+x+3y$. As $M_G$ is rational, $c_1(K_{M_G})^2=12-\chi(M_G)$,
which implies $y=0$. Hence $x=4$, and case (ii) follows. This finishes the proof.

\end{proof}

Now finally, we have

\begin{lemma}
Suppose $M_G$ is rational, but for any prime order subgroup $H$, $M_H$ has torsion canonical class. Then the order $n$ of $G$ must either $4$ or $8$. Moreover, if $n=8$, then the $G$-action 
falls into one of the following two cases: 
\begin{itemize}
\item [{(i)}] the $G$-action has $2$ isolated fixed points, all of type $(1,3)$, $2$ isolated points of isotropy of order $4$ of type $(1,3)$, 
and $12$ isolated points of isotropy of order $2$;
\item [{(ii)}] the $G$-action has $2$ isolated fixed points, all of type $(1,5)$, $2$ isolated points of isotropy of order $4$ of type $(1,1)$, 
and $12$ isolated points of isotropy of order $2$.
\end{itemize}
\end{lemma}

\begin{proof}
It is easy to check that if $G$ contains an element $g$ of order $9$, $12$, or $16$,
then for the angles $\theta_1,\theta_2$ of $g$ in Lemma 2.7, the integrability conditions in 
Lemma 2.7(1) are violated. It follows easily that $k\leq 3$ and $l=0$ in
$n=2^k3^l$, i.e., $n=4$ or $8$. 

With the preceding understood, suppose $n=8$. We fix a generator $g$ such that 
$\theta_1=\frac{2\pi}{8}$, $\theta_2=\frac{2\pi q}{8}$ in Lemma 2.7, where $q$ is odd and 
$0<q<8$. We note that $q\neq 1$ or $7$, for otherwise, the integrability conditions in 
Lemma 2.7(1) are violated. 
On the other hand, let $H$ be the subgroup of order $4$ generated by $g^2$. Then 
by Lemma 2.11, there are two cases, (i) and (ii), as listed therein. 

Suppose we are in case (i) of Lemma 2.11 where $M_H$ has torsion canonical class.
In this case, $b_2^{+}(M/H)=3$, which easily implies that $q=3$ in $\theta_2$. As a corollary,
$L(g,M)=2$, and $b_2^{+}(M/G)=b_2^{-}(M/G)=1$, so that $\chi(M/G)=4$.
Examining the action of $g$ on $M^H$, with $L(g,M)=2$, it follows easily that $M/G$ has $6$ 
isolated singular points, where two of them have isotropy 
of order $8$, one of isotropy of order $4$, and three of isotropy of order $2$. Now we determine the action of $g$ at the two fixed points. We note that the minimal resolution of a singular point of
order $8$ of type $(1,3)$ in $M_{G}$ is a pair of $(-3)$-spheres intersecting transversely and 
positively at one point. Its contribution to $c_1(K_{M_{G}})^2$ is easily seen to be $-1$. 
All other types of singular points of $M/G$ are Du Val singularities, so make zero contribution. 
On the other hand, the 
minimal resolution of a singular point of order $8$ of type $(1,7)$ in $M_{G}$ is a linear chain 
of seven $(-2)$-spheres, so its contribution to $\chi(M_{G})$ is $7$. 
With $c_1(K_{M_{G}})^2=12-\chi(M_{G})$, it follows easily that there cannot be any fixed point of $g$ of type $(1,7)$. This finishes the discussion on case (i).

The analysis for case (ii) of Lemma 2.11, where $M_H$ is rational, is completely analogous, hence
omitted. This finishes the proof. 

\end{proof}

\section{Symplectic surfaces in a rational $4$-manifold}

Let $(X,\omega)$ be a symplectic rational $4$-manifold where $X=\C\P^2\# N\overline{\C\P^2}$.
We shall denote the canonical line bundle of $(X,\omega)$ by $K_\omega$ to indicate the dependence on $\omega$. We also use $K_X$ when the dependence on $\omega$ needs not to be emphasized. 

We begin with the definition of reduced bases of $(X,\omega)$.
To this end, let $\E_X$ be the set of classes in $H^2(X)$ which can be represented by a smooth 
$(-1)$-sphere, and let $\E_\omega:=\{E\in \E_X| c_1(K_\omega)\cdot E= -1\}$. Then each class in 
$\E_\omega$ can be represented by a symplectic $(-1)$-sphere (cf. \cite{LiLiu0}); 
in particular, $\omega(E)>0$ for any $E\in\E_\omega$. 

\begin{definition}
A basis $H, E_1,\cdots, E_N$ of $H^2(X)$ is called a {\bf reduced basis} of 
$(X,\omega)$ if the following are true:
\begin{itemize}
\item it has a standard intersection form, i.e., $H^2=1$, $E^2_i=-1$ and $H\cdot E_i=0$ for any $i$, and $E_i\cdot E_j=0$ for any $i\neq j$; 
\item $E_i\in\E_{\omega}$ for each $i$, and moreover, if $N\geq 3$, the following area conditions are satisfied: 
$\omega(E_N)=\min_{E\in\E_{\omega}}\omega(E)$, and for any $2<i<N$, 
$\omega(E_i)=\min_{E\in\E_i}\omega(E)$, where
$\E_i:=\{E\in\E_{\omega}| E\cdot E_j=0 \;\; \forall j>i\}$ for any $i<N$; 
\item $c_1(K_{\omega})=-3H+E_1\cdots+E_N$.
\end{itemize}
\end{definition}

Without loss of generality, we assume $\omega(E_1)\geq \omega(E_2)$. Then 
the following constraints on the symplectic areas are straightforward from Definition 3.1. 
\begin{itemize}
\item $\omega(H)>0$, and $\omega(E_i)\geq \omega(E_j)$ for any $i<j$;
\item for any $i\neq j$, $H-E_i-E_j\in\E_{\omega}$, so that $\omega(H-E_i-E_j)>0$; 
\item $\omega(H-E_i-E_j-E_k)\geq 0$ for any distinct $i,j,k$.
\end{itemize}
Reduced bases always exist, see \cite{LW} for more details. We remark that a reduced basis is not necessarily unique, however, the symplectic areas of its classes 
$$(\omega(H),\omega(E_1),\cdots, \omega(E_N))$$ 
uniquely determine the symplectic structure $\omega$ up to symplectomorphisms, cf. \cite{KK}. 

Secondly, we recall the following technical result concerning reduced bases, which will be used in
Section 5. 

\begin{lemma}
{\em(}cf. \cite{KK}{\em)}
Let $N\geq 2$. Then for {\bf any} $\omega$-compatible almost complex structure $J$, any class $E\in\E_{\omega}$ of minimal symplectic area can be represented by an embedded $J$-holomorphic sphere. In particular, for $N\geq 3$, the class $E_N$ in a reduced basis 
$H, E_1,\cdots, E_N$ can be represented by a $J$-holomorphic $(-1)$-sphere for any $J$.
\end{lemma}

With the preceding understood, we fix a reduced basis $H, E_1,\cdots, E_N$ of $(X,\omega)$. 
Then for any $A\in H^2(X)$, we can write
$$
A=aH-\sum_{i=1}^N b_i E_i, \mbox{ where $a,b_i\in\Z$.}
$$ 
We first derive some general constraints on the coefficients $a$ and $b_i$ when $A$ is represented by a connected, embedded symplectic surface, particularly, when $A$ is the class of a symplectic 
$(-\alpha)$-sphere for $\alpha>1$. These constraints are consequences of the 
fundamental work of Li-Liu \cite{LiLiu} and Li-Li \cite{LL} on symplectic rational $4$-manifolds. 

First of all, a few useful facts. For a generic $\omega$-compatible almost complex structure 
$J$, the class $H$ and any class $E\in \E_\omega$ can be represented by a $J$-holomorphic sphere (cf. \cite{LiLiu0}). In particular, this implies that for any $E\in \E_\omega$, where
$E\neq E_i$, $1\leq i\leq N$, the coefficients in $E=aH-\sum_{i=1}^N b_i E_i$ satisfy $a>0$, 
$b_i\geq 0$ for all $i$ by the positivity of intersection of $J$-holomorphic curves. Similarly, if 
$A=aH-\sum_{i=1}^N b_i E_i$ is the class of a connected, embedded symplectic surface 
with $A^2\geq 0$, then by choosing an $\omega$-compatible almost complex structure 
$J$ such that the symplectic surface is $J$-holomorphic, we see easily that $a>0$ and 
$b_i\geq 0$ for all $i$. 

The situation is more subtle when $A^2<0$ and $A$ is not a class in $\E_\omega$. We begin with the following lemma.

\begin{lemma}
Suppose $A=aH-\sum_{i=1}^N b_i E_i$ is the class of a connected, embedded symplectic 
surface of genus $g$. 
\begin{itemize}
\item [{(1)}] If $a>0$, then $b_i\geq 0$ for all $i$.
\item [{(2)}] The $a$-coefficient of $A$ satisfies the following inequality: $(a-1)(a-2)\geq 2g$,
with ``$=$" if and only if $b_i=0$ or $1$ for all $i$. 
\end{itemize}
\end{lemma}

\begin{proof}
For part (1), we begin by noting that the
genus $g$ of the symplectic surface representing $A$ is given by the adjunction formula
$$
g=\frac{1}{2}(A^2+c_1(K_\omega)\cdot A)+1.
$$
Suppose to the contrary that $a>0$ but $b_k<0$ for some $k$. Then we consider the reflection 
$R(E_k)$ on $H^2(X)$ defined by the class $E_k$, where 
$$
R(E_k)\beta =\beta+2 (\beta\cdot E_k)E_k, \;\; \forall \beta\in H^2(X).
$$
If we let $\tilde{A}$ be the image of $A$ under $R(E_k)$ and write $\tilde{A}=\tilde{a} H-\sum_{i=1}^N \tilde{b}_i E_i$,
then $\tilde{a}=a$, $\tilde{b}_k=-b_k>0$, and $\tilde{b}_i=b_i$ for all $i\neq k$. It follows easily that $\tilde{A}^2=A^2$
and $c_1(K_\omega)\cdot \tilde{A}-c_1(K_\omega)\cdot A=2\tilde{b}_k>0$. Finally, since $R(E_k)$ is induced by an
orientation-preserving diffeomorphism of $X$ (cf. \cite{LL}), the class $\tilde{A}$ is represented by a smoothly embedded,
connected surface of genus $g$.

Now the condition $a>0$ enters the argument. Pick a sufficiently small $\epsilon>0$, and let $e:=H-\sum_{i=1}^N \epsilon E_i\in
H^2(X,\R)$. Then $a>0$ implies that $e\cdot \tilde{A}=a-\sum_{i=1}^N \epsilon \tilde{b}_i>0$ for sufficiently small $\epsilon>0$.
On the other hand, we claim that $e$ lies in the symplectic cone associated to the symplectic canonical class $c_1(K_\omega)$.
To see this, we only need to verify that (i) $e^2=1-N\epsilon^2>0$, which is obviously true 
when $\epsilon>0$ is sufficiently small, and (ii)
$e\cdot E>0$ for any class $E\in \E_\omega$ (cf. \cite{LiLiu}). To see (ii) is true, we write 
$E=uH-\sum_{i=1}^N v_i E_i$. Then $u^2=\sum_i v_i^2-1$ and $u\geq 0$, and 
$e\cdot E=u-\epsilon \sum_i v_i$. If $E=E_l$ for some $l$, then $e\cdot E=\epsilon>0$.
If $u>0$, then $e\cdot E=\sqrt{\sum_i v^2-1}-\epsilon \sum_i v_i>0$ when $\epsilon>0$ is sufficiently small. Hence the claim that
$e$ lies in the symplectic cone associated to the symplectic canonical class $c_1(K_\omega)$.

Now the fact that $e\cdot \tilde{A}>0$ together with the fact that $e$ lies in the symplectic cone associated to the symplectic 
canonical class $c_1(K_\omega)$ imply the following inequality on the symplectic genus $\eta(\tilde{A})$ of $\tilde{A}$ 
(cf. \cite{LL}, Definition 3.1, p. 130):
$$
\eta(\tilde{A})\geq \frac{1}{2}(\tilde{A}^2+c_1(K_\omega)\cdot \tilde{A})+1.
$$
On the other hand, the minimal genus is bounded from below by the symplectic genus (cf. \cite{LL}, Lemma 3.2). Thus 
$g\geq \eta(\tilde{A})$, which implies that 
$c_1(K_\omega)\cdot A\geq c_1(K_\omega)\cdot \tilde{A}$, a contradiction.
This finishes off part (1) of the lemma.

For part (2), the adjunction formula $A^2+c_1(K_\omega)\cdot A+2=2g$ gives
$$
a^2-\sum_{i=1}^N b_i^2-3a+\sum_{i=1}^N b_i+2=2g.
$$
With $\sum_{i=1}^N b_i^2-\sum_{i=1}^N b_i=\sum_{i=1}^N b_i(b_i-1)\geq 0$, we obtain
easily $(a-1)(a-2)\geq 2g$, with ``$=$" if and only if $b_i=0$ or $1$ for all $i$. 
This finishes off part (2), and the proof of the lemma is complete.

\end{proof}

The following lemma deals with the case where the $a$-coefficient of $A$ is negative.

\begin{lemma}
Let $A=aH-\sum_{i=1}^N b_i E_i$ be the class of a connected, embedded symplectic surface 
of genus $g$ such that $a<0$. Then 
\begin{itemize}
\item [{(1)}] the symplectic surface representing $A$ must be a symplectic $(-\alpha)$-sphere 
where $\alpha>2$, i.e, $g=0$ and $A^2<-2$, and 
\item [{(2)}] the expression $A=aH-\sum_{i=1}^N b_i E_i$ must be in the following form:
$$
A=aH+(|a|+1) E_{j_1}-E_{j_2}-\cdots-E_{j_s}, \mbox{  where $s=\alpha-2|a|$, }
$$
in particular, $2|a|<\alpha$. Moreover, $E_{j_1}=E_1$ and $\omega(E_1)>\omega(E_i)$ 
for any $i>1$. 
\end{itemize}
\end{lemma}

\begin{proof}
Let $b_i^{-}=\max (0,-b_i)$ and $b_i^{+}=\max (0,b_i)$, and consider the class
$$
\tilde{A}=|a| H-\sum_{i=1}^N (b_i^{-}+ b_i^{+}) E_i. 
$$
Since $b_i^{-}=|b_i|$ when $b_i<0$ and equals $0$ otherwise, and $b_i^{+}=b_i$ when $b_i>0$ and equals $0$ otherwise,
it follows easily that $\tilde{A}$ is the image of $-A$ under the action of the composition of the reflections $R(E_k)$, where
$k$ is running over the set of indices such that $b_k>0$. In particular, $\tilde{A}$ is represented by a smoothly embedded surface of genus $g$. As in the proof of the previous lemma,
$e:=H-\sum_{i=1}^N \epsilon E_i$ lies in the symplectic cone associated to the symplectic canonical class $c_1(K_\omega)$ 
when $\epsilon>0$ is sufficiently small. Furthermore, as $a\neq 0$, we have $e\cdot \tilde{A}>0$, 
so that 
$$
g\geq \eta(\tilde{A})\geq \frac{1}{2}(\tilde{A}^2+c_1(K_\omega)\cdot \tilde{A})+1,
$$
where $\eta(\tilde{A})$ denotes the symplectic genus of $\tilde{A}$ (cf. \cite{LL}). The above inequality is equivalent to
$$
-3|a|+ \sum_{i=1}^N  (b_i^{-}+ b_i^{+}) \leq -A^2+2g-2.
$$
On the other hand, the adjunction formula for $A$ gives the equation 
$-3a+\sum_{i=1}^N b_i=-A^2+2g-2$, which implies easily, when combined with the above 
inequality, that $\sum_{i=1}^N b_i^{+}\leq -A^2+2g-2$. It follows that $\sum_{i=1}^N b_i^{-}\leq 3|a|$.

Note that the adjunction formula $A^2+c_1(K_\omega)\cdot A+2=2g$ also implies easily that 
$$
2g+\sum_{i=1}^N b_i(b_i-1)=a^2-3a+2=(a-1)(a-2)=(|a|+1)(|a|+2).
$$
(The last equality is due to the assumption that $a<0$.) It follows that $b_i^{-}\leq |a|+1$ 
for each $i$, and moreover, if $b_i^{-}=|a|+1$ for some $i$, then $g=0$, and for any $j\neq i$,
$b_j=0$ or $1$. With this understood, we shall next exclude the possibility that $b_i^{-}\leq |a|$ 
for any $i$, using the constraints of symplectic areas for a reduced basis. 

Suppose to the contrary that $b_i^{-}\leq |a|$ for all $i$. Then we will write $A$ as follows:
$$
A=-(|a| H-\sum_{i=1}^N b_i^{-} E_i) -\sum_{i=1}^N b_i^{+} E_i.
$$
Since $b_i^{-}\leq |a|$ for all $i$ and $\sum_{i=1}^N b_i^{-}\leq 3|a|$, the class $|a| H-\sum_{i=1}^N b_i^{-} E_i$ 
can be written as a sum of classes of the form $H$, $H-E_i$, $H-E_i-E_j$, or $H-E_i-E_j-E_k$, where distinct indices 
stand for distinct classes. Since all these classes have non-negative symplectic areas, it follows that 
$\omega(A)\leq 0$, which is a contradiction. Hence 
$$
A=aH+(|a|+1) E_{j_1}-E_{j_2}-\cdots-E_{j_s}, \mbox{ where } s=-A^2-2|a|.
$$
In particular, $A^2=-2|a|-s<-2$, and moreover, $2|a|<-A^2$.  Finally, if there is a class $E_i$ 
such that $\omega(E_{j_1})\leq \omega(E_i)$, then
$$
\omega(aH+(|a|+1) E_{j_1})\leq -(|a|-1)\omega(H-E_{j_1})-\omega(H-E_{j_1}-E_i)<0,
$$
which implies that $\omega(A)<0$. It follows easily that $E_{j_1}=E_1$, and 
$\omega(E_1)>\omega(E_i)$ for any $i>1$. This finishes the proof. 

\end{proof}

In the rest of this section, we shall be focusing on the possible homological expressions
of a symplectic $(-\alpha)$-sphere, in particular, for $\alpha=2$ and $3$.
The constraints in Lemmas 3.3 and 3.4 allow us to easily determine all the possible expressions 
of the class $A$ of a symplectic $(-\alpha)$-sphere in terms of the reduced 
basis $H, E_1,\cdots, E_N$ when the $a$-coefficient of $A$ is relatively small, say $a\leq 3$. 

To this end, write $A=aH-\sum_{i=1}^N b_i E_i$, and observe that in the following equation 
$$
\sum_{i=1}^N b_i(b_i-1)=a^2-3a+2=(a-1)(a-2)
$$
which is satisfied by the coefficients $a, b_i$ of $A$, the left-hand side is always a nonnegative, even integer. In particular, when $a=1$ or $2$, $b_i$ must be either $0$ or $1$. For $a=0$, the area condition $\omega(A)>0$ implies that exactly one of the $b_i$'s equals $-1$ and the rest are either $0$ or $1$. For $a=3$, exactly one of the $b_i$'s equals either $2$ or $-1$, however, the latter 
possibility is ruled out by Lemma 3.3. The rest of the $b_i$'s are either $0$ or $1$. We summarize
the discussions in the following 

\vspace{2mm}

\noindent{\bf Observation: }{\it Let $A=aH-\sum_{i=1}^N b_i E_i$ be the class of a symplectic 
$(-\alpha)$-sphere where $a\leq 3$. Then $A$ must take the following expression
$$
A=aH-(a-1)E_{j_1}-E_{j_2}-\cdots-E_{j_{2a+\alpha}}.
$$}

If $a>3$ but is small, the possibilities for the values of $b_i$ can be easily determined.
However, when $a$ is large, though there are only finitely many solutions for the $b_i$'s for a fixed
value of $a$, it is in general impossible to determine all the possible solutions for the $b_i$'s.
Finally, note that there is no a priori upper bound for the $a$-coefficient in terms of $N$ and $\alpha$.

With this understood, the following technical lemma plays a key role in determining the expression 
of $A$ when the $a$-coefficient is large, for the case where $\alpha=2$ or $3$. 

\begin{lemma}
Let $A=aH-\sum_{i=1}^N b_i E_i$ be any class which satisfies 
$$
A^2=-\alpha, \; c_1(K_\omega)\cdot A=\alpha-2,  \mbox{ where } \alpha=2, 3.
$$
If $a>3$, then there are at least $\alpha+7$ terms in $A$ with non-zero $b_i$-coefficient.
\end{lemma}

\begin{proof}
We begin by recalling a reduction procedure useful in this kind of problems. For any distinct indices $i,j,k$, we set 
$H_{ijk}:=H-E_i-E_j-E_k$. Then $H_{ijk}$ satisfies the following conditions:
$$
H_{ijk}^2=-2, \; c_1(K_\omega)\cdot H_{ijk}=0, \mbox{ and } \omega(H_{ijk})\geq 0.
$$
Furthermore, there is a reflection $R_{ijk}$ on $H^2(M)$ associated to $H_{ijk}$, which is defined by the following formula:
$$
R_{ijk}(A):=A+ (A\cdot H_{ijk}) H_{ijk}, \; \forall A\in H^2(X). 
$$
To ease the notation, let $\tilde{A}:=R_{ijk}(A)$. Then it is easy to see that 
$$
\tilde{A}^2=A^2, c_1(K_\omega)\cdot \tilde{A}=c_1(K_\omega)\cdot A, \mbox{ and } \tilde{A}\cdot H_{ijk}=-A\cdot H_{ijk}.
$$ 
The last equality implies that 
$$
A=R_{ijk}(\tilde{A})=\tilde{A}+(\tilde{A}\cdot H_{ijk}) H_{ijk}.
$$
Finally, note that the operation $R_{ijk}$ will decrease (resp. increase) the $a$-coefficient in the expression of $A$ 
if and only if $A\cdot H_{ijk}<0$ (resp. $A\cdot H_{ijk}>0$), where $A\cdot H_{ijk}=a-(b_i+b_j+b_k)$.
See \cite{LW} or \cite{BP} for further discussions on this reduction procedure.

With the preceding understood, let $A=aH-\sum_{i=1}^N b_i E_i$ be any class satisfying the conditions in the lemma,
i.e., $A^2=-\alpha$, $c_1(K_\omega)\cdot A=\alpha-2$,  where $\alpha=2, 3$, and assume $a>3$.
Suppose to the contrary that $A$ has no more than $\alpha+6$ terms in the expression with non-zero 
$b_i$-coefficient. 

\vspace{2mm}

{\bf Claim:} {\it There are distinct indices $i,j,k$ such that (i) $b_i,b_j,b_k$ are positive,  and (ii) $A\cdot H_{ijk}=a-(b_i+b_j+b_k)<0$.
}

\vspace{2mm}

{\bf Proof of Claim:} 
We shall prove by contradiction. But first, we observe that there are at least $3$ terms in 
$A$ with the $b_i$-coefficient positive. To see this, note that the conditions $A^2=-\alpha$, 
$c_1(K_\omega)\cdot A=\alpha-2$ easily imply that
$$
\sum_{i=1}^N b_i(b_i-1)=(a-1)(a-2).
$$
Since $a>3$, it follows that for any $i$, if $b_i>0$, then $b_i\leq a-1$ must be true. Therefore, 
if there were at most $2$ terms in $A$ with the $b_i$-coefficient positive, then 
$\sum_{i=1}^N b_i\leq 2(a-1)$, which contradicts 
$-3a+\sum_{i=1}^N b_i=c_1(K_\omega)\cdot A=\alpha-2$.

With the preceding understood, suppose the claim is not true. Then it follows that 
$b_i+b_j+b_k\leq a$ holds true for any distinct indices $i,j,k$, where $b_i,b_j,b_k$ 
are not necessarily positive or non-zero.
Consider first the case where $\alpha=2$. Pick a $b_i$-coefficient, say $b_s$, such that 
$b_s>0$. Then we have 
$$
\sum_{i=1}^N b_i=\sum_{i=1}^N b_i+b_s-b_s\leq 3a-b_s\leq 3a-1,
$$
which is a contradiction to $-3a+\sum_{i=1}^N b_i=\alpha-2=0$. 
A similar argument also confirms the claim for $\alpha=3$. 
This finishes off the proof of the claim. 

\vspace{1.5mm}

Now going back to the proof of the lemma, we pick the indices $i,j,k$ given by the claim above, 
and perform the operation $R_{ijk}$ to reduce $A$ to $\tilde{A}:=R_{ijk}(A)$, which continues to obey the conditions on $A$, i.e., 
$$
\tilde{A}^2=-\alpha \mbox { and } c_1(K_\omega)\cdot \tilde{A}=\alpha-2.
$$
Set $c:=b_i+b_j+b_k-a$. We shall derive an upper bound on $c$. To this end, note that 
$$
b_i(b_i-1)+b_j(b_j-1)+b_k(b_k-1)\leq (a-1)(a-2).
$$
Using the inequality $3(b_i^2+b_j^2+b_k^2)\geq (b_i+b_j+b_k)^2$, we obtain
$$
\frac{b_i+b_j+b_k}{3}(\frac{b_i+b_j+b_k}{3}-1)\leq \frac{b_i^2+b_j^2+b_k^2}{3}-\frac{b_i+b_j+b_k}{3}\leq \frac{1}{3}(a-1)(a-2).
$$
Since $a>3$, this gives $\frac{b_i+b_j+b_k}{3}-1\leq \frac{1}{\sqrt{3}}(a-2)$, and consequently, 
$c\leq \sqrt{3}(a-2)+3-a$. It follows that the $a$-coefficient of $\tilde{A}$, denoted by 
$\tilde{a}$, will be at least $2$, because 
$$
\tilde{a}=a-c\geq (2-\sqrt{3})a+2\sqrt{3}-3\geq (2-\sqrt{3})\times 4+2\sqrt{3}-3=5-2\sqrt{3}>1.
$$
Finally,  because $b_i,b_j,b_k$ are non-zero, this operation does not introduce any new terms with non-zero $b_i$-coefficient, so $\tilde{A}$ continues to have no more than $\alpha+6$ terms 
in its expression with non-zero $b_i$-coefficient.

After finitely many steps, we will arrive at a class, continuing to be denoted by $\tilde{A}$, 
whose $a$-coefficient lies in the range $2\leq \tilde{a}\leq 3$. We may assume $\tilde{A}$ is the first class whose $a$-coefficient
lies in this range; in particular, the $a$-coefficient of the previous class, denoted by $A$, obeys $a>3$. We shall examine $\tilde{A}$ according to the value of $\tilde{a}$ below. 
To this end, we denote by $\tilde{b}_i$ the $b_i$-coefficients of $\tilde{A}$. Then it is helpful to observe that $\tilde{b}_i+\tilde{b}_j+\tilde{b}_k-\tilde{a}=-c<0$, because of the relation 
$\tilde{A}\cdot H_{ijk}=-A\cdot H_{ijk}$. 

Suppose $\tilde{a}=2$. Then $\tilde{A}=2H-E_{j_1}-\cdots-E_{j_{\alpha+4}}$. The condition $a>3$ requires that in this case
we must have $c\geq 2$, and consequently, $\tilde{b}_i+\tilde{b}_j+\tilde{b}_k-\tilde{a}=-c\leq -2$.
Since the $b_i$-coefficients of $\tilde{A}$ are non-negative and $\tilde{a}=2$, it follows that 
$\tilde{b}_i=\tilde{b}_j=\tilde{b}_k=0$ and $c=\tilde{a}=2$ must be true. In particular,
the indices $i,j,k$ are not appearing in the expression of $\tilde{A}$, and it follows that $A$ takes 
the form
$$
A=4H-2E_i-2E_j-2E_k-E_{j_1}-\cdots-E_{j_{\alpha+4}}, \mbox{ where $j_s\neq i,j,k$,} 
$$
which has $\alpha+7$ terms with non-zero $b_i$-coefficient, contradicting the assumption. 

Suppose $\tilde{a}=3$. Then the expressions for $\tilde{A}$ are 
$$
\tilde{A}=3H-2E_{j_1}-\cdots -E_{j_{\alpha+6}} \mbox{ or } \tilde{A}=3H-E_{j_1}-\cdots 
-E_{j_{\alpha+8}} +E_{j_{\alpha+9}}.
$$
The latter case is ruled out immediately as $\tilde{A}$ has $\alpha+9$ many terms with non-zero 
$b_i$-coefficient. For the former case, we note that with $\tilde{a}=3$, $c\geq 1$, 
$\tilde{b}_i+\tilde{b}_j+\tilde{b}_k\leq 3-1=2$. It follows easily that the following are the only 
possibilities for $\tilde{b}_i, \tilde{b}_j, \tilde{b}_k$:
$$
(\tilde{b}_i, \tilde{b}_j, \tilde{b}_k)=(2,0,0), (1,1,0), (1,0,0), (0,0,0).
$$
With this understood, note that $b_l=\tilde{b}_l+c$ for $l=i,j,k$. Since at least one of 
$\tilde{b}_i, \tilde{b}_j, \tilde{b}_k$ is zero, it follows that the number of
terms in the expression of $A$ with non-zero $b_i$-coefficient is at least $1$ more than the 
number of terms with non-zero $b_i$-coefficient in $\tilde{A}$. Now $\tilde{A}$ has $\alpha+6$ 
many terms of non-zero $b_i$-coefficient, so $A$ must have at least $\alpha+7$ many terms, 
which is a contradiction. 
This completes the proof of the lemma. 

\end{proof}

With the preceding understood, we now state a lemma which is of fundamental importance
for our project on symplectic Calabi-Yau $4$-manifolds. The key observation is that, when combined with Lemma 3.5, the area condition $\omega(A)<-c_1(K_\omega)\cdot [\omega]$ will 
give severe constraints on the $a$, $b_i$-coefficients of $A$; in particular, it implies an upper bound on the $a$-coefficient of $A$ in terms of $N$ for the case of $\alpha=2$ or $3$. 

\begin{lemma}
Let $A=aH-\sum_{i=1}^N b_i E_i$ be the class of a symplectic $(-\alpha)$-sphere where 
$\alpha=2$ or $3$, such that $\omega(A)<-c_1(K_\omega)\cdot [\omega]$. 
Then $A$ must be of the following form
$$
A=aH-(a-1)E_{j_1}-E_{j_2}-\cdots-E_{j_{2a+\alpha}}.
$$
In particular, $a\leq \frac{1}{2}(N-\alpha)$.
\end{lemma}

\begin{proof}
It suffices to only consider the situation where $a>3$. 
First, note that by Lemma 3.3, $b_i\geq 0$ for all $i$. 
If we let $b_i^{+}=\max(1, b_i)$, then $b_i^{+}=b_i$ 
when $b_i>0$ and $b_i^{+}=1$ when $b_i=0$. Next we observe that 
$-3a+\sum_{i=1}^{N}b_i=\alpha-2$. Let $n$ be the number of $b_i$'s which are non-zero. Then 
because $n\geq \alpha+7$ by Lemma 3.5, we have
$$
\sum_{i=1}^N  (b_i^{+}-1)=\sum_{i=1}^N b_i -n=3a+\alpha-2-n\leq 3(a-3).
$$

On the other hand, we claim that there must be one $b_i$ such that $b_i=a-1$. Suppose to the
contrary that this is not true. Then for each $i$, $b_i^{+}-1\leq a-3$ must be true. 
With this understood, note that the class $(a-3) H-\sum_i (b_i^{+}-1) E_i$ can be written as a sum of 
classes of the form $H$, $H-E_i$, $H-E_i-E_j$, or $H-E_i-E_j-E_k$, where distinct indices 
stand for distinct classes, because $\sum_{i=1}^{N} (b_i^{+}-1)\leq 3(a-3)$, and for each $i$, $b_i^{+}-1\leq a-3$. However, observe that we can write 
$$
A=-c_1(K_\omega) + (a-3) H-\sum_{i=1}^{N} (b_i^{+}-1) E_i 
+\sum_{i=1}^{N} \max (0,1-b_i) E_i, 
$$
from which it follows easily that $\omega(A)\geq -c_1(K_\omega)\cdot [\omega]$, contradicting the area assumption in the lemma. Hence the claim.

Now we observe that in the equation $\sum_{i=1}^N b_i(b_i-1)=(a-1)(a-2)$ which is satisfied by
the $a$, $b_i$-coefficients of $A$, if $b_i=a-1$ for some $i$, then the rest of the $b_i$'s 
are all equal to either $0$ or $1$. With this understood, the equation 
$-3a+\sum_{i=1}^N b_i=\alpha-2$ implies that the number of $b_i$'s equaling 
$1$ must be $2a+\alpha-1$. It follows immediately that $A$ must take the expression 
$$
A=aH-(a-1)E_{j_1}-E_{j_2}-\cdots-E_{j_{2a+\alpha}}.
$$
\end{proof}

We remark that if $A$ is the class of a symplectic $(-\alpha)$-sphere whose $a$-coefficient 
satisfies $a>3$ and there are at least $\alpha+7$ terms in the expression of $A$ having non-zero
$b_i$-coefficients, then the same proof shows that the condition 
$\omega(A)<-c_1(K_\omega)\cdot [\omega]$ would imply that $A$ also takes the special
expression in Lemma 3.6. However, in general it is not true that there are always 
at least $\alpha+7$ terms having non-zero $b_i$-coefficients in the expression of a 
symplectic $(-\alpha)$-sphere. For example, the following class, which has only $10$ terms with 
non-zero $b_i$-coefficients, can be represented by a symplectic $(-4)$-sphere (cf. \cite{DLW}):
$A=6H-2E_{j_1}-2E_{j_2}-\cdots-2E_{j_{10}}$.

\section{Non-existence of certain symplectic configurations}
In this section, we give several results concerning nonexistence of certain configurations 
of symplectic surfaces in rational $4$-manifolds. To prove these results, we examine the 
possible homological expressions of the components in the configurations in a certain reduced 
basis, using the constraints established in Section 3, and show that the configurations 
can not exist even at the homology level.
These nonexistence results will then be used in Section 5 to eliminate several possibilities of the fixed-point set 
structure obtained in Section 2 concerning the $2$-dimensional fixed components, which have 
resisted all the known obstructions available so far. 

First, we prove a lemma which allows us to impose certain auxiliary area conditions.

\begin{lemma}
Let $(X,\omega)$ be a symplectic $4$-manifold, and let $D=\sqcup_i D_i\subset X$, where each $D_i=\cup_j C_{ij}$ is a configuration 
of symplectic surfaces intersecting transversely and positively according to a negative definite plumbing graph $\Gamma_i$. Then for any
given collection of positive real numbers $\{a_{ij}\}$, there exists a $\delta_0>0$, such that for any choice of $\{\delta_i\}$ where 
$0<\delta_i<\delta_0$, there is a symplectic $4$-manifold $(\tilde{X},\tilde{\omega})$ with 
$D\subset \tilde{X}$, which has the following significance:
\begin{itemize}
\item $D=\sqcup_i D_i$ is a set of symplectic configurations in $(\tilde{X},\tilde{\omega})$, and there is a diffeomorphism $\psi: \tilde{X}\rightarrow X$ which is identity on $D$, such that 
$\psi^\ast c_1(K_\omega)=c_1(K_{\tilde{\omega}})$,
\item the $\tilde{\omega}$-symplectic area of each surface $C_{ij}$ equals $\delta_i a_{ij}$,
i.e.,  $\tilde{\omega}(C_{ij})=\delta_i a_{ij}$. 
\end{itemize}
\end{lemma}

\begin{proof}
First of all, we may assume without loss of generality that the intersections of $C_{ij}$ are $\omega$-orthogonal, because we can always slightly perturb the symplectic surfaces to achieve this (cf. \cite{Gompf}). With this understood, since the plumbing graph $\Gamma_i$ is negative definite, 
each configuration $D_i$ has a regular neighborhood $U_i$ such that $L_i:=\partial U_i$ 
is a convex contact boundary (in the strong sense), cf. \cite{GS}. Furthermore, by a theorem of
Park and Stipsicz \cite{PS}, the 
contact structure on $L_i$ is the Milnor fillable contact structure (cf. \cite{CNP}). We denote by $\alpha_i$ the contact form on $L_i$, where $\omega=d\alpha_i$ on $L_i$. It is clear that we can arrange so that $\{U_i\}$ are disjoint in $X$. 

Now for any given collection of positive real numbers $\{a_{ij}\}$, let $({U}_i^\prime, \omega^\prime_i)$ be a convex regular neighborhood 
of $D_i=\cup_j C_{ij}$ constructed in \cite{GS} such that $\omega^\prime_i(C_{ij})=a_{ij}$. Fixing an identification 
$\partial {U}_i^\prime=L_i$, we let $\alpha_i^\prime$ denote the contact form on $L_i$ such that $\omega^\prime_i =d\alpha_i^\prime$ on $L_i$. Then by \cite{PS}, $\alpha_i^\prime=e^{f_i}\alpha_i$ for some smooth function $f_i$ on $L_i$. With this understood, we set
$\delta_0>0$ by the condition $\delta^{-1}_0:= \max_i \{\sup_{x\in L_i}e^{f_i(x)}\}$.

Given any $\{\delta_i\}$ where $0<\delta_i<\delta_0$, we set $C_i:=\log \delta_i$. Then it is easy to see that $C_i+f_i(x)<0$ for any 
$x\in L_i$. With this understood, we let 
$$W_i:=\{(x,t)\in L_i\times\R | C_i+f_i(x)\leq t\leq 0\},
$$ 
given with the symplectic structure $d(e^t\alpha_i)$. 
We define $(\tilde{U_i}, \tilde{\omega}_i)$ to be the symplectic $4$-manifold obtained by gluing 
$(U_i^\prime, \delta_i \omega_i^\prime)$ to $W_i$ via the contactomorphism sending 
$x\in L_i=\partial U_i^\prime$ to $(x, C_i+f_i(x))\in W_i$.
Note that each $(\tilde{U_i}, \tilde{\omega}_i)$ has a convex contact boundary $\partial \tilde{U}_i=L_i$
where $\tilde{\omega}_i=d\alpha_i$ on $L_i$. With this understood, we define $(\tilde{X}, \tilde{\omega})$ to be the symplectic $4$-manifold obtained by removing 
$\cup_i U_i$ from $X$ and then gluing back $\cup_i \tilde{U}_i$ along $\cup_i L_i$. It is easy
to see that there is a diffeomorphism $\psi: \tilde{X}\rightarrow X$ which is identity on $D$, such that 
$\psi^\ast c_1(K_\omega)=c_1(K_{\tilde{\omega}})$, and the $\tilde{\omega}$-symplectic area of each surface $C_{ij}$ equals $\delta_i a_{ij}$. This finishes the proof of the lemma. 

\end{proof}

The second lemma contains two useful observations. In particular, the first observation 
implies that in a configuration of symplectic surfaces there is at most one symplectic sphere with negative $a$-coefficient.

\begin{lemma}
$\left(1\right)$ Let $A_1,A_2$ be the classes of two symplectic spheres whose $a$-coefficients are negative.
Then $A_1\cdot A_2<0$.

$\left(2\right)$ Let $B=aH-\sum_{i=1}^N b_iE_i$ be a nonzero class satisfying 
$B^2=c_1(K_\omega)\cdot B=0$. 
If $a\geq 0$, then $a\geq 3$. Moreover, for $a=3$, the following are
the only possible expressions
for $B$:
$$
B=3H-E_{j_1}-\cdots-E_{j_9}.
$$
\end{lemma}

\begin{proof}
For (1), let $a_1,a_2$ be the $a$-coefficients of $A_1,A_2$ respectively, which are negative
by assumption. Then it follows easily from the expression in Lemma 3.4 that
$$
A_1\cdot A_2\leq a_1a_2-(|a_1|+1)(|a_2|+1)=-(|a_1|+|a_2|+1)<0.
$$

For (2), we first note that $B\neq 0$ and $B^2=0$ imply easily that $a\neq 0$ in $B$.
With this understood, we note that the conditions $B^2=c_1(K_\omega)\cdot B=0$ are equivalent to
$$
a^2-\sum_{i=1}^Nb_i^2=-3a+\sum_{i=1}^N b_i=0.
$$
It follows easily that $a(a-3)=\sum_{i=1}^N b_i(b_i-1)\geq 0$. With the assumption that $a\geq 0$,
it follows immediately that $a\geq 3$. Moreover, if $a=3$, each $b_i$ must be either 
$0$ or $1$, from which the expression of $B$ follows easily. This finishes the proof of the lemma.

\end{proof}

With these preparations, we now prove the aforementioned nonexistence results. 

\begin{proposition}
Let $\{B_i\}$ be a nonempty set of disjoint symplectic surfaces in $X=\C\P^2\# 10 \overline{\C\P^2}$,
where there is  at most one spherical component, and $F_1,F_2,F_3$ be a disjoint union of 
symplectic $(-3)$-spheres in the complement of $B_i$, such that 
$$
c_1(K_X)=-\frac{2}{3}\sum_i B_i-\frac{1}{3}(F_1+F_2+F_3).
$$ 
Suppose $F_{4,1}$, $F_{4,2}$ are a pair of symplectic $(-2)$-spheres in the complement of $B_i$ 
and $F_1,F_2,F_3$, such that $F_{4,1}$, $F_{4,2}$ intersect transversely and positively at one point. 
Then $\{B_i\}$ must consist of one component which is a torus. 
\end{proposition}

\begin{proof}
First of all, since $c_1(K_X)$ is represented by $F_1,F_2,F_3$ and $B_i$, which are disjoint from
the two $(-2)$-spheres $F_{4,1}$, $F_{4,2}$, it is clear that, by Lemma 4.1, we may assume without
loss of generality that the following area condition holds:
$$
\omega(F_{4,1})=\omega(F_{4,2})<-c_1(K_X)\cdot [\omega].
$$
Then by Lemma 3.6, the $a$-coefficients of $F_{4,1}$, $F_{4,2}$ lie in the range $0\leq a\leq 4$,
and moreover, their classes take the special form in Lemma 3.6. Furthermore, again by Lemma 4.1,
we can also arrange so that $F_1,F_2,F_3$ have the same area, which is sufficiently small,
so that $\omega(F_k)<\omega(B_i)$ for each $i,k$. 

With this understood, we next derive some basic information about $B_i$. First,
$c_1(K_X)=-\frac{2}{3}\sum_i B_i-\frac{1}{3}(F_1+F_2+F_3)$ implies that 
$c_1(K_X)^2=\frac{4}{9}\sum_i B_i^2-1$, and with $X=\C\P^2\# 10 \overline{\C\P^2}$, it follows 
easily that $\sum_i B_i^2=0$. On the other hand, if we denote by $g_i$ the genus of $B_i$, then
the adjunction formula applied to each $B_i$ gives us $-\frac{2}{3}B_i^2+B_i^2=2g_i-2$,
which is equivalent to $B_i^2=6(g_i-1)$ for each $i$. In particular, $B_i^2<0$ if and only if
$B_i$ is spherical, hence by our assumption, there is at most one component $B_i$ with
$B_i^2<0$, and such a component must be a $(-6)$-sphere. 

With the preceding understood, we observe that the proposition follows readily if 
there is no $B_i$ such that $B_i^2<0$. Under this condition, 
it is easy to see that each $B_i$ must be a torus. To see that 
there is only one component in $\{B_i\}$, we note that by Lemma 4.2(2), the $a$-coefficient 
of each $B_i$ is at least $3$. On the other hand, each $B_i$ contributes at least
$\frac{2}{3}\times 3=2$ to the $a$-coefficient of $-c_1(K_X)$, which equals $3$, while 
the total contribution from $F_1,F_2,F_3$ to the $a$-coefficient of $-c_1(K_X)$ is at least 
$\frac{1}{3}\times (-1)=-\frac{1}{3}$ by Lemmas 3.4(2) and 4.2(1). 
Hence the claim. Therefore, it boils down to show that there is no $B_i$ such that $B_i^2<0$.

Suppose to the contrary that there is a component, call it $B_1$, such that $B_1^2<0$. 
Since $b_2^{+}(X)=1$, there must be exactly one $B_i$, call it $B_2$, such that $B_2^2>0$, and
the rest of the $B_i$'s have $B_i^2=0$ hence are tori if there is any. Furthermore, as $B_1$ is a
$(-6)$-sphere, $B_2$ must be a genus-$2$ surface with $B_2^2=6$. By a similar argument 
analyzing the contributions of $B_i$ to the $a$-coefficient of $-c_1(K_X)$, using Lemmas 3.3 and 4.2,
it follows easily that $B_1,B_2$ are the only components in $\{B_i\}$. Finally, note that the sum of the 
$a$-coefficients of $F_1,F_2,F_3$ is at most $3$. 

\vspace{1.5mm}

{\bf Case (1):} Suppose $a=-2$ in $B_1$. Then by Lemma 3.4(2), we can write $B_1=-2H+3E_1-E_p$
for some $E_p$. We consider the possibilities for the classes of $F_1,F_2,F_3$. Note that by 
Lemma 4.2(1), $a\geq 0$ in  $F_1,F_2,F_3$. Consequently, $a\leq 3$ in $F_1,F_2,F_3$.
Suppose $a=3$ in one of them, say $F_1$. Then $B_1\cdot F_1=0$ easily implies that 
$$
F_1=3H-2E_1-E_{i_1}-\cdots-E_{i_8},
$$ 
where $E_p$ does not show up in $F_1$. But this is a contradiction:
$$
\omega(F_1-B_1)=\omega(5H-5E_1-E_{i_1}-\cdots-E_{i_8})+\omega(E_p)>0
$$
as the class $5H-5E_1-E_{i_1}-\cdots-E_{i_8}$ can be written as a sum of classes of the form
$H-E_i-E_j$ and $H-E_i-E_j-E_k$, which all have nonnegative areas. If $a=2$ in $F_1$, then 
one can check easily that $F_1\cdot B_1<0$ is always true. If $a=1$ in $F_1$, then $F_1$ must
take the form $F_1=H-E_1-E_p-E_q-E_r$ for some $E_q,E_r$. In particular, since $F_2,F_3$ are
disjoint from $F_1$, we must have $a\neq 1$ in $F_2,F_3$. It follows that both $F_2,F_3$ should have $a=0$. Since the sum of the $a$-coefficients of $F_1,F_2,F_3$ is always an odd number, it follows that the sum must equal $1$. Consequently, we must have 
$$
F_1=H-E_1-E_p-E_q-E_r,
$$
and both $F_2,F_3$ have zero $a$-coefficients. It follows that the sum of the $a$-coefficients
of $B_1,B_2$ equals $4$, so that $a=6$ in $B_2$. 

To proceed further, we write $B_2=6H-\sum_{i=1}^{10} b_i E_i$. Note that $B_2$ has genus $2$,
so that $c_1(K_X)\cdot B_2+B_2^2=2\times 2-2=2$. With $B_2^2=6$, this implies that
$$
-18+\sum_{i=1}^{10} b_i +6=2,\;\; 36-\sum_{i=1}^{10} b_i^2=6.
$$
Consequently, $\sum_{i=1}^{10} b_i(b_i-1)=16$, and as a result, note that $b_i\leq 4$ for each $i$. On the other hand,
$B_2\cdot B_1=0$, which gives $-12+3b_1-b_p=0$. Since $b_1\leq 4$, we must have $b_p=0$ and $b_1=4$. Then
$\sum_{i=2}^{10} b_i(b_i-1)=16-4\times 3=4$ implies that in $b_2,\cdots,b_{10}$, there are exactly two of them equaling $2$; the
rest are either $1$ or $0$. With $F_1\cdot B_2=0$, it follows easily that
$$
B_2=6H-4E_1-E_q-E_r-2E_{i_1}-2E_{i_2}-E_{i_3}-\cdots-E_{i_6}.
$$
With this understood, we note that 
$$
2(B_1+B_2)+F_1=9H-3E_1-3E_p-3E_q-3E_r-4E_{i_1}-4E_{i_2}-2E_{i_3}-\cdots-2E_{i_6}.
$$
This implies that without loss of generality, 
$$
F_2=E_{i_1}-E_{i_3}-E_{i_4}, \;\;\; F_3=E_{i_2}-E_{i_5}-E_{i_6}.
$$

With the preceding understood, let $A$ be the class of any of the $(-2)$-spheres $F_{4,1}$,
$F_{4,2}$. Then recall that because of the area condition we imposed at the beginning, the
$a$-coefficient of $A$ lies in the range $0\leq a\leq 4$, and its expression must be of the form
specified in Lemma 3.6. With this understood, if $a=4$ in $A$, then 
$$
A=4H-3E_{j_1}-E_{j_2}-\cdots-E_{j_{10}},
$$ 
containing all $10$ $E_i$-classes. It is easy to see that $A\cdot F_2\neq 0$, which rules out this possibility. If $a=3$ in $A$, then we can write $A=3H-2E_{j_1}-E_{j_2}-\cdots-E_{j_{8}}$. Then $B_1\cdot A=0$ implies that $E_{j_1}=E_1$ must be true, and $E_p$ is not contained in $A$. With this understood, $A\cdot F_2=A\cdot F_3=0$
implies that one of the $E_i$-classes in each pair $(E_{i_3},E_{i_4})$, $(E_{i_5},E_{i_6})$ can not appear in $A$. Together with $E_p$, there are $3$ $E_i$-classes not contained in $A$, which is a contradiction as there are only $10$ $E_i$-classes in total. 
If $a=2$ in $A$, then it is easy to see that $A\cdot B_1<0$. Hence we must have either
$a=1$ or $a=0$ in $A$. If $a=1$ in $A$, then $A\cdot B_1=0$ implies that $A$ contains both $E_1$ and $E_p$. But this leads to $A\cdot F_1<0$, which is a contradiction. This shows that $A=E_s-E_t$ for some $E_s,E_t$. It is easy to check that there are only $3$ possibilities:
$E_q-E_r$, $E_{i_3}-E_{i_4}$, and $E_{i_5}-E_{i_6}$. We just showed that the classes of 
$F_{4,1}$, $F_{4,2}$ must be from the three classes above. But they mutually intersect trivially with
each other, contradicting the fact that $F_{4,1}\cdot F_{4,2}=1$. Hence Case (1) is ruled out.

\vspace{1.5mm}

{\bf Case (2):} Suppose $a=-1$ in $B_1$. Then $B_1=-H+2E_1-E_x-E_y-E_z$ for some $E_x,E_y,E_z$. Again by Lemma 4.2(1), 
$a\geq 0$ in  $F_1,F_2,F_3$. If $a=3$ in $F_1$, then it is easy to see from $F_1\cdot B_1=0$, that $E_1$ must appear in $F_1$
with coefficient $-2$, and two of $E_x,E_y,E_z$ can not appear in $F_1$. But $F_1$ contains $9$ $E_i$-classes and there are
totally $10$ $E_i$-classes, which is a contradiction. If $a=2$ in $F_1$, then $F_1\cdot B_1=0$ implies that
$F_1=2H-E_1-E_{i_1}-\cdots-E_{i_6}$. But this gives a contradiction 
$$
\omega(F_1-B_1)=\omega(3H-3E_1-E_{i_1}-\cdots-E_{i_6})+\omega(E_x+E_y+E_z)>0,
$$
as the class $3H-3E_1-E_{i_1}-\cdots-E_{i_6}$ can be written as a sum of classes of the form
$H-E_i-E_j-E_k$, which all have nonnegative areas. Consequently, $a=1$ in $F_1$ and $a=0$ in $F_2,F_3$, where 
$$
F_1=H-E_1-E_x-E_u-E_v
$$ 
for some $E_u,E_v$. By the same argument as in Case (1), the sum of the $a$-coefficients of $B_1,B_2$ equals $4$, so that $a=5$ in $B_2$. 

Let $B_2=5H-\sum_{i=1}^{10} b_i E_i$. Then $c_1(K_X)\cdot B_2+B_2^2=2$ and $B_2^2=6$ imply that
$$
-15+\sum_{i=1}^{10} b_i+6=2,\;\; 25-\sum_{i=1}^{10} b_i^2=6.
$$
As we argued in Case (1), $B_2$ must have the following expression: 
$$
B_2=5H-3E_1-E_y-E_u-E_v-2E_{i_1}-E_{i_2}-E_{i_3}-E_{i_4}.
$$
After computing $2(B_1+B_2)+F_1$, we see that $E_y,E_{i_1}$ must be the $E_i$-classes in $F_2,F_3$ which has a $(+1)$-coefficient.
It follows then 
$$
F_2=E_y-E_z-E_{i_4},\;\;\; F_3=E_{i_1}-E_{i_2}-E_{i_3}
$$ 
without loss of generality.

With the preceding understood, let $A$ be the class of any of the $(-2)$-spheres $F_{4,1}$,
$F_{4,2}$. If $a=4$ in $A$, we have $A\cdot F_2\neq 0$ which is not allowed as in Case (1). 
If $a=3$ in $A$, then we can write $A=3H-2E_{j_1}-E_{j_2}-\cdots-E_{j_{8}}$. Then $B_1\cdot A=0$ implies that $E_{j_1}=E_1$ must be true, and exactly one of $E_x,E_y,E_z$ appears in $A$. With
$F_2\cdot A=0$, we see that $E_x$ is contained in $A$. But this leads to $A\cdot F_1=-2$, which is a contradiction. To proceed further,
we rule out $a=2$ in $A$ by a similar argument as in Case (1). Now suppose $a=1$ in $A$. Then $B_1\cdot A=0$ implies that 
$E_1$ and exactly one of $E_x,E_y,E_z$ appears in $A$. Then $A\cdot F_2=0$ implies $A$ must contain $E_x$. But we then get
$A\cdot F_1<0$ which is a contradiction. This leaves only two possibilities for $A$: $E_u-E_v$, $E_{i_2}-E_{i_3}$. But these two classes intersect trivially, contradicting $F_{4,1}\cdot F_{4,2}=1$. 
Hence Case (2) is also eliminated. 

\vspace{1.5mm}

{\bf Case (3):} Suppose $a=0$ in $B_1$. Then since $a\geq 4$ in $B_2$, we see immediately that the sum of the $a$-coefficients of $F_1,F_2,F_3$ is either $1$ or $-1$. In the former case, $a=4$ in $B_2$. If we write $B_2=4H-\sum_{i=1}^{10}b_i E_i$,
then $c_1(K_X)\cdot B_2+B_2^2=2$ and $B_2^2=6$ give
$$
B_2=4H-2E_{j_1}-E_{j_2}-\cdots-E_{j_7}.
$$
But $B_1$ takes the form of $B_1=E_{i_1}-E_{i_2}-\cdots-E_{i_6}$. The fact that there are totally $10$ $E_i$-classes implies easily that $B_1\cdot B_2<0$. In the latter case, $a=5$ in $B_2$. But then by Lemma 3.4(2), exactly one of $F_1,F_2,F_3$ has $a=-1$. Suppose it is $F_1$. Then
$F_1=-H+2E_1$. It is easy to see that $F_1\cdot B_2$ is always odd because the $a$-coefficient of
$B_2$ is $5$. This rules out Case (3). 

\vspace{1.5mm}

{\bf Case (4):} Suppose $a=1$ in $B_1$. Then $a=4$ in $B_2$ and $F_1=-H+2E_1$. But note that $B_1\cdot F_1$ is always odd, hence this is not possible. This rules out Case (4). 

\vspace{1.5mm}

{\bf Case (5):} Suppose $a>1$ in $B_1$. Then with $a\geq 4$ in $B_2$, the total contribution of
$B_1,B_2$ to the $a$-coefficient of $-3c_1(K_X)$ is at least $12$. But the $a$-coefficient of $-3c_1(K_X)$ is $9$, so $F_1,F_2,F_3$ must contribute $-3$ to $a$-coefficient of $-3c_1(K_X)$.
This is not possible by Lemmas 3.4(2) and 4.2(1). Hence Case (5) is eliminated. 

The above discussions show that there is no component $B_i$ with $B_i^2<0$. Hence the
proposition is proved. 

\end{proof}

\begin{proposition}
Let $F_1,F_2,\cdots,F_9$ be a disjoint union of symplectic $(-3)$-spheres in a rational 
$4$-manifold $X$, and let $\{B_i\}$ be a set of disjoint symplectic surfaces, possibly empty, which
lie in the complement of $F_1,F_2,\cdots,F_9$, such that  
$$
c_1(K_X)=-\frac{1}{3}(F_1+F_2+\cdots+F_9)-\frac{2}{3}\sum_i B_i.
$$
Then $\{B_i\}$ must be empty if each $B_i$ is a torus of self-intersection zero.
\end{proposition}

\begin{proof}
We shall prove by contradiction. Suppose $\{B_i\}\neq \emptyset$, where each $B_i$ is
a torus with $B_i^2=0$. We first note that $c_1(K_X)^2=-3+\frac{4}{9}\sum_i B_i^2=-3$,
so that $X=\C\P^2 \# 12\overline{\C\P^2}$. Again, by analyzing the contributions of $B_i$ 
to the $a$-coefficient of $-c_1(K_X)$, it follows easily that there is only one component in 
$\{B_i\}$, and moreover, the sum of the $a$-coefficients of $F_1,\cdots,F_9$ can be at most $3$. 

With the preceding understood, the following is the key observation:  

\vspace{1.5mm}

{\it The maximal number of disjoint symplectic $(-3)$-spheres in $\C\P^2 \# 12\overline{\C\P^2}$
with $a$-coefficient equaling $0$ is six, and moreover, such six $(-3)$-spheres must be of the form:
\begin{itemize}
\item $E_{i_1}-E_{i_2}-E_{i_3}$, $E_{i_2}-E_{i_3}-E_{i_4}$,
\item $E_{j_1}-E_{j_2}-E_{j_3}$, $E_{j_2}-E_{j_3}-E_{j_4}$,
\item $E_{k_1}-E_{k_2}-E_{k_3}$, $E_{k_2}-E_{k_3}-E_{k_4}$,
\end{itemize}
where $i_1,i_2,i_3,i_4,j_1,j_2,j_3,j_4,k_1,k_2,k_3,k_4$ are distinct indices.}

\vspace{1.5mm}

To see this,  let $A=E_i-E_j-E_k$, $A^\prime=E_r-E_s-E_t$ be two distinct symplectic $(-3)$-spheres such that $A\cdot A^\prime=0$. Then it is easy to see that if $E_r$ is not contained in $A$ and $E_i$ not
in $A^\prime$, the indices $i,j,k,r,s,t$ must be distinct. On the other hand, without loss of generality,
assume that $E_r$ appears in $A$, say $r=j$,  then $k=s$ or $t$ must be true. The above claim follows easily from the fact that we only have these two alternatives. 

With the preceding understood, note that by Lemma 3.4(2), $a\geq -1$ in each $F_k$. Moreover, if $a=-1$, the class must be $-H+2E_1$, and there is at most one such $(-3)$-sphere in $F_1,\cdots,F_9$ by Lemma 4.2(1). 

We claim that the class $A=-H+2E_1$ can not be represented by any of the $(-3)$-spheres $F_k$.
To see this, note that if $A^\prime$ is the class of one of $F_k$ which has positive $a$-coefficient,
then $A\cdot A^\prime\neq 0$ unless the $a$-coefficient of $A^\prime$ is an even number. Now
with the fact that the sum of the $a$-coefficients of $F_1,\cdots,F_9$ can be at most $3$, it follows 
easily that at least six of the nine $(-3)$-spheres $F_1,\cdots,F_9$ have zero $a$-coefficient.
But this is a contradiction because it is easy to see that $A=-H+2E_1$ intersects nontrivially with
one of the six $(-3)$-spheres. Hence the claim that the class $A=-H+2E_1$ can not occur.
It follows easily that six of the nine $(-3)$-spheres $F_1,\cdots,F_9$ have zero $a$-coefficient, and three of them have $a$-coefficient equaling $1$. Moreover, note that the $a$-coefficient of $B$
must be $3$.

To proceed further, we denote the single component of $\{B_i\}$ by $B$. 
Note that as $B$ is disjoint from the six $(-3)$-spheres with
zero $a$-coefficient, it must be the class:
$$
B=3H-E_{i_1}-E_{i_2}-E_{i_4}-E_{j_1}-E_{j_2}-E_{j_4}-E_{k_1}-E_{k_2}-E_{k_4}.
$$
In other words, the three $E_i$-classes which are missing from $B$ are $E_{i_3},E_{j_3},E_{k_3}$.
With this understood, let $A=H-E_{l_1}-E_{l_2}-E_{l_3}-E_{l_4}$ be any of the three $(-3)$-spheres 
whose $a$-coefficient equals $1$. Then $A\cdot B=0$ implies that exactly three of the four $E_i$-classes $E_{l_1}, E_{l_2}, E_{l_3}, E_{l_4}$ must appear in $B$. Without loss of generality, let
$E_{l_4}$ be the one not contained in $B$, and without loss of generality, assume 
$E_{l_4}=E_{i_3}$. Then since $A$ intersects trivially with the $(-3)$-sphere 
$E_{i_2}-E_{i_3}-E_{i_4}$, it is easy to see that $A$ must also contain the class $E_{i_2}$.
Now with both $E_{i_2},E_{i_3}$ contained in $A$, the intersection of $A$ with the $(-3)$-sphere 
$E_{i_1}-E_{i_2}-E_{i_3}$ must be negative. This is a contradiction, hence the proposition is
proved.

\end{proof}

\begin{proposition}
Let $F_{j,1}$, $F_{j,2}$, where $1\leq j\leq 5$, be a disjoint union of five pairs of symplectic 
$(-3)$-sphere and $(-2)$-sphere intersecting transversely and positively at one point in a 
rational $4$-manifold $X$, and let $\{B_i\}$ be a set of disjoint symplectic surfaces, possibly 
empty, lying in the complement of $F_{j,1}$, $F_{j,2}$, such that 
$$
c_1(K_X)=-\sum_{j=1}^5(\frac{2}{5}F_{j,1}+\frac{1}{5}F_{j,2})
-\frac{4}{5}\sum_i B_i.
$$
Then $\{B_i\}$ must be empty if each $B_i$ is a torus of self-intersection zero.
\end{proposition}

\begin{proof}
We prove by contradiction. Suppose to the contrary that $\{B_i\}$ is nonempty, with each
$B_i$ being a torus of self-intersection zero. Then again, there can be only one component in
$\{B_i\}$. We call it $B$. Moreover, the $a$-coefficient of $B$ is either $4$ or $3$.

Before we proceed further, note that $c_1(K_X)^2=-2$, so that $X=\C\P^2 \# 11\overline{\C\P^2}$.
In particular, there are only $11$ $E_i$-classes in $X$. 

\vspace{1.5mm}

{\bf Case (1):} Suppose $a=4$ in $B$. Then if we write $B=4H-\sum_{i=1}^{11} b_i E_i$,
the $b_i$'s satisfy the following equation: $4(4-3)=\sum_{i=1}^{11} b_i(b_i-1)$ (see the proof
of Lemma 4.2). It follows easily that 
$$
B=4H-2E_{j_1}-2E_{j_2}-E_{j_3}-\cdots-E_{j_{10}}.
$$
With this understood, note that since the contribution of $B$ to the $a$-coefficient of 
$-5c_1(K_X)$ is $16>15$, it follows easily that there must be a $(-3)$-sphere $F_{j,1}$ having $a=-1$, 
with the remaining four $(-3)$-spheres having $a=0$. By Lemma 3.4(2), the class of the $(-3)$-sphere with $a=-1$ 
must be $-H+2E_1$, and since its intersection 
with $B$ is zero, either $E_1=E_{j_1}$ or  $E_1=E_{j_2}$ must be true. Without loss of
generality, assume $E_{j_1}=E_1$. Then it is clear that none of the four $(-3)$-spheres 
with $a=0$ can contain the class $E_1=E_{j_1}$. 

With the preceding understood, it is easy to see that the expressions of the four $(-3)$-spheres 
with $a=0$ fall into the following two possibilities without loss of generality:
\begin{itemize}
\item [{(!)}] $E_{i_1}-E_{i_2}-E_{i_3}$, $E_{i_2}-E_{i_3}-E_{i_4}$, $E_{i_5}-E_{i_6}-E_{i_7}$, 
$E_{i_6}-E_{i_7}-E_{i_8}$,
\item [{(!!)}] $E_{i_1}-E_{i_2}-E_{i_3}$, $E_{i_2}-E_{i_3}-E_{i_4}$, $E_{i_5}-E_{i_6}-E_{i_7}$, 
$E_{i_8}-E_{i_9}-E_{i_{10}}$.
\end{itemize}
Suppose we are in case (!). Consider the pair of $(-3)$-spheres $E_{i_1}-E_{i_2}-E_{i_3}$ and
$E_{i_2}-E_{i_3}-E_{i_4}$. If the class $E_{i_1}$ is not contained in the expression of $B$,
then it is easy to see that none of the four classes $E_{i_1}, E_{i_2}, E_{i_3},E_{i_4}$ are 
contained in $B$. But this contradicts the fact that there are only $11$ $E_i$-classes in total.
Hence $E_{i_1}$ must be contained in $B$. We know that $E_{i_1}\neq E_{j_1}$. If
$E_{i_1}=E_{j_2}$, then both $E_{i_2},E_{i_3}$ are contained in $B$, and it follows that 
$E_{i_4}$ does not show up in the expression of $B$. On the other hand, if $E_{i_1}=E_{j_s}$
for some $s>2$, then it is easy to see that $E_{i_3}$ can not show up in $B$. In any event,
one of $E_{i_3}$, $E_{i_4}$ does not appear in the expression of $B$. With this understood, 
the same argument shows that one of $E_{i_7}$, $E_{i_8}$ also does not appear in the 
expression of $B$. But this clearly contradicts the fact that there are totally only $11$ 
$E_i$-classes, hence case (!) is not possible. The argument for case (!!) is similar. First, note
that one of $E_{i_3}$, $E_{i_4}$ does not appear in $B$ as we have argued in case (!). 
Secondly, consider the pair of $(-3)$-spheres $E_{i_5}-E_{i_6}-E_{i_7}$ and
$E_{i_8}-E_{i_9}-E_{i_{10}}$. We observe that one of the classes $E_{i_5},E_{i_8}$ 
is not equal to $E_{j_2}$.
Without loss of generality, assume $E_{i_5}\neq E_{j_2}$. Then one of $E_{i_6}, E_{i_7}$
can not be contained in $B$. So totally there are at least $2$ $E_i$-classes not
contained in $B$, which contradicts the fact that there are only $11$ $E_i$-classes.
Hence case (!!) is also not possible. This rules out Case (1).

\vspace{1.5mm}

{\bf Case (2):} Suppose $a=3$ in $B$. Then by Lemma 4.2(2), 
$B=3H-E_{j_1}-\cdots-E_{j_{9}}$.
With this understood, we first observe that the class $-H+2E_1$ intersects nontrivially with $B$,
so none of the five $(-3)$-spheres can have $a<0$. On the other hand, from the proof of Proposition 4.4, it is
easy to see that the five $(-3)$-spheres can not all have $a=0$. Now observe that the contribution of $B$ to the 
$a$-coefficient of $-5c_1(K_X)$ is $12$. It follows easily that exactly one of the five $(-3)$-spheres 
has $a=1$, and the other four all have $a=0$. The possible expressions of the four $(-3)$-spheres with $a=0$ 
are given in either (!) or (!!) listed in Case (1). In the second case (!!), it is easy to see that 
there are three $E_i$-classes in the four $(-3)$-spheres with $a=0$ which do not show up in $B$. 
This contradicts the fact that there are only $11$ $E_i$-classes, hence (!!) is not possible. In case (!), 
it is easy to see that $E_{i_3}$, $E_{i_7}$ are precisely the two $E_i$-classes that are not in the expression of $B$.
To derive a contradiction, we consider the $(-3)$-sphere with $a=1$. We write its class as
$A=H-E_{l_1}-E_{l_2}-E_{l_3}-E_{l_4}$. Then we note that one of $E_{i_1}$ and $E_{i_5}$, 
say $E_{i_1}$, must appear in the above expression. It follows that $E_{i_1}, E_{i_2}, E_{i_4}$ must all appear in $A$, but not $E_{i_3}$. Without loss of generality, assume 
$\{E_{i_1}, E_{i_2}, E_{i_4}\}=\{E_{l_1}, E_{l_2}, E_{l_3}\}$. Then $A\cdot B=0$ implies easily
that $E_{l_4}$ can not show up in $B$. It follows that $E_{l_4}=E_{i_7}$ must be true. But this 
implies that $A$ has nonzero intersection with the $(-3)$-sphere $E_{i_5}-E_{i_6}-E_{i_7}$, 
which is a contradiction. Hence (!) is also not possible. This rules out Case (2) as well, and 
the proof of the proposition is complete. 

\end{proof}

\section{The proof of main theorems}

We begin with the key technical lemma, which classifies the possible homological expressions of
a disjoint union of $8$ symplectic $(-2)$-spheres in $\C\P^2\# 9\overline{\C\P^2}$ under a very
delicately chosen assumption on the symplectic structure. 

\begin{lemma}
Let $F_1,F_2,\cdots,F_8$ be a disjoint union of $8$ symplectic $(-2)$-spheres in 
$X=\C\P^2\# 9\overline{\C\P^2}$. Suppose the symplectic structure $\omega$ obeys the
following constraints:
\begin{itemize}
\item one of $F_k$ has $\omega$-area $\delta_1$, the remaining seven have $\omega$-area $\delta_2$;
\item $\delta_2<\delta_1<2\delta_2$, and $7\delta_i< -c_1(K_X)\cdot [\omega]$ for $i=1,2$.
\end{itemize}
Then for any given reduced basis $H, E_1,E_2,\cdots, E_9$ of $(X,\omega)$, there are three possibilities for the classes of $F_1,F_2,\cdots,F_8$:
\begin{itemize}
\item [{(a)}] $F_1=3H-2E_{i_1}-E_{i_2}-\cdots-E_{i_7}-E_{i_8}$, and 
$F_2=H-E_{i_2}-E_{i_3}-E_{i_4}$, $F_3=H-E_{i_2}-E_{i_5}-E_{i_6}$, 
$F_4=H-E_{i_2}-E_{i_7}-E_{i_8}$, $F_5=H-E_{i_3}-E_{i_5}-E_{i_7}$, 
$F_6=H-E_{i_3}-E_{i_6}-E_{i_8}$, $F_7=H-E_{i_4}-E_{i_5}-E_{i_8}$, 
$F_8=H-E_{i_4}-E_{i_6}-E_{i_7}$.
\item [{(b)}] $F_1=H-E_{l_1}-E_{l_2}-E_{l_3}$, $F_2=H-E_{l_1}-E_{l_4}-E_{l_5}$, 
$F_3=H-E_{l_1}-E_{l_6}-E_{l_7}$, $F_4=H-E_{l_2}-E_{l_4}-E_{l_6}$, 
$F_5=H-E_{l_3}-E_{l_5}-E_{l_6}$, $F_6=H-E_{l_2}-E_{l_5}-E_{l_7}$, 
$F_7=H-E_{l_3}-E_{l_4}-E_{l_7}$, and $F_8=E_{l_8}-E_{l_9}$.
\item [{(c)}] $F_1=H-E_{l_1}-E_{l_2}-E_{l_3}$, $F_2=H-E_{l_1}-E_{l_4}-E_{l_5}$, 
$F_3=H-E_{l_1}-E_{l_6}-E_{l_7}$, $F_4=H-E_{l_1}-E_{l_8}-E_{l_9}$, 
$F_5=E_{l_2}-E_{l_3}$, $F_6=E_{l_4}-E_{l_5}$, $F_7=E_{l_6}-E_{l_7}$, $F_8=E_{l_8}-E_{l_9}$. 
\end{itemize}
\end{lemma}

\begin{proof}

By Lemma 3.6, $a\leq 3$ in each $F_k$. 

\vspace{1.5mm}

{\bf Case (1):} Suppose there is a $F_k$ whose $a$-coefficient equals $3$. 
We may assume without loss of generality that it is $F_1$, and write 
$$
F_1=3H-2E_{i_1}-E_{i_2}-\cdots-E_{i_7}-E_{i_8}.
$$
Furthermore, we denote by $E_{i_9}$ the unique $E_i$-class that is missing in $F_1$. 

Let $A$ be the class of any of the remaining $(-2)$-spheres, i.e., $F_2,F_3,\cdots,F_8$.
Our first observation is that $a\neq 3$ in $A$. To see this, we note that if the $a$-coefficient
of $A$ equals $3$, then $A\cdot F_1=0$ implies that $A$ must take the following form without
loss of generality:
$$
A=3H-E_{i_1}-2E_{i_2}-\cdots-E_{i_7}-E_{i_9}.
$$
With this understood, we observe that 
$$
F_1+A+c_1(K_X)=3H-2E_{i_1}-2E_{i_2}-E_{i_3}-E_{i_4}-\cdots-E_{i_7}, 
$$
which can be written as a sum of three terms of the form $H-E_i-E_j-E_k$. It follows that
$\omega(A+F_1)\geq -c_1(K_X)\cdot [\omega]$, which is a contradiction. Hence the claim.

To proceed further, we first examine the classes $A$ whose $a$-coefficient equals $1$. 
Note that if $A$ is a class with $a=1$, then $A\cdot F_1=0$ implies that if $E_{i_1}$ appears in $A$, then so does $E_{i_9}$.
This allows us to divide the classes $A$ with $a=1$ into two types:
$$
\mbox{($\alpha$)} \; A=H-E_{i_1}-E_{i_9}-E_x,\;\;\;\;\;  \mbox{($\beta$)} \; A=H-E_r-E_s-E_x,
$$
where $E_x,E_r,E_s\in\{E_{i_2},E_{i_3},\cdots,E_{i_8}\}$.

\vspace{1.5mm}

{\bf Claim:} {\it There are no classes $A$ with $a=2$.}

\vspace{1.5mm}

{\bf Proof of Claim:} We first observe that if $A$ is a class with $a=2$, then $E_{i_1}$ is not contained
in $A$. This is because if $E_{i_1}$ is contained in $A$, then $A\cdot F_1=0$ implies that $E_{i_9}$
must also be contained in $A$, and $A$ takes the following form
$$
A=2H-E_{i_1}-E_{i_9}-E_{k_1}-E_{k_2}-E_{k_3}-E_{k_4}.
$$
But this would lead to a contradiction
$$
\omega(F_1+A)+c_1(K_X)\cdot [\omega]
=\omega(2H-2E_{i_1}-E_{k_1}-E_{k_2}-E_{k_3}-E_{k_4})\geq 0,
$$
as $2H-2E_{i_1}-E_{k_1}-E_{k_2}-E_{k_3}-E_{k_4}$ is a sum of terms of the form 
$H-E_i-E_j-E_k$.

With the preceding understood, suppose to the contrary that there is a class $A$ with $a=2$.
Then without loss of generality, we may write it as 
$$
A_1=2H-E_{i_2}-E_{i_3}-E_{i_4}-E_{i_5}-E_{i_6}-E_{i_7}.
$$
Moreover, if $A$ is another class of $F_2,F_3,\cdots,F_8$ with $a=2$, then it is easy to
check that $A_1\cdot A<0$. Hence $A_1$ is the only one with $a=2$.

Next we examine the possible classes of $A$ with $a=1$, which intersects trivially with $F_1$ and $A_1$.
It is easy to see that if $A$ is a class with $a=1$ and $A\cdot A_1=0$, then $A$ can not be
of type ($\alpha$), and for a type ($\beta$) class, $A$ must contain $E_{i_8}$. It is easy to see
that maximally, there are three such type ($\beta$) classes that are mutually disjoint, i.e., 
$$
A_2=H-E_{i_2}-E_{i_3}-E_{i_8}, A_3=H-E_{i_4}-E_{i_5}-E_{i_8}, A_4=H-E_{i_6}-E_{i_7}-E_{i_8}
$$
without loss of generality. The remaining three classes of $A$ must all have $a$-coefficient equaling $0$, and it
is easy to see that, without loss of generality, they are
$$
A_5=E_{i_2}-E_{i_3}, \; A_6=E_{i_4}-E_{i_5}, \; A_7=E_{i_6}-E_{i_7}.
$$
To derive a contradiction, we appeal to the area constraints. First, we observe that the area of $F_1$ must be
greater than the area of any of $A_5,A_6,A_7$. For example, 
$$
\omega(F_1-A_5)=\omega(3H-2E_{i_1}-2E_{i_2}-E_{i_4}-E_{i_5}-\cdots-E_{i_8})\geq 0
$$
as $3H-2E_{i_1}-2E_{i_2}-E_{i_4}-E_{i_5}-\cdots-E_{i_8}$ is a sum of terms of the form 
$H-E_i-E_j-E_k$. 
Furthermore, note that if $\omega(F_1-A_5)=0$, then $\omega(H-E_x-E_y-E_z)=0$ for any three classes $E_x,E_y,E_z$
from the set $\{E_{i_1},E_{i_2},E_{i_4}, E_{i_5},\cdots,E_{i_8}\}$. In particular, $E_{i_4},E_{i_5}$, $E_{i_6},E_{i_7}$ have the
same area, contradicting $\omega(A_6)>0$, $\omega(A_7)>0$. It follows that $\omega(F_1)=\delta_1$ and the remaining classes
have the same area equaling $\delta_2<\delta_1$. With this understood, we note that $\omega(F_1-A_5-A_4)=\omega(2H-2E_{i_1}-2E_{i_2}-E_{i_4}-E_{i_5})\geq 0$
as $2H-2E_{i_1}-2E_{i_2}-E_{i_4}-E_{i_5}$ is a sum of terms of the form $H-E_i-E_j-E_k$, contradicting the constraint $\delta_1<2\delta_2$. This finishes off the proof of the Claim. 

\vspace{1.5mm}

Now back to the discussion on Case (1), we claim that no type ($\alpha$) classes can occur. Suppose 
to the contrary that there is a type ($\alpha$) class, call it $A_1$. It is easy to see that any other type ($\alpha$) class has a negative intersection with $A_1$, hence $A_1$ is the only type ($\alpha$) class. Without loss of generality, let $A_1=H-E_{i_1}-E_{i_9}-E_{i_8}$.  Now let $A$ be
any type ($\beta$) class such that $A\cdot A_1=0$. Then $A$ must contain $E_{i_8}$, and 
furthermore, it is easy to see that maximally, there are three such type ($\beta$) classes which are mutually disjoint. Without loss of generality, they are 
$$
A_2=H-E_{i_2}-E_{i_3}-E_{i_8}, A_3=H-E_{i_4}-E_{i_5}-E_{i_8}, A_4=H-E_{i_6}-E_{i_7}-E_{i_8}
$$
The remaining three classes of $A$ must all have $a$-coefficient equaling $0$, and it
is easy to see that, without loss of generality, they are
$$
A_5=E_{i_2}-E_{i_3}, \; A_6=E_{i_4}-E_{i_5}, \; A_7=E_{i_6}-E_{i_7}.
$$
This possibility can be ruled out using the area constraints as we did in the proof of the Claim. 
Hence no type ($\alpha$) classes can occur.

With the preceding understood, we further observe that no class $A$ with $a=0$ can be realized by 
$F_2,F_3,\cdots,F_8$. Suppose, without loss of generality,  $A_1=E_{i_7}-E_{i_8}$ is realized. 
Let $A$ be a type ($\beta$) class which intersects trivially with $A_1$. Then it is easy to see that either $A$ 
contains both $E_{i_7},E_{i_8}$, or $A$ contains neither $E_{i_7}$ nor $E_{i_8}$. It is clear that there can be 
at most one type ($\beta$) class which contains both $E_{i_7},E_{i_8}$. Without loss of generality, we let it be
$A_2=H-E_{i_2}-E_{i_7}-E_{i_8}$. Then any other type ($\beta$) classes which intersect trivially with $A_1,A_2$
must contain $E_{i_2}$, and there are maximally two such classes: $H-E_{i_2}-E_{i_3}-E_{i_4}$, $H-E_{i_2}-E_{i_5}-E_{i_6}$.
With this understood, note that there are at most two other classes, both having $a=0$, that are allowed, i.e.,
$E_{i_3}-E_{i_4}$, $E_{i_5}-E_{i_6}$, bringing total number of allowable classes for $F_2,F_3,\cdots,F_8$ to $6$.
But apparently, there are not enough many classes, hence our claim.

The above discussions show that the classes of $F_2,F_3,\cdots,F_8$ are all of type ($\beta$).
With this understood, we first rule out the possibility that no triple of $F_2,F_3,\cdots,F_8$ shares
a common $E_i$-class. Suppose to the contrary that this is the case. Then without loss of 
generality, we write 
$$
F_2=H-E_{i_2}-E_{i_3}-E_{i_4}, \;\;\;  F_3=H-E_{i_2}-E_{i_5}-E_{i_6}.
$$
Note that by our assumption, $F_4$ can not contain $E_{i_2}$. With this understood, 
$F_4\cdot F_2=F_4\cdot F_3=0$ implies that we may write $F_4=H-E_{i_3}-E_{i_5}-E_{i_7}$
without loss of generality. Now observe that $F_5$ can not contain $E_{i_2}, E_{i_3},E_{i_5}$.
Hence $F_5=H-E_{i_4}-E_{i_6}-E_{i_7}$ must be true. Now examining the class of $F_6$,
by our assumption it can not contain any of $E_{i_2},E_{i_3},\cdots, E_{i_7}$. This is clearly
a contradiction. Hence the claim.

With the preceding understood, we may write without loss of generality that
$$
F_2=H-E_{i_2}-E_{i_3}-E_{i_4}, \;\;  F_3=H-E_{i_2}-E_{i_5}-E_{i_6},\;\; 
F_4=H-E_{i_2}-E_{i_7}-E_{i_8}.
$$
With this given, it is easy to see that the other four $(-2)$-spheres must be
$$
F_5=H-E_{i_3}-E_{i_5}-E_{i_7}, \;\;\; F_6=H-E_{i_3}-E_{i_6}-E_{i_8},
$$
and 
$$
F_7=H-E_{i_4}-E_{i_5}-E_{i_8}, \;\;\; F_8=H-E_{i_4}-E_{i_6}-E_{i_7}.
$$
This possibility of classes of $F_1,F_2,\cdots,F_8$ is listed as Case (a) of the lemma. 

\vspace{1.5mm}

{\bf Case (2):} Suppose $a\leq 2$ in all eight $(-2)$-spheres $F_1,F_2,\cdots,F_8$.

\vspace{1.5mm}

{\bf (i):} Assume at least two of $F_1,F_2,\cdots,F_8$ have $a$-coefficient equaling $2$. Without loss of generality,
let $F_1,F_2$ be such two $(-2)$-spheres. It is easy to see from $F_1\cdot F_2=0$ that $F_1,F_2$ must have exactly
$4$ $E_i$-classes in common. Hence without loss of generality, we may write them as 
$$
F_1=2H-E_{j_1}-E_{j_2}-E_{j_3}-E_{j_4}-E_{j_5}-E_{j_6}, F_2=2H-E_{j_1}-E_{j_2}-E_{j_3}-E_{j_4}-E_{j_7}-E_{j_8}.
$$
With this understood, we denote by $E_{j_9}$ the unique $E_i$-class that is missing in $F_1$, $F_2$. Moreover, we denote by
$A$ the class of any of the remaining $(-2)$-spheres, i.e., $F_3,F_4,\cdots,F_8$. 

\vspace{1.5mm}

{\bf Claim:} {\it There are no classes $A$ which contains $E_{j_9}$.}

\vspace{1.5mm}

{\bf Proof of Claim:} First, it is easy to see that if $A$ is a class with $a=0$ which contains $E_{j_9}$, the intersection of $A$ with
one of $F_1,F_2$ will be nonzero. Now suppose $A$ is a class with $a=1$ which contains $E_{j_9}$. Then $A\cdot F_1=
A\cdot F_2=0$ implies that $A$ must be of the form $A=H-E_x-E_y-E_{j_9}$ for some $E_x,E_y\in\{E_{j_1},\cdots,E_{j_4}\}$.
With this understood, we note that
$$
F_1+F_2+A+c_1(K_X)=2H-E_{j_1}-E_{j_2}-E_{j_3}-E_{j_4}-E_x-E_y,
$$
which is a sum of terms of the form $H-E_i-E_j-E_k$, leading to a contradiction in areas: 
$\omega(F_1+F_2+A)\geq -c_1(K_X)\cdot [\omega]$. Finally, suppose $A$ is a class 
with $a=2$ which contains $E_{j_9}$. Then $A\cdot F_1=A\cdot F_2=0$ implies that, without loss of generality, 
$$
A=2H-E_{j_1}-E_{j_2}-E_{j_3}-E_{j_5}-E_{j_7}-E_{j_9}. 
$$
In this case, we have $F_1+F_2+A+c_1(K_X)=3H-2E_{j_1}-2E_{j_2}-2E_{j_3}-E_{j_4}-E_{j_5}-E_{j_7}$, which by the same reason also leads to the
contradiction in areas: $\omega(F_1+F_2+A)\geq -c_1(K_X)\cdot [\omega]$. Hence the Claim. 

\vspace{1.5mm}

Now back to the discussion on Case (2), it is easy to see that there are two other classes $A$ with $a=2$ and trivial
mutual intersection, which intersect trivially with $F_1,F_2$; we denote them by $A_1,A_2$, where
$$
A_1=2H-E_{j_1}-E_{j_2}-E_{j_5}-E_{j_6}-E_{j_7}-E_{j_8}, A_2=2H-E_{j_3}-E_{j_4}-E_{j_5}-E_{j_6}-E_{j_7}-E_{j_8}.
$$
On the other hand, let $A$ be a class with $a=1$ which intersects trivially with $F_1,F_2$. Then $A$ must be of the form
$A=H-E_r-E_s-E_t$, where $E_r\in \{E_{j_1}, E_{j_2}, E_{j_3}, E_{j_4}\}$, $E_s\in \{E_{j_5},E_{j_6}\}$, and
$E_t\in \{E_{j_7},E_{j_8}\}$. 

With the preceding understood, if both of $A_1,A_2$ are realized by the $(-2)$-spheres, then it is easy to see that no classes $A$
with $a=1$ can be realized. On the other hand, it is easy to see that there are maximally $4$ classes $A$ with $a=0$:
$$
E_{j_1}-E_{j_2}, \; E_{j_3}-E_{j_4}, \; E_{j_5}-E_{j_6}, \; E_{j_7}-E_{j_8}.
$$
Hence all of them must be realized. With this understood, it is easy to see that three of $F_1,F_2,A_1,A_2$ and all of the classes 
with $a=0$ must have the smaller area $\delta_2$. As a consequence, we may assume without loss of generality that $\omega(F_1)=\omega(E_{j_1}-E_{j_2})$. Then observe that 
$2H-2E_{j_1}-E_{j_3}-E_{j_4}-E_{j_5}-E_{j_6}$ is a sum of terms of the form $H-E_i-E_j-E_k$, 
so that 
$$
\omega(2H-2E_{j_1}-E_{j_3}-E_{j_4}-E_{j_5}-E_{j_6})=\omega(F_1)-\omega(E_{j_1}-E_{j_2})=0
$$
implies that $E_{j_3}, E_{j_4}, E_{j_5}, E_{j_6}$ have the same area. But this contradicts the fact that 
the classes $E_{j_3}-E_{j_4}$, $E_{j_5}-E_{j_6}$ are realized by the symplectic $(-2)$-spheres. 
It follows that $A_1,A_2$ can not be both realized. 

Suppose only one of $A_1,A_2$, say $A_1$, is realized.
Then there are four classes $A$ with $a=1$ that are possible, i.e., 
$$
A_3=H-E_{j_3}-E_{j_5}-E_{j_7}, \; A_4=H-E_{j_3}-E_{j_6}-E_{j_8}, 
$$
and
$$
A_5=H-E_{j_4}-E_{j_5}-E_{j_8},  \; A_6=H-E_{j_4}-E_{j_6}-E_{j_7}.
$$
If all of $A_3,A_4,A_5,A_6$ are realized, then 
the remaining $(-2)$-sphere must have $a$-coefficient equaling $0$, and it must be the class $A_7=E_{j_1}-E_{j_2}$ without loss of generality. 
But this leads to a contradiction in areas as follows: note that 
$$
\omega(F_1-A_7)=\omega(2H-2E_{j_1}-E_{j_3}-E_{j_4}-E_{j_5}-E_{j_6})\geq 0
$$
as $2H-2E_{j_1}-E_{j_3}-E_{j_4}-E_{j_5}-E_{j_6}$ is a sum of terms of the form $H-E_i-E_j-E_k$.
Furthermore, if $\omega(F_1-A_7)=0$, the four classes $E_{j_3}, E_{j_4}, E_{j_5}, E_{j_6}$ must have the same area.
It follows easily that $\omega(A_7)=\delta_2<\delta_1$. The same argument applies with $F_1$ being replaced by $F_2$
or $A_1$. Note that at least two of $F_1,F_2,A_1$ must have the smaller area $\delta_2$. It follows easily that the six classes 
$E_{j_3}, E_{j_4}, E_{j_5}, E_{j_6},E_{j_7},E_{j_8}$ must have the same area. But this would imply that all the eight $(-2)$-spheres
have the same area, which is a contradiction. Finally, note that if any of $A_3,A_4,A_5,A_6$ is realized, $A_7$ is the only possible
class with $a=0$. If none of $A_3,A_4,A_5,A_6$ is realized, the allowable classes with $a=0$ are $E_{j_3}-E_{j_4}$, 
$E_{j_5}-E_{j_6}$, $E_{j_7}-E_{j_8}$, in addition to $A_7$. It follows that neither $A_1$ nor $A_2$ can be realized. 

The above discussion shows that $F_1,F_2$ are the only two $(-2)$-spheres with $a=2$. From the discussion, it is also clear
that the maximal number of mutually disjoint classes with $a=1$ which intersect trivially with $F_1,F_2$ is $4$, which,
without loss of generality, are given by $A_3,A_4,A_5,A_6$. If any of them is realized, there is only one possible class with $a=0$,
i.e., $A_7=E_{j_1}-E_{j_2}$. If none of the $a=1$ classes are realized, then there are maximally $4$ classes with $a=0$ that
are allowed. In any event, we do not have enough classes that can be realized. Thus (i) is eliminated. 

\vspace{1.5mm}

{\bf (ii):} Assume only one of $F_1,F_2,\cdots,F_8$ has $a$-coefficient equaling $2$. Without loss of generality, assume it is $F_1$,
and we write 
$$
F_1=2H-E_{k_1}-E_{k_2}-E_{k_3}-E_{k_4}-E_{k_5}-E_{k_6}.
$$
We denote the remaining three $E_i$-classes by $E_{k_7},E_{k_8},E_{k_9}$, and denote by $A$ the class of any of
the $(-2)$-spheres $F_2,F_3,\cdots,F_8$. 

Examining classes $A$ with $a=1$ which intersect trivially with $F_1$, we note that $A$ must be of the form 
$$
A=H-E_r-E_s-E_t, \mbox{ where } E_r,E_s\in \{ E_{k_1},\cdots,E_{k_6}\} \mbox{ and } E_t\in \{E_{k_7},E_{k_8},E_{k_9}\}.
$$

Consider first the case where amongst the classes $A$ with $a=1$, the $E_i$-classes $E_{k_1},E_{k_2},\cdots,E_{k_6}$
can only appear once. It is easy to see that in this case, all the $a=1$ classes must have a common $E_i$-class which must
be one of $E_{k_7},E_{k_8},E_{k_9}$. It is clear that there are maximally three such classes with $a=1$, i.e., 
$$
H-E_{k_1}-E_{k_2}-E_{k_7}, \; \; H-E_{k_3}-E_{k_4}-E_{k_7}, \;\; H-E_{k_5}-E_{k_6}-E_{k_7}
$$
without loss of generality. The remaining four $(-2)$-spheres must have $a$-coefficient equaling $0$, and they must be 
$$
E_{k_1}-E_{k_2}, \; E_{k_3}-E_{k_4}, \; E_{k_5}-E_{k_6}, E_{k_8}-E_{k_9}
$$
without loss of generality. With this understood, we note that the area of $F_1$ must be the larger 
$\delta_1$, with the remaining seven $(-2)$-spheres having area $\delta_2$. However, 
as $2H-2E_{k_1}-2E_{k_3}-E_{k_5}-E_{k_6}$ is a sum of terms of the form $H-E_i-E_j-E_k$, 
it follows that 
$$
\omega(F_1)-\omega(E_{k_1}-E_{k_2})-\omega(E_{k_3}-E_{k_4})=
\omega(2H-2E_{k_1}-2E_{k_3}-E_{k_5}-E_{k_6})\geq 0,
$$
which contradicts the constraint $\delta_1<2\delta_2$. Hence this first case is ruled out.

Next we assume that the $E_i$-classes $E_{k_1},E_{k_2},\cdots,E_{k_6}$ can appear at most
twice in the $a=1$ classes, and at least one of them, say $E_{k_1}$, appeared twice. 
Then without loss of generality, we may assume 
$$
A_1=H-E_{k_1}-E_{k_2}-E_{k_7}, \; A_2=H-E_{k_1}-E_{k_3}-E_{k_8}
$$
are realized by the $(-2)$-spheres. Since there are at most $4$ mutually disjoint classes with 
$a=0$ that can possibly be realized by the $(-2)$-spheres, we must have another $a=1$ class,
call it $A_3$. By our assumption, $A_3$ can not contain $E_{k_1}$. The fact that $A_3$ 
intersects trivially with $A_1,A_2$ implies that either $A_3=H-E_{k_2}-E_{k_3}-E_{k_9}$,
or without loss of generality, $A_3=H-E_{k_3}-E_{k_4}-E_{k_7}$. In the former case,
none of $E_{k_1},E_{k_2},E_{k_3}$ can appear anymore by our assumption, which implies 
easily that there can be no more $a=1$ classes. On the other hand, 
there is only one possible $a=0$ class, say $E_{k_5}-E_{k_6}$. Hence the former case is
not possible. In the latter case, $E_{k_1},E_{k_3},E_{i_7}$ can no longer appear. We note that
there is only one possible $a=0$ class, i.e., $E_{k_5}-E_{k_6}$, so there must be 
three more $a=1$ classes. Call them $A_4,A_5,A_6$. Then observe that $A_4,A_5,A_6$
intersect trivially with $A_2$, so all of them must contain $E_{k_8}$. Likewise, 
$A_4,A_5,A_6$ intersect trivially with $A_1$, so that they must all contain $E_{k_2}$,
which is clearly a contradiction. Thus this second case is also ruled out.

Finally, assume one of the $E_i$-classes $E_{k_1},E_{k_2},\cdots,E_{k_6}$, say $E_{k_1}$, 
appears in the $a=1$ classes three times. Without loss of generality, we assume 
$$
A_1=H-E_{k_1}-E_{k_2}-E_{k_7}, \; A_2=H-E_{k_1}-E_{k_3}-E_{k_8}, \;
A_3=H-E_{k_1}-E_{k_4}-E_{k_9}
$$
are realized by the $(-2)$-spheres. Again, there is only one possible $a=0$ class, i.e., 
$E_{k_5}-E_{k_6}$, so there must be three more $a=1$ classes, which are denoted by
$A_4,A_5,A_6$. It is easy to see that the following are the only possibility:
$$
A_4=H-E_{k_3}-E_{k_4}-E_{k_7}, \; A_5=H-E_{k_2}-E_{k_4}-E_{k_8}, \;
A_6=H-E_{k_2}-E_{k_3}-E_{k_9}.
$$
In order to rule out this last case, we observe that 
$$
F_1+\sum_{i=1}^6 A_i +c_1(K_X)= 5H-3(E_{k_1}+\cdots+E_{k_4})-E_{k_7}-E_{k_8}-E_{k_9}.
$$
The right-hand side is a sum of terms of the form $H-E_i-E_j-E_k$, hence has non-negative area. 
But this leads to a contradiction to the constraint
$7\delta_i < -c_1(K_X)\cdot [\omega]$ for $i=1,2$. Hence (ii) is also eliminated.

\vspace{1.5mm}

{\bf (iii):} It remains to consider the case where the $a$-coefficient of $F_1,F_2,\cdots,F_8$ 
equals either $1$ or $0$. We begin by noting that there are at least four $(-2)$-spheres with
$a=1$. 

The first possibility is that each $E_i$-class appears amongst the $a=1$ classes at most three times. 
To analyze this case, we take two of the $(-2)$-spheres with $a=1$, say $F_1,F_2$, and we write them as
$$
F_1=H-E_{l_1}-E_{l_2}-E_{l_3}, \; F_2=H-E_{l_1}-E_{l_4}-E_{l_5}.
$$
Assume $F_3$ also has $a$-coefficient equaling $1$. Then there are two possibilities for $F_3$:
either $F_3=H-E_{l_1}-E_{l_6}-E_{l_7}$ or $F_3=H-E_{l_2}-E_{l_4}-E_{l_6}$ without loss of
generality. There is at least one more $(-2)$-sphere with $a=1$, say $F_4$. Then
if $F_3=H-E_{l_1}-E_{l_6}-E_{l_7}$, we may assume without loss of generality that
$F_4=H-E_{l_2}-E_{l_4}-E_{l_6}$ because of our assumption that each $E_i$-class appears 
amongst the $a=1$ classes at most three times. If $F_3=H-E_{l_2}-E_{l_4}-E_{l_6}$ in the latter case,
we may assume $F_4=H-E_{l_3}-E_{l_5}-E_{l_6}$ (note that the other choice $F_4=H-E_{l_1}-E_{l_6}-E_{l_7}$ is equivalent to the former case). In any event, with these choices for 
$F_1,F_2,F_3,F_4$, there can be at most one $(-2)$-sphere with $a=0$. Consequently, there
must be three more $(-2)$-spheres with $a=1$. One can check easily that without loss of
generality, in this case the eight $(-2)$-spheres are 
$$
F_1=H-E_{l_1}-E_{l_2}-E_{l_3}, \; F_2=H-E_{l_1}-E_{l_4}-E_{l_5}, \; F_3=H-E_{l_1}-E_{l_6}-E_{l_7}, 
$$
$$
F_4=H-E_{l_2}-E_{l_4}-E_{l_6}, \; F_5=H-E_{l_3}-E_{l_5}-E_{l_6}, \; F_6=H-E_{l_2}-E_{l_5}-E_{l_7},
$$
$$
F_7=H-E_{l_3}-E_{l_4}-E_{l_7}, \mbox{ and } F_8=E_{l_8}-E_{l_9},
$$
which is listed as Case (b) of the lemma. 

The remaining possibility is that one of the $E_i$-classes appears in the $a=1$ classes four times.
In this case, it is easy to check that without loss of generality, the eight $(-2)$-spheres are 
$$
F_1=H-E_{l_1}-E_{l_2}-E_{l_3}, F_2=H-E_{l_1}-E_{l_4}-E_{l_5}, F_3=H-E_{l_1}-E_{l_6}-E_{l_7}, 
$$
$$
F_4=H-E_{l_1}-E_{l_8}-E_{l_9}, F_5=E_{l_2}-E_{l_3}, F_6=E_{l_4}-E_{l_5}, F_7=E_{l_6}-E_{l_7}, 
F_8=E_{l_8}-E_{l_9}.
$$
This is listed as Case (c) of the lemma. The proof of the lemma is complete. 

\end{proof}

In the following lemma, $D\subset \C$ is an open disc centered at the origin, with radius unspecified.
Let $\Psi: D\times D\rightarrow \C^2$ be a diffeomorphism onto a neighborhood of $0\in \C^2$, given by equations
$z_1=\psi(z,w)$, $z_2=w$, where $z_1,z_2$ are the standard holomorphic coordinates on $\C^2$ and $z,w$ are 
a local complex coordinate on the first and second factor in $D\times D$. Furthermore, assume $\Psi$ satisfies 
the following conditions: $\psi(z,w)$ is holomorphic in $w\in D$ (but only $C^\infty$ in $z\in D$), 
and $\psi(0,w)=0$ for all $w\in D$. 

\begin{lemma}
Let $C\subset \C^2$ be an embedded holomorphic disc containing the origin, where $C$ intersects the $z_2$-axis 
with a tangency of order $n>1$. Let $F: D\rightarrow \C^2$ be a holomorphic parametrization of $C$ such that $F(0)=0$.
Then the map $\pi_1\circ \Psi^{-1}\circ F: D\rightarrow D$ is an $n$-fold branched covering in a neighborhood of 
$0\in D$, ramified at $0$, where $\pi_1: D\times D\rightarrow D$ is the projection onto the first factor. 
\end{lemma}

\begin{proof}
Considering the parametrization $\Psi^{-1}\circ F$ of $C$ in the coordinates $(z,w)$, it is clear that after a re-parametrization of 
the domain $D$ if necessary, we may assume that $\Psi^{-1}\circ F$ is given by $z=f(\xi)$, $w=\xi$, where $\xi$ is a local holomorphic coordinate on the domain $D$. We remark that $\Psi^{-1}\circ F$ is $J$-holomorphic with respect to the almost complex structure 
$J$ on $D\times D$, where $J$ is the pullback of the standard complex structure on $\C^2$ via $\Psi$. 

We shall compute $\partial_{\bar{w}} f$ for the function $f$, where $f$ is considered a function of $w$ (as $w=\xi$).
To this end, we set $z_k=x_k+ \sqrt{-1} y_k$, $k=1,2$, and
$z=s+\sqrt{-1} t$, $w=u+\sqrt{-1}v$. Then with respect to the coordinates $(s,t,u,v)$ and $(x_1,y_1,x_2,y_2)$, the Jacobian of $\Psi$ 
is given by the matrix 
$$
D\Psi =\left (\begin{array}{cc}
A & B\\
0 & I
\end{array} \right ),
$$
where $A=\left (\begin{array}{cc}
\frac{\partial x_1}{\partial s} & \frac{\partial x_1}{\partial t}\\
\frac{\partial y_1}{\partial s} & \frac{\partial y_1}{\partial t}
\end{array} \right )$, $B=\left (\begin{array}{cc}
\frac{\partial x_1}{\partial u} & \frac{\partial x_1}{\partial v}\\
\frac{\partial y_1}{\partial u} & \frac{\partial y_1}{\partial v}
\end{array} \right )$. 
Let $J_0=\left (\begin{array}{cc}
0 & -1\\
1 & 0
\end{array} \right )$ be the matrix representing the standard complex structure. Then the assumptions that 
$\psi(z,w)$ is holomorphic in $w\in D$ and $\psi(0,w)=0$ for all $w\in D$ imply that $J_0 B =B J_0$ and $B=0$ along the
disc $z=0$.

With the preceding understood, we note that the almost complex structure $J$ is given by the matrix
$$
J=D\Psi \left (\begin{array}{cc}
J_0 & 0\\
0 & J_0
\end{array} \right ) (D\Psi)^{-1}= \left (\begin{array}{cc}
A J_0 A^{-1} & (-AJ_0A^{-1} +J_0) B\\
0 & J_0
\end{array} \right ). 
$$
Now the Jacobian of $\Psi^{-1}\circ F$ is $\left (\begin{array}{c}
Df \\
I
\end{array} \right )$ where $Df$ is the Jacobian of $f$. If follows easily that the $J$-holomorphic equation satisfied by 
$\Psi^{-1}\circ F$, i.e., 
$$
J\left (\begin{array}{c}
Df \\
I
\end{array} \right )=\left (\begin{array}{c}
Df \\
I
\end{array} \right ) J_0
$$
is equivalent to the equation $Df+(AJ_0A^{-1}) \cdot Df \cdot J_0=(AJ_0A^{-1}J_0+I)B$. Intrinsically, this can be written as 
$$
\partial_{\bar{w}} f=\frac{1}{2} (AJ_0A^{-1}J_0+I)B. 
$$

With the above understood, we note that since $B=0$ along the disc $z=0$, we have $||B||\leq C_1 |z|$ near $z=0$ for 
some constant $C_1>0$. It follows easily that the function $f$ obeys the inequality $|\partial_{\bar{w}} f|\leq C_2 |f|$ for some
constant $C_2>0$. By the Carleman similarity principle (e.g. see Siebert-Tian \cite{ST}, Lemma 2.9), there is a complex valued
function $g$ of class $C^\alpha$ and a holomorphic function $\phi$, such that $f(w)=\phi(w)g(w)$, where $g(0)\neq 0$.
Note that $\phi$ vanishes at $w=0$ of order $n$ because by the assumption, the holomorphic disc $C$ intersects the $z_2$-axis 
with a tangency of order $n$. After a further change of coordinate, we may assume that $f(w)=w^n g(w)$ for a $C^\alpha$-class
function $g$, where $w\in D$.

Our next goal is to show that for any $c\neq 0$, with $|c|$ sufficiently small, the equation
$$
f(w)=c
$$
has exactly $n$ distinct solutions lying in a small neighborhood of $0\in D$. To see this, we take $h(w)$ to be an $n$-th root
of the function $g(w)$, i.e., $h(w)^n=g(w)$, which is also of $C^\alpha$-class. Let $\lambda_1,\lambda_2,\cdots,\lambda_n$
be the $n$-th roots of $c$. For each $i=1,2,\cdots,n$, we consider the equation 
$$
wh(w)=\lambda_i. 
$$
Set $P(w):=\frac{1}{h(0)} (\lambda_i- w(h(w)-h(0)))$. Then the above equation becomes $w=P(w)$.
With this understood, let $B(r)\subset D$ be the closed disc of radius $r$. Then for $r>0$ sufficiently small, 
$P: B(r)\rightarrow B(r)$ is a well-defined continuous map,  as long as $|\lambda_i|\leq \frac{1}{2} |h(0)|\cdot r$. 
Now we pick any $w_1\in B(r)$ and define inductively $w_{k+1}=P(w_k)$ for $k\geq 1$. Since $B(r)$ is compact, 
the sequence $\{w_k\}$ has a convergent subsequence. The limit $w_0\in B(r)$ satisfies the equation $w_0=P(w_0)$.

It follows easily that when $c\neq 0$ lies in the disc of radius $(\frac{1}{2} |h(0)|\cdot r)^n$, the equation $f(w)=c$ has at least $n$ distinct solutions, all lying in the disc $B(r)$. The local intersection number of the holomorphic disc $C$ with each holomorphic disc $z=c$ equals $n$. This implies that the equation $f(w)=c$ has precisely $n$ distinct solutions in $B(r)$, and the intersection of $C$ with each holomorphic disc $z=c\neq 0$ is transversal. It follows 
easily that the map $\pi_1\circ \Psi^{-1}\circ F: D\rightarrow D$ 
is an $n$-fold branched covering in a neighborhood of $0\in D$, ramified at $0$. 
This finishes the proof.

\end{proof}

With these preparations, we now prove the main theorems. 

\vspace{1.5mm}

\noindent{\bf Proof of Theorem 1.1:}

\vspace{1.5mm}

We first consider the case where $M_G$ is irrational ruled. It is easily seen that there is a subgroup $H$ of prime order $p$ such that $M_H$ is irrational ruled. By Lemma 2.2 and Lemma 2.6(i), the fixed-point set of $H$ consists of only tori of self-intersection zero. Moreover, from the proofs
it is known that $M_H$ is a $\s^2$-bundle over $T^2$, and $M$ is simply a branched cover of $M_H$ along the fixed-point set.

With this understood, we denote by $\{B_i\}$ the image of the fixed-point set of $H$ in $M_H$, which is a disjoint union of
symplectic tori of self-intersection zero. Let $F$ be the fiber class of the $\s^2$-fibration on $M_H$. Then we note that 
$c_1(K_{M_H})=\frac{1-p}{p}\sum_i B_i$ (cf. Proposition 3.2 in \cite{C3}), and $c_1(K_{M_H})\cdot F=-2$. It follows easily
that $p=2$ or $3$, and $(\sum_i B_i)\cdot F=4$ or $3$ accordingly.

To proceed further, we choose an $\omega$-compatible almost complex structure $J$ on $M_H$, where $\omega$ denotes the
symplectic structure on $M_H$, such that $J$ is integrable in a neighborhood of each $B_i$. Note that this is possible because 
$\omega$ admits a standard model near each $B_i$. Now by Gromov's theory, there exists a $\s^2$-bundle structure on $M_H$,
with base $T^2$ and each fiber $J$-holomorphic. We denote by $\pi: M_H\rightarrow T^2$ the corresponding projection onto the
base. Then by Lemma 5.2, the restriction $\pi|_{B_i}: B_i\rightarrow T^2$ is a branched covering where the ramification occurs 
exactly at the non-transversal intersection points of $B_i$ with the fibers. But each $B_i$ is a torus, so that $\pi|_{B_i}$ must be
unramified, or equivalently, $B_i$ intersects each fiber transversely. With this understood, it follows easily that the pre-image of
each fiber of the $\s^2$-bundle in $M$ is a symplectic torus (here we use the fact that
$(\sum_i B_i)\cdot F=4$ or $3$ respectively according to whether $p=2$ or $3$), 
giving $M$ a structure of a $T^2$-bundle over $T^2$ with
symplectic fibers.  This finishes the proof for the case where $M_G$ is irrational ruled.

Next we assume $M_G$ is rational and $G=\Z_2$. By Lemma 2.3 and Lemma 2.6(ii), the fixed-point set $M^G$ consists of $8$ isolated points and a disjoint union of $2$-dimensional 
components $\{Y_i\}$, where 
$\sum_i Y_i^2=2(1-b_2^{-}(M/G))$, and $b_2^{-}(M/G)\in \{0,1,2\}$. We denote by $B_i$ the image of $Y_i$ in $M_G$. Then $B_i^2=2Y_i^2$ for each $i$, and 
$c_1(K_{M_G})=-\frac{1}{2}\sum_i B_i$ (cf. \cite{C3}, Proposition 3.2), so that 
$$
c_1(K_{M_G})^2=\frac{1}{4}\sum_i B_i^2=1-b_2^{-}(M/G). 
$$
It follows easily that $M_G=\C\P^2\# N\overline{\C\P^2}$ where $N=8,9$ or $10$, 
corresponding to $b_2^{-}(M/G)=0$, $1$ or $2$ respectively. Moreover, note that $M_G$ contains $8$ symplectic $(-2)$-spheres 
coming from the resolution of the $8$ isolated singular points of $M/G$. 

By Theorem 1.4, the case where $M_G=\C\P^2\# 8\overline{\C\P^2}$ is immediately ruled out. The case where 
$M_G=\C\P^2\# 10\overline{\C\P^2}$ is ruled out as follows. We consider the double branched cover $Z$ of $M_G$ with
branch loci $\{B_i\}$. Then $Z$ is easily seen a symplectic Calabi-Yau $4$-manifold with $b_1=0$, which is an integral homology $K3$ surface (compare \cite{C}, Theorems 1.1 and 1.2). 
Note that $Z$ contains $16$ embedded $(-2)$-spheres in the complement of the branch set. 
Now observe that in the case of $M_G=\C\P^2\# 10\overline{\C\P^2}$, $\sum_i Y_i^2=-2$, so that there must be one 
$Y_i$ with $Y_i^2<0$. This $Y_i$ gives rise to an embedded $(-2)$-sphere in $Z$, in addition to the 
$16$ embedded $(-2)$-spheres, so that $Z$ contains $17$ disjointly embedded $(-2)$-spheres. 
But this contradicts a theorem of Ruberman in \cite{Ru}, which says that an integral homology
$K3$ surface can contain at most $16$ disjointly embedded $(-2)$-spheres. 
Hence $M_G=\C\P^2\# 10\overline{\C\P^2}$ is ruled out. Finally, we note that the same argument shows that in the case of $M_G=\C\P^2\# 9\overline{\C\P^2}$, the surfaces $B_i$
must be tori of self-intersection zero.

We continue by analyzing the case of $M_G=\C\P^2\# 9\overline{\C\P^2}$ in more detail. First,
we note that there are at most two components in $\{B_i\}$. This is because 
$c_1(K_{M_G})=-\frac{1}{2}\sum_i B_i$, and the $a$-coefficient of each $B_i$ with respect to
a given reduced basis is at least $3$ (cf. Lemma 4.2(2)). Next, we determine the homology classes 
of the $8$ symplectic $(-2)$-spheres $F_1,F_2,\cdots,F_8$ in $M_G$. By Lemma 4.1, we can choose
a symplectic structure on $M_G$ so that the area constraints in Lemma 5.1 are satisfied. 
(Note that this is possible because $-c_1(K_{M_G})\cdot [\omega]=\frac{1}{2}\sum_i \omega(B_i)>0$.) Then the classes of $F_1,F_2,\cdots,F_8$ are given in $3$ cases as listed in Lemma 5.1. We claim that case (a) and case (b) cannot occur. To see this, suppose we are in case (a). It is easy to check, with the area constraints in Lemma 5.1, that the class $E_{i_9}$ has the smallest area among the
$E_i$-classes in the reduced basis. With this understood, we choose an almost complex structure $J$
such that each symplectic $(-2)$-sphere $F_k$ is $J$-holomorphic. Then by Lemma 3.2, the class
$E_{i_9}$ can be represented by a $J$-holomorphic $(-1)$-sphere $C$. Symplectically blow down $M_G$ along $C$, noting that $C$ is disjoint from the $(-2)$-spheres $F_k$ as $C\cdot F_k=0$,
we obtain $8$ disjointly embedded symplectic $(-2)$-spheres in $\C\P^2\# 8\overline{\C\P^2}$,
contradicting Theorem 1.4. Case (b) is similarly eliminated. Consequently, the homology classes
of $F_1,F_2,\cdots,F_8$ are given by case (c) of Lemma 5.1. 

Our next step is to show that there is an embedded symplectic sphere with self-intersection zero, denoted by $F$, which lies in the complement of $F_1,F_2,\cdots,F_8$ and intersects transversely
and positively with $B_i$. This can be seen as follows. It is easy to check that in case (c) of 
Lemma 5.1, the class $E_{l_1}$ has the largest area. By Lemma 3.2, we can choose $\omega$-compatible almost complex structures $J$ so that $B_i$ and $F_k$ are all $J$-holomorphic, and successively represent the classes $E_{l_s}$, $s\geq 2$, beginning with the one of the smallest area, by a $J$-holomorphic
$(-1)$-sphere. By successively symplectically blowing down the classes $E_{l_s}$, $s\geq 2$,
we reach $\C\P^2\# \overline{\C\P^2}$, with $E_{l_1}$ being the $(-1)$-class (see \cite{C}, 
Section 4, for a discussion on the general procedure). Note that the 
$(-2)$-spheres $F_1,F_2,F_3,F_4$ descend to $4$ disjointly embedded symplectic spheres 
of self-intersection zero (they all have class $H-E_{l_1}$); in fact there is a symplectic 
$\s^2$-fibration of $\C\P^2\# \overline{\C\P^2}$ containing them as fibers. With this understood, we can take a fiber $F$ in the complement which intersects transversely and positively with the descendant of $B_i$ in $\C\P^2\# \overline{\C\P^2}$. We then symplectically blow up 
$\C\P^2\# \overline{\C\P^2}$ successively, reversing the symplectic blowing down procedure, 
in order to go back to $M_G$. In this way, we recover the $8$ symplectic $(-2)$-spheres $F_1,F_2,\cdots,F_8$ and the tori $B_i$, although the symplectic structure on $M_G$ may be different since we don't keep track of the sizes of the symplectic blowing up.

Now we symplectically blow down $F_1,F_2,\cdots,F_8$, which results in a symplectic $4$-orbifold
$X$ with $8$ isolated singular points, all of isotropy of order $2$. In the complement of the
singularities, there lies the embedded symplectic sphere $F$ with $F^2=0$, and the tori $B_i$.
By \cite{Gompf}, we can assume that $F$ and $B_i$ intersect symplectically orthogonally without
loss of generality. 

With the preceding understood, we consider the set $\J$ of $\omega$-compatible almost complex structures 
on $X$ which satisfy the following conditions: fix a sufficiently small regular neighborhood $V$ of
$\cup_i B_i$, not containing any singular points of $X$, and fix an integrable 
$\omega$-compatible almost 
complex structure $J_0$ on $V$, then for each $J\in \J$, $J=J_0$ on $V$ and $F$ is
$J$-holomorphic. With this understood, note that for any $J\in \J$, the deformation of the
$J$-holomorphic
sphere $F$ is unobstructed (cf. \cite{HLS}). We denote by $\M_J$ the moduli space of $J$-holomorphic spheres 
having the homology class of $F$. Then $\M_J\neq \emptyset$ and is a smooth $2$-dimensional
manifold. In the present situation, $\M_J$ is not compact, but can be compactified using the 
orbifold version of Gromov compactness theorem (cf. \cite{ChenS1, CR}). The key issue here is to
understand the compactification $\overline{\M_J}$ of $\M_J$, at least for a generic $J\in\J$.

\begin{lemma}
Let $\{S_n\}$ be a sequence in $\M_J$ which converges to a Gromov limit $\sum_i m_i C_i\in
\overline{\M_J}\setminus \M_J$. Then for a generic $J\in\J$, 
$\{C_i\}$ consists of a single component of multiplicity $2$,
which is an embedded orbifold sphere containing exactly $2$ singular points of $X$.
\end{lemma}

\begin{proof}
Since $J$ is generic, there is no $J$-holomorphic $(-\alpha)$-sphere lying in the complement of the singular points of $X$ for any $\alpha>1$.  Moreover, just as in the smooth
case, $\{S_n\}$ can not split off a $J$-holomorphic $(-1)$-sphere lying entirely in the smooth
locus of $X$. It follows easily that in the Gromov limit $\sum_i m_i C_i\in
\overline{\M_J}\setminus \M_J$, each component $C_i$ must contain
a singular point of $X$.

With this understood, we take an arbitrary component $C_i$. Suppose $C_i$ contains $k>0$
singular points of $X$. Then we can pick an orbifold Riemann sphere $\Sigma$ with $k$
orbifold points of order $2$, which are denoted by $z_1,z_2,\cdots,z_k$, and find a $J$-holomorphic
map $f: \Sigma\rightarrow X$ parametrizing $C_i$. Recall that such a map $f$ near an orbifold
point $z_j$, assuming of order $m_j$, is given by a pair $(\hat{f}_j,\rho_j)$, where 
$\hat{f}_j: D\rightarrow \C^2$ is a local
lifting of $f$ near $z_j$ to the uniformizing system at $f(z_j)\in X$, and 
$\rho_j:\Z_{m_j}\rightarrow G_{f(z_j)}$ is an injective homomorphism to the isotropy group 
$G_{f(z_j)}$ at $f(z_j)\in X$, with respect to which $\hat{f}_j$ is equivariant. 
With this understood, we let $g\in\Z_{m_j}$ 
be the generator acting on $D$ by a rotation of angle $2\pi/m_j$, and let $(m_{j,1},m_{j,2})$, 
$0\leq m_{j,1},m_{j,2}<m_j$, be the weights of the action of $\rho_j(g)\in G_{f(z_j)}$
on $\C^2$. Then the dimension of the moduli space of $J$-holomorphic curves containing $C_i$ equals $2d$, where $d\in\Z$ and is given by
$$
d=c_1(TX)\cdot C_i +2 -\sum_{j=1}^k \frac{m_{j,1}+m_{j,2}}{m_j}-(3-k).
$$
See \cite{CR, ChenS1}. Note that in the present situation, $m_j=2$ and $m_{j,1}=m_{j_2}=1$
for each $j$. It follows easily that $d=c_1(TX)\cdot C_i-1$; in particular, $c_1(TX)\cdot C_i\in\Z$.
Moreover, since $J$ is generic, we have $d\geq 0$, which implies that $c_1(TX)\cdot C_i\geq 1$.

As an immediate corollary, we note that $\{C_i\}$ either consists of two components, each with multiplicity $1$, or a single component with multiplicity $2$, and moreover, $c_1(TX)\cdot C_i=1$
for each $i$. This is because $c_1(TX)\cdot F=2$,
and $F=\sum_i m_i C_i$. We can further rule out the possibility of two components as follows.
Suppose there are two components $C_1,C_2$ in $\{C_i\}$. Then 
$C_1^2+2C_1\cdot C_2+C_2^2=F^2=0$ implies that one of $C_1^2,C_2^2$ must be negative.
Without loss of generality, assume $C_1^2<0$. Then $C_2^2\geq 0$ because $b_2^{-}(X)=1$.
With this understood, we note that $C_1\cdot C_2\geq \frac{1}{2}$ by the orbifold intersection
formula in \cite{C0} (see also \cite{ChenS1}). This implies $C_1^2\leq -1$. Now we apply 
the orbifold adjunction inequality (cf. \cite{C0,ChenS1}) to $C_1$, which gives 
$$
C_1^2-c_1(TX)\cdot C_1+2\geq k(1-\frac{1}{2}).
$$
With $C_1^2\leq -1$ and $c_1(TX)\cdot C_1=1$, it follows that $k=0$, which is a contradiction. 
Hence the claim that there is only one component in $\{C_i\}$. 

Let $C$ denote the single component which has multiplicity $2$, and let $f: \Sigma\rightarrow X$
be a $J$-holomorphic parametrization of $C$. Then we note that $C^2=0$ and $c_1(TX)\cdot C=1$. Applying the orbifold adjunction formula to $C$ (cf. \cite{C0,ChenS1}), we get
$$
C^2-c_1(TX)\cdot C+2=k(1-\frac{1}{2})+\sum k_{[z,z^\prime]} +\sum k_z,
$$
where $k_{[z,z^\prime]},k_z\in \Q$ are nonnegative and have the following significance. For any $z,z^\prime\in \Sigma$, where $z\neq z^\prime$, such that $f(z)=f(z^\prime)$, the number 
$k_{[z,z^\prime]}>0$. Moreover, if $f(z)=f(z^\prime)$ is a smooth point of $X$, then 
$k_{[z,z^\prime]}\in\Z$. Likewise, for any $z\in\Sigma$, if $f$ is not a local orbifold embedding near
$z$, then $k_z>0$. Moreover, if $f(z)$ is a smooth point of $X$, then $k_z\in\Z$. With this
understood, it follows easily that $k\leq 2$, and if $k=2$, then all
$k_{[z,z^\prime]},k_z=0$, which means that $C$ is an embedded $2$-dimensional suborbifold.
To rule out the possibility that $k=1$, we first observe that in this case, $k_{[z,z^\prime]}\in\Z$.
This is because as $k=1$, we can not have a pair of points $z,z^\prime\in\Sigma$, where 
$z\neq z^\prime$, such that $f(z)=f(z^\prime)$ is a singular point of $X$. It follows easily 
that all $k_{[z,z^\prime]}$ must be zero, and $k_z=\frac{1}{2}$ at the unique singular
point $f(z)$ on $C$. The number $k_z$ is the local self-intersection number of $C$ at 
the singular point, and $k_z=\frac{1}{2}$ means that in the uniformizing system near the 
singular point, $C$ is given by a $J$-holomorphic (singular) disc with a local self-intersection 
$1$ at the origin. It follows that the singularity at the origin must be a cusp singularity and the
$J$-holomorphic disc is parametrized by a pair of functions $z_1=t^2, z_2=t^3+\cdots$, where
$t\in D$. However, it is clear that such defined $J$-holomorphic disc is not invariant under the 
$\Z_2$-action $(z_1,z_2)\mapsto (-z_1,-z_2)$, which is a contradiction. Hence $k=1$ is ruled out. 
This finishes the proof of the lemma. 

\end{proof}

It follows easily that the compactified moduli space $\overline{\M_J}$ gives rise to a $J$-holomorphic
$\s^2$-fibration on $X$, which contains $4$ multiple fibers, each with multiplicity $2$. We denote
by $\pi: X\rightarrow B$ the $\s^2$-fibration. It is easy to see that the base 
$B$ is an orbifold sphere, with
$4$ orbifold points of order $2$. Furthermore, note that for each $i$, 
$\pi |_{B_i}: B_i\rightarrow B$ is a branched covering in the complement of the multiple fibers
by Lemma 5.2.

To proceed further, we note that $c_1(K_X)=-\frac{1}{2}\sum_i B_i$, so that $(\sum_i B_i)\cdot F=4$.
Let $z_1,z_2,z_3,z_4$ be the orbifold points of $B$, and let $w_1,\cdots,w_k\in B$ be the points 
parametrizing those regular fibers which do not intersect transversely with $\cup_i B_i$.
We denote by $x_l$ the number of intersection points of $\cup_i B_i$ with the multiple fiber
at $z_l$, $l=1,2,3,4$, and denote by $y_j$ the number of intersection points of $\cup_i B_i$ with 
the regular fiber at $w_j$, where $j=1,2,\cdots,k$. Then note that $x_l\leq 2$ and $y_j<4$ for
each $l,j$. On the other hand, we observe the following relation in Euler numbers:
$$
\sum_i\chi(B_i)-\sum_{j=1}^k y_j-\sum_{l=1}^4 x_l=4(\chi(|B|)-k-4),
$$
where $|B|=\s^2$ is the underlying space of $B$. 
With $x_l\leq 2$ and $y_j<4$, it follows easily that $k$ must be zero, and $x_l=2$ for each $l$.
This means that $\cup_i B_i$ intersects each regular fiber transversely at $4$ points and 
intersects each multiple fiber at $2$ points. 

Finally, we observe that $X=|M/G|$, i.e., $X$ is the
symplectic $4$-orbifold obtained by de-singularizing $M/G$ along the $2$-dimensional 
singular components. With this understood, it is easy to see that under the
projection $M\rightarrow X=|M/G|$, the pre-image of each regular fiber in the
$\s^2$-fibration on $X$ is a symplectic $T^2$ in $M$, giving rise to a $T^2$-fibration over $B$
on $M$ (here we use the fact that $\cup_i B_i$ intersects each regular fiber transversely at 
$4$ points and the projection $M\rightarrow X$ is a double cover branched over $\cup_i B_i$).
Moreover, the pre-image of each multiple $\s^2$-fiber is a multiple $T^2$-fiber
of multiplicity $2$ in the $T^2$-fibration on $M$. It is known that such a $4$-manifold $M$
is diffeomorphic to a hyperelliptic surface or a secondary Kodaira surface, see \cite{FM}.
Since $b_1(M)\neq 1$, $M$ must be diffeomorphic to a hyperelliptic surface.
This finishes the proof of Theorem 1.1.

\vspace{1.5mm}

\noindent{\bf Proof of Theorems 1.2 and 1.3:}

\vspace{1.5mm}

Suppose $G$ is of prime order $p$. The case where $M_G$ has torsion canonical class is contained in Lemmas 2.1 and 2.8(2), and the case where $M_G$ is irrational ruled is in Lemmas 2.2 and 2.6(i), with $p=2$ or $3$ from the proof of Theorem 1.1.

Suppose $M_G$ is rational. Then by Lemmas 2.3, 2.4, 2.6 and 2.8, the order $p=2,3$ or $5$. 
Concerning the fixed-point set structure, the case of $G=\Z_2$ follows readily from the proof of Theorem 1.1. For $G=\Z_3$, the fixed-point set structure for
the isolated points is determined in Lemmas 2.4 and 2.9. Regarding the $2$-dimensional fixed components, we explore the embedding $D\rightarrow M_G$. 
In order to determine $M_G$ in each
case, we use the formula in Proposition 3.2 of \cite{C3} to determine $c_1(K_{M_G})$, based on the
singular set structure of the quotient orbifold $M/G$, then we compute $c_1(K_{M_G})^2$. 
This allows us
to determine the diffeomorphism type of $M_G$ as $M_G$ is a rational $4$-manifold. 
In the case of $b_1(M)=2$, it is easy to see that $M_G=\C\P^2\# 10 \overline{\C\P^2}$.
If the set of $2$-dimensional fixed components is nonempty, Proposition 4.3 implies that it must consist of a single torus.
In the case of $b_1(M)=4$, $M_G=\C\P^2\# 12 \overline{\C\P^2}$, and Proposition 4.4 implies that there are no 
$2$-dimensional fixed components. 
For $G=\Z_5$ where $b_1(M)=4$, 
the fixed-point set structure for the isolated points is determined in Lemma 2.10.
The possible $2$-dimensional fixed components are excluded by Proposition 4.5. 

For the case where $G$ is of non-prime order, the order of $G$ and the fixed-point set structure are determined in Lemmas
2.11 and 2.12. This completes the discussion on Theorems 1.2 and 1.3.

\vspace{1.5mm}

\noindent{\bf Proof of Theorem 1.4:}

First, consider the case of $N=8$. We begin by showing that one can choose a symplectic structure
on $X$ such that the area constraints in Lemma 5.1 are fulfilled. To see this, by Lemma 4.1 we can 
choose symplectic structures $\omega$ on $X$ such that one of the $8$ symplectic $(-2)$-spheres has area $\delta_1$ and the remaining $7$ symplectic $(-2)$-spheres have area $\delta_2$, where
$\delta_2<\delta_1<2\delta_2$, and $\delta_1,\delta_2$ can be arbitrarily small. It remains to
show that one can arrange so that $7\delta_i< -c_1(K_X)\cdot [\omega]$, $i=1,2$, hold true. 
For this,
we recall the fact that for $X=\C\P^2\# N \overline{\C\P^2}$, where $N\leq 8$, 
$-c_1(K_X)$ can be represented by pseudo-holomorphic curves, and moreover,
one can require the pseudo-holomorphic curves to pass through any given point in $X$,
see Taubes \cite{T}. We pick a point $x_0\in X$ in the complement of the $8$ symplectic $(-2)$-spheres and require the pseudo-holomorphic curves representing $-c_1(K_X)$ to pass through
$x_0$. Then it is easy to see that no matter how small we choose the areas $\delta_1,\delta_2$,
$-c_1(K_X)\cdot [\omega]>\delta_0$ for some $\delta_0$ independent of the choice of 
$\delta_1,\delta_2$. It follows that we can arrange so that $7\delta_i< -c_1(K_X)\cdot [\omega]$, $i=1,2$, hold true. 

With the preceding understood, by the same argument as in Lemma 5.1, 
we can show that the homology classes of the $8$ symplectic $(-2)$-spheres must be
given as in case (a) of Lemma 5.1. Then by Lemma 3.2, we can successively symplectically blow
down the $E_{i_s}$ classes for $s\geq 2$ and reach to the $4$-manifold 
$\C\P^2\# \overline{\C\P^2}$, with $E_{i_1}$ being the $(-1)$-class, such that the $7$ symplectic 
$(-2)$-spheres $F_2,F_3,\cdots,F_8$ descend to a configuration of symplectic spheres 
of the class $H$,
which intersect transversely and positively according to the incidence relation of the Fano plane;
that is, the $7$ spheres intersect in $7$ points, where each point is contained in $3$ spheres. 
By a theorem of Ruberman and Starkston (cf. \cite{RuS}), such a configuration cannot exist in 
$\C\P^2$. Thus to derive a contradiction, we need to represent the class $E_{i_1}$ by a 
symplectic $(-1)$-sphere in the complement of the $7$ symplectic spheres,
to further blow down $\C\P^2\# \overline{\C\P^2}$.

To this end, we note that the configuration of $7$ symplectic spheres in $\C\P^2\# \overline{\C\P^2}$
is $J$-holomorphic with respect to some compatible almost complex structure $J$. On the other
hand, the class $E_{i_1}$ is represented by a finite set of $J$-holomorphic curves 
$\sum_i m_i C_i$ by Taubes' theorem (cf. \cite{LiLiu0}). Now the key observation is that if there are 
more than one components in $\{C_i\}$, then one of them must have a negative 
$a$-coefficient in the reduced basis $H, E_{i_1}$. But such a component intersects negatively with 
any of the $7$ $J$-holomorphic spheres in the configuration (which has class $H$). This is a contradiction, hence $E_{i_1}$ must be represented by a single $J$-holomorphic curve, 
which is a $(-1)$-sphere and lies in the complement of the configuration of $7$ symplectic spheres. 
This finishes the proof for the case of $N=8$. 

The argument for the case of $N=7$ is similar. For $N=9$, it is easy to see from Lemma 5.1 
that the homology class for the $9$-th symplectic $(-2)$-sphere does not exist. This completes the proof of Theorem 1.4.

\vspace{1.5mm}

{\bf Acknowledgement:} We thank R. Inanc Baykur, Tian-Jun Li, and Weiwei Wu for useful communications. We are also grateful to an anonymous referee whose comments help improve the
presentation of this paper.

\vspace{2mm}

{\Small University of Massachusetts, Amherst.\\
{\it E-mail:} wchen@math.umass.edu


\begin{thebibliography}{}

\bibitem{Bauer} S. Bauer, {\em Almost complex $4$-manifolds with vanishing first Chern class}, J. Diff. Geom. {\bf 79} (2008),
25-32.

\bibitem{Bay} R. I. Baykur, private communications. 

\bibitem{BP} O. Buse and M. Pinsonnault. {\em Packing numbers of rational ruled four-manifolds}, 
Journal of Symplectic Geometry {\bf 11} (2013), 269-316.

\bibitem{CNP} C. Caubel, A. N\'{e}methi and P. Popescu-Pampu, {\em Milnor open books and 
Milnor fillable contact $3$-manifolds}, Topology {\bf 45} (2006), 673-689.

\bibitem{C0} W. Chen, {\em Orbifold adjunction formula and symplectic cobordisms 
                  between lens spaces}, Geometry and Topology  {\bf 8} 
                  (2004), 701-734.

\bibitem{ChenS1} W. Chen, {\em Pseudoholomorphic curves in four-orbifolds and some
                   applications}, in Geometry and Topology of
                   Manifolds, Boden, H.U. et al ed., Fields Institute
                             Communications {\bf 47}, 11-37.
                             Amer. Math. Soc., Providence, RI, 2005.
                                                                  
\bibitem{C3} W. Chen, {\em Resolving symplectic orbifolds with applications to finite group actions},
Journal of G\"{o}kova Geometry Topology, {\bf 12} (2018), 1-39.

\bibitem{C} W. Chen, {\em On a class of symplectic $4$-orbifolds with vanishing canonical class},
arXiv:2005.04688v3 [math.GT] 9 Nov 2020.


\bibitem{CK1} W. Chen and S. Kwasik, {\em Symplectic symmetries of $4$-manifolds}, 
Topology {\bf 46} no. 2 (2007), 103-128. 


\bibitem{CK2} W. Chen and S. Kwasik, {\em Symmetries and exotic smooth structures on
a $K3$ surface}, Journal of Topology {\bf 1} (2008), 923-962.

\bibitem{CR} W. Chen and Y. Ruan, {\em Orbifold Gromov-Witten theory},
                                  in Orbifolds
                                  in Mathematics
                                  and Physics, Adem, A. et al ed.,
                                  Contemporary Mathematics
                                  {\bf 310}, pp. 25-85. AMS.
                                  Providence, RI, 2002.


\bibitem{Don} S. Donaldson, {\em Some problems in differential geometry and topology}, 
Nonlinearity {\bf 21} (2008), no. 9, 157-164.     

\bibitem{DLW} J.G. Dorfmeister, T.-J. Li, and W. Wu, {\em Stability and existence of surfaces in symplectic $4$-manifolds with $b^{+}=1$}, J. Reine Angew. Math. (Crelle's Journal), 
2018 (742), 115-155.
                       
\bibitem{FP} J. Fine and D. Panov, {\em The diversity of symplectic Calabi-Yau $6$-manifolds},
Journal of Topology {\bf 6} no. 3 (2013), 644-658.       

\bibitem{FY} J. Fine and C. Yao, {\em A report on the hypersymplectic flow}, arXiv:2001.11755v1 [math.DG].                   

\bibitem{FV} S. Friedl and S. Vidussi, {\em On the topology of symplectic Calabi-Yau $4$-manifolds}, Journal of Topology 
{\bf 6} (2013), 945-954.

\bibitem {FM} R. Friedman and J. Morgan, {\em Smooth Four-Manifolds and 
                            Complex Surfaces}, Ergebnisse der Math. Vol. 27,
                           Springer-Verlag, 1994.
                           
\bibitem{Fu} A. Fujiki, {\em Finite automorphism groups of complex tori of dimension two},
Publ. RIMS. Kyoto Univ. {\bf 24} (1988), 1-97. 

\bibitem {FF} Y. Fukumoto and M. Furuta, {\em Homology $3$-spheres bounding
                         acyclic $4$-manifolds}, Math. Res. Lett. {\bf 7}
                         (2000), no. 5-6, 757-766.
                         
\bibitem{GS} D. Gay and A. Stipsicz, {\em Symplectic surgeries and normal surface 
singularities}, Algebr. Geom. Topol. {\bf 9} (2009), no.4, 2203-2223.      

\bibitem{G} H. Geiges, {\em Symplectic structures on $T^2$-bundles over $T^2$}, Duke Math. J.
{\bf 67} (1992) no.3, 539-555.                      
                           
\bibitem{Gompf} R. Gompf, {\em A new construction of symplectic manifolds}, Ann. of Math. 
{\bf 143}, no.3 (1995), 527-595. 

\bibitem{G0} M. Gromov, {\em Pseudoholomorphic curves in symplectic manifolds}, Invent. Math. {\bf 82} no.2 (1985), 307-347. 

\bibitem{HLS} H. Hofer, V. Lizan, and J.-C. Sikorav, {\em On genericity for holomorphic 
curves in $4$-dimensional almost complex manifolds}, J. Geom. Anal. {\bf 7} (1998), 149-159. 
  
\bibitem{KK} Y. Karshon and L. Kessler, {\em Distinguishing symplectic blowups of the complex projective plane}, 
Journal of Symplectic Geometry {\bf 15} (2017), no.4, 1089-1128.

\bibitem{LL} B.-H. Li and T.-J. Li, {\em Symplectic genus, minimal genus and diffeomorphisms}, Asian J. Math. {\bf 6} (2002), no. 1, 123-144.

\bibitem{Li} T.-J. Li, {\em Symplectic 4-manifolds with Kodaira dimension zero}, 
J. Diff. Geom. {\bf 74} (2006) 321-352.

\bibitem{Li1}T.-J. Li, {\em Quaternionic bundles and Betti numbers of symplectic $4$-manifolds with
Kodaira dimension zero}, Int. Math. Res. Not., 2006, Art. ID 37385, 28.

\bibitem{LiLiu0} T.-J. Li and A.-K. Liu, {\em Symplectic structures on ruled surfaces and a generalized adjunction inequality}, Math. Res. Letters {\bf 2} (1995), 453-471. 

\bibitem{LiLiu} T.-J. Li and A.-K. Liu, {\em Uniqueness of symplectic canonical class, surface cone and symplectic cone of $4$-manifolds with $b^{+}=1$}, J. Diff. Geom. {\bf 58} no. 2 (2001), 331-370. 

\bibitem{LW}T.-J. Li  and W. Wu, {\em Lagrangian spheres, symplectic surfaces and the symplectic mapping class group}, Geometry and Topology {\bf 16} no. 2 (2012), 1121-1169. 

\bibitem{MR} L. Mart\'{i}n-Merch\'{a}n and J. Rojo, {\em Resolution of $4$-dimensional symplectic orbifolds}, arXiv: 2003.08191v1 [math.SG] 18 Mar 2020. 

\bibitem{McD} D. McDuff, {\em The structure of rational and ruled symplectic $4$-manifolds}, J. Amer. Math. Soc. {\bf 3}(1990),
679-712; Erratum: J. Amer. Math. Soc. {\bf 5}(1992), 987-988.
                       
\bibitem{MSz} J. Morgan and Z. Szabo, {\em Homotopy $K3$ surfaces and Mod $2$ Seiberg-Witten invariants}, Math. Res. Lett. {\bf 4 }(1997),17-21.

\bibitem{N} N. Nakamura, {\em Mod $p$ vanishing theorem of Seiberg-Witten invariants for $4$-manifolds with $\Z_p$-actions},
Asian J. Math. {\bf 10} no. 4  (2006), 731-748.

\bibitem{PS} H. Park and A.I. Stipsicz, {\em Smoothings of singularities and symplectic surgery},
J. Symplectic Geometry {\bf 12} (2014), 585-597.

\bibitem{Ru} D. Ruberman, {\em Configurations of $2$-spheres in the $K3$ surface and other $4$-manifolds}, Math. Proc. Camb. Phil. Soc. {\bf 120} (1996), 247-253.

\bibitem{RuS} D. Ruberman and L. Starkston, {\em Topological realizations of line arrangements}, International Mathematics 
Research Notices, 2019, no.8, 2295-2331.

\bibitem{Sr} D. Ruberman and S. Strle, {\em Mod $2$ Seiberg-Witten invariants 
of homology tori}, Math. Res. Lett. {\bf 7} (2000), 789-799.


\bibitem{SF} K. Sakamoto and S. Fukuhara, {\em Classification of $T^2$-bundles over $T^2$},
Tokyo J. Math. {\bf 6} (1983), no.2, 311-327.

\bibitem{ST} B. Siebert and G. Tian, {\em Lectures on pseudo-holomorphic curves and the symplectic isotopy problem}, 
Symplectic $4$-manifolds and algebraic surfaces, pp. 269-341, Lecture Notes in Math., {\bf 1938}, Springer, Berlin, 2008.

\bibitem{T} C.H. Taubes, {\em Seiberg-Witten and Gromov Invariants for Symplectic $4$-manifolds}, Proceedings of
                           the First IP Lectures Series, Volume II, R. Wentworth ed., International Press, 2000.
                                                                                                   

\end{thebibliography}
\end{document}